\documentclass[12pt]{amsart}
\usepackage{amsmath}
\usepackage{amssymb}
\usepackage{amsfonts}
\usepackage{amsthm}
\usepackage{mathrsfs}
\usepackage{comment}
\usepackage[normalem]{ulem}
\usepackage{enumerate}
\usepackage{geometry}
\usepackage{mathtools}
\mathtoolsset{showonlyrefs}
\usepackage{hyperref}
\hypersetup{colorlinks=true,
linktoc=all,linkcolor=black,
citecolor=black}
\usepackage{tikz}
\usepackage{bbm}
\usepackage{mathptmx}

\newcommand{\C}{\mathbb{C}}
\newcommand{\D}{\mathbb{D}}
\newcommand{\T}{\mathbb{T}}
\newcommand{\Te}{\mathbb{T}}

\newcommand{\Z}{\mathbb{Z}}
\newcommand{\R}{\mathbb{R}}

\newcommand{\calV}{\mathscr{V}}
\newcommand{\calD}{\mathscr{D}}

\newcommand{\calL}{\mathscr{L}}
\newcommand{\calQ}{\mathscr{Q}}

\newcommand{\diff}{\mathrm{d}}

\newcommand{\Tope}{\mathbf{T}}

\newcommand{\Iop}{\mathbf{I}}

\newcommand{\Jop}{\mathbf{J}}
\newcommand{\Strans}{\mathrm{S}}
\newcommand{\Ttrans}{\mathrm{T}}
\newcommand{\Rtrans}{\mathrm{R}}

\newcommand{\e}{\mathrm{e}}
\newcommand{\Ordo}{\mathrm{O}}
\newcommand{\ordo}{\mathrm{o}}
\newcommand{\imag}{\mathrm{i}}
\renewcommand{\Re}{\mathrm{Re}\,}
\newcommand{\Mfun}{\mathrm{M}}

\renewcommand{\hat}{\widehat}

\newcommand{\re}{\mathrm{Re}}
\newcommand{\im}{\mathrm{Im}}
\newcommand{\Hyp}{\mathbb{H}}
\newcommand{\ybeta}{y}

\DeclareFontFamily{U}{rcjhbltx}{}
\DeclareFontShape{U}{rcjhbltx}{m}{n}{<->rcjhbltx}{}
\DeclareSymbolFont{hebrewletters}{U}{rcjhbltx}{m}{n}
\DeclareMathSymbol{\shin}{\mathord}{hebrewletters}{152}

\input pdf-trans
\newbox\qbox
\def\usecolor#1{\csname\string\color@#1\endcsname\space}
\newcommand\bordercolor[1]{\colsplit{1}{#1}}
\newcommand\fillcolor[1]{\colsplit{0}{#1}}
\newcommand\outline[1]{\leavevmode%
  \def\maltext{#1}%
  \setbox\qbox=\hbox{\maltext}%
  \boxgs{Q q 2 Tr \thickness\space w \fillcol\space \bordercol\space}{}%
  \copy\qbox%
}
\makeatother
\newcommand\colsplit[2]{\colorlet{tmpcolor}{#2}\edef\tmp{\usecolor{tmpcolor}}%
  \def\tmpB{}\expandafter\colsplithelp\tmp\relax%
  \ifnum0=#1\relax\edef\fillcol{\tmpB}\else\edef\bordercol{\tmpC}\fi}
\def\colsplithelp#1#2 #3\relax{%
  \edef\tmpB{\tmpB#1#2 }%
  \ifnum `#1>`9\relax\def\tmpC{#3}\else\colsplithelp#3\relax\fi
}
\bordercolor{black}
\fillcolor{white}
\def\thickness{.3}

\newtheorem{thm}{Theorem}[subsection]
\newtheorem{cor}[thm]{Corollary}
\newtheorem{lem}[thm]{Lemma}

\newtheorem{prop}[thm]{Proposition}

\theoremstyle{definition}
\newtheorem{defn}[thm]{Definition}

\newtheorem{example}[thm]{Example}
\theoremstyle{remark}
\newtheorem{rem}[thm]{Remark}

\usepackage{graphicx}

\numberwithin{equation}{subsection}
\numberwithin{figure}{subsection}

\title[Hyperbolic Fourier series and the Klein-Gordon equation]
{Hyperbolic Fourier series and the Klein-Gordon equation}

\begin{document}


\author[H. Hedenmalm]{H. Hedenmalm$^\ast$}

\address{Department of Mathematics
\\
The Royal Institute of Technology
\\
SE -- 100 44 Stockholm
\\
Sweden
\\ 
$\&$ Department  of Mathematics and Computer Sciences
\\
St. Petersburg State University
\\
St. Petersburg
\\
Russia
\\
$\&$ Department of Mathematics and Statistics
\\
University of Reading
\\
Reading
\\
U.K.}
\email{haakan00@gmail.com}

\author[A. Montes-R.]{A. Montes-Rodr\'{\i}guez} 
 \address{Department of Mathematics\\
 University of Sevilla\\
 ES--4180 Sevilla,
Spain}
 \email{amontes@us.es}
 
 

\let\thefootnote\relax\footnotetext{$^\ast$Corresponding author}



\maketitle

\date{\today}


\begin{abstract} 
In an effort to extend classical Fourier theory,
Hedenmalm and Montes-Rodr\'{\i}guez (2011) found that the function system
\[
e_m(x)=\e^{\imag\pi mx},\quad 
e_n^\dagger
(x)=e_n(-1/x)=\e^{-\imag\pi n/x}
\]
is weak-star complete  in $L^{\infty}(\R)$ when $m,n$ range over the integers
with $n\ne0$.
It turns out that the system can be used to provide unique representation
of functions and more generally distributions on the real line $\R$.
For instance, we may represent uniquely the unit point mass at a point
$x\in\R$:
\[
\delta_x(t)=A_0(x)+\sum_{n\ne0}\big(A_n(x)\,\e^{\imag\pi nt}
+B_n(x)\,\e^{-\imag\pi n/t}\big),
\]
with at most polynomial growth of the coefficients, so that the sum
converges in the sense of distribution theory.
In a natural sense, the system $\{A_n,B_n\}_n$ is biorthogonal to the
initial system $\{e_n,e_n^\dagger\}_n$ on the real line.
More generally, for a distribution $f$ on the compactified real line,
we may decompose it in a \emph{hyperbolic Fourier series}
\[
f(t)=a_0(f)+\sum_{n\ne0}\big(a_n(f)\,\e^{\imag\pi nt}+b_n(f)\,\e^{-\imag\pi n/t}\big),
\]
understood to converge in the sense of distribution theory.
Such hyperbolic Fourier series arise from two different considerations. One is
the Fourier interpolation problem of recovering a radial function $\phi$ 
on $\R^d$ from partial information on $\phi$ and its Fourier transform
$\hat \phi$, studied by Radchenko and Viazovska (2019).  
Another consideration is the interpolation theory of the Klein-Gordon
equation $\partial_x\partial_y u+u=0$.
For instance, the biorthogonal system $\{A_n,B_n\}_n$ leads to a collection
of solutions that vanish along the
lattice-cross of points $(\pi k,0)$ and $(0,\pi l)$ save for one of
these points.  
These interpolating solutions allow for 
restoration of a given solution $u$ from its values on the lattice-cross.
\end{abstract}

\maketitle


\section{Overview}
We develop the theory of hyperbolic Fourier series on the real line in 
Section \ref{int}.
We first do this in the sense of distribution theory, where it becomes
natural to consider harmonic extensions to the upper half-plane.
We consider more general harmonic functions in the upper half-plane,
and expand them in harmonic hyperbolic Fourier series, which can be 
viewed as expansion of the boundary "trace" ultradistributions. We explore
the relationship between growth and uniqueness of the coefficients in 
the expansion. The results are typically stated in Section \ref{int} and
the proofs are deferred to  Section \ref{sec:uniq-beyond} for the uniqueness
aspects, and to Section \ref{sec:Schwarz} for the existence. The key to the
proof of the existence is an application of the Schwarz reflection principle
for harmonic functions along a circular boundary. This gives rise to
an algorithm which successively expands the solution to the entire 
upper half-plane, with control of the growth of the solution (see Section
\ref{sec:Schwarz}). This permits us to obtain sharper coefficient bounds 
than what was available with the earlier methods based on integral kernels
and Eichler cohomology.
The expansion of a point mass on the line is particularly
curious, as it gives rise to a biorthogonal system of functions. Here, the
main results are stated in Section \ref{int} and the proofs are deferred
to  Section \ref{sec:Diracmass}. 
The hyperbolic Fourier series are intimately connected 
with solutions to the Klein-Gordon equation, and that aspect is studied
in Section \ref{sec:Klein-Gordon}, with proofs supplied in Sections
\ref{sec:KG-uniq} and \ref{sec:Diracmass}. We also introduce two
kinds of skewed hyperbolic Fourier series: power skewed and 
exponentially skewed. The power skewed hyperbolic Fourier series
are connected with recent advances in Fourier interpolation of radial
functions, whereas exponential skewing comes from considering specific 
perturbations the lattice-cross for the interpolation problem for solutions 
to the Klein-Gordon equation.
As for the details, see Section \ref{sec:skewed}. We throw light on
the connection with Fourier interpolation for radial functions in Section
\ref{sec:radial_FI}. 

A preliminary version of the present work was announced in 2021, see 
\cite{BHMRV1}.

\section{Some notation}
We write $\Z$ for the integers, $\R$ for the real line, 
$\C$ for the complex plane, $\Hyp$
for the upper half-plane, and we write, e.g., $\R_{>0}$ for the 
positive reals. We sometimes separate infinities, $-\infty$ and 
$+\infty$  should be thought of as different, while $\infty$ unites them in the
complex plane. We write $+\imag\infty$ in place of $\imag(+\infty)$.

\section{Hyperbolic Fourier series}
\label{int}

\subsection{Hyperbolic Fourier series: general considerations} 
If $f$ is a function on the real line $\R$ which is $2$-periodic,
i.e. $f(t+2)=f(t)$ holds for all $t\in\R$, we may attempt to represent 
it in the form of a Fourier series
\[
f(t)\sim \sum_{n\in\Z}a_n\,\e^{\imag\pi nt},
\]
where we write "$\sim$" in place of "$=$" to indicate that the 
representation need not be valid pointwise, but is nevertheless 
correct in a generalized sense. One such reasonable sense is that
of distribution theory, which would ask that
\begin{equation}
f(\varphi)=\langle f,\varphi\rangle_\R=\sum_{n\in\Z}a_n \langle 
\varphi,e_n \rangle_\R,
\label{eq:FSDS1}
\end{equation}
where we write $e_n(t):=\e^{\imag\pi nt}$, provided that $\varphi$ is a
\emph{test function}, which typically could be assumed to be
$C^\infty$-smooth and compactly supported. The notation $f(\varphi)$
may be confusing given that we use the function notation $f(t)$ at the
same time, so we will prefer to speak of $\langle f,\varphi\rangle_\R$
instead. Here, we should specify that we think of the indicated pairing 
as an extension of the bilinear form 
\[
\langle f,g\rangle_\R=\int_\R f(t)g(t)\diff t,
\] 
which is well-defined, e.g., for $f,g$ measurable with $fg\in L^1(\R)$. 
For smooth test functions $\varphi$, the decay 
$\langle \varphi,e_n\rangle_\R=\Ordo(|n|^{-k})$ 
as $|n|\to+\infty$ holds for any fixed positive integer $k$, so that 
convergence of the right-hand side of \eqref{eq:FSDS1} holds provided 
that $|a_n|=\Ordo(|n|^k)$ holds for some fixed $k$ as $|n|\to+\infty$. 
Please note that by considering a smaller collection of test functions 
$\varphi$, we can get faster decay of $\langle e_n,\varphi\rangle_\R$
as $|n|\to+\infty$, which in turn will allow us to make sense of the 
convergence in \eqref{eq:FSDS1} also when the coefficients $a_n$ grow
more rapidly as $|n|\to+\infty$.  
We would however like to represent more general functions or 
distributions on  the line $\R$ which need \emph{not be periodic}. 
Traditionally, this is done by considering Fourier integrals as the limit
of Fourier series with the period tending to infinity. Here, we follow 
the hunch in \cite{hed}, which suggests that we should add an additional
system of a similar form. As for notation, we let 
$\Z_{\ne0}:=\Z\setminus\{0\}$ denote the
nonzero integers.

\begin{defn}
A series of the form
\[
a_0+\sum_{n\in\Z_{\ne0}}\big(a_n\,\e^{\imag\pi nt}+b_n\,\e^{-\imag\pi n/t}\big)
\]
is called a \emph{hyperbolic Fourier series}. If we agree that $b_0=0$,
we may also write it in the form
 \[
\sum_{n\in\Z}\big(a_n\,\e^{\imag\pi nt}+b_n\,\e^{-\imag\pi n/t}\big).
\]
\end{defn}

\begin{rem}
This defines formal series, no convergence requirement is made whatsoever.
It is clear how to add two such series, but it is generally unclear how to
multiply them. The difficulty is that the product
\[
\e^{\imag\pi nt}\e^{-\imag\pi m/t}=\e^{\imag\pi nt-\imag\pi m/t}
\]
is not of the required form unless $n=0$ or $m=0$.
\end{rem}

We want to use hyperbolic Fourier series to represent functions or more 
generally distributions (or even ultradistributions) on the real line.  If we write 
\[
f_1(t)=a_0+\sum_{n\in\Z_{\ne0}}a_n\,\e^{\imag\pi nt},\quad
f_2(t)=\sum_{n\in\Z_{\ne0}}b_n\,\e^{\imag\pi nt},
\]
which represent $2$-periodic distributions on the line if the coefficient
growth is at most polynomial, then the hyperbolic Fourier series takes the
form
\[
f_1(t)+f_2(-1/t)=a_0+\sum_{n\in\Z_{\ne0}}\big(a_n\,\e^{\imag\pi nt}
+b_n\,\e^{-\imag\pi n/t}\big).
\]
\noindent {\sc Underlying philosophy:} If we think in terms of $f_1,f_2$ as 
functions, then each of them is pretty much arbitrary on the interval $]-1,1[$, 
with periodic extensions beyond that interval. In particular, 
$f_2(-1/t)$ is an arbitrary function on $\R\setminus[-1,1]$, so together with 
$f_1(t)$, which is arbitrary on $]-1,1[$, there is a chance that we may be able 
to represent more or less any function or distribution on the line.
This is inspired
by ideas from differential geometry, where we use local patches to describe a
given manifold. Here, we would use as coordinate functions $t\mapsto t$ on
$]-1,1[$ and $t\mapsto -1/t$ on $\R\setminus[-1,1]$. See Figure \ref{fig:1} for
a visualization of the compactified real line $\R\cup\{\infty\}$ as a 
great circle on the Riemann sphere. 

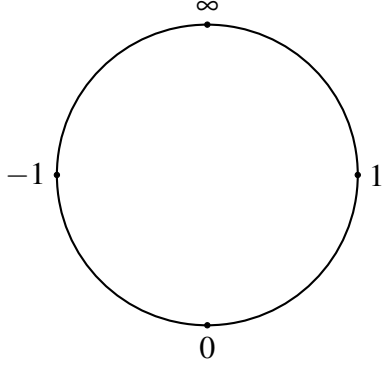
\begin{figure}
\label{fig:1}
\begin{center}
\begin{tikzpicture}
\draw[thick] (2,1) circle [radius=2];
\draw(4,1)node[right]{$1$};
\draw(2,3)node[above]{$\infty$};
\draw(0,1)node[left]{$-1$};
\draw(2,-1)node[below]{$0$};
\filldraw [black] (0,1) circle (1pt);
\filldraw [black] (4,1) circle (1pt);
\filldraw [black] (2,3) circle (1pt);
\filldraw [black] (2,-1) circle (1pt);
\end{tikzpicture}
\end{center}
\caption{The compactified real line thought of as a circle}
\end{figure}



\subsection{Convergence of hyperbolic Fourier series to a
distribution}
The above definition of hyperbolic Fourier series is formal. We now turn to
convergence issues. Let $\Jop_1,\Jop_1^\ast$ denote the transformations
\begin{equation}
\Jop_1\varphi(t):=t^{-2}\varphi(-1/t),\quad \Jop_1^\ast\psi(t):=\psi(-1/t), 
\label{eq:Jop1}
\end{equation}
for real $t$. By the change-of-variables formula, we know that
\begin{equation}
\langle \Jop_1\varphi,\psi\rangle_\R
=\langle\varphi,\Jop_1^\ast \psi\rangle_\R,
\label{eq:Jop2}
\end{equation}
whenever any of the two sides is well-defined as a Lebesgue integral.
If $\psi$ is a distribution on $\R$, we can use \eqref{eq:Jop2} as the 
definition of $\Jop_1^\ast\psi$ as a distribution. This suggests that we
should introduce as space of test functions
\[
C^\infty_1(\R\cup\{\infty\}):=\big\{\varphi\in C^\infty(\R):\,\Jop_1\varphi\in
C^\infty(\R)\big\},
\]
which amounts to considering the space of test functions that give rise to
$C^\infty$-smooth $1$-forms on the great circle model of $\R\cup\{\infty\}$
(see Figure \ref{fig:1}). The dual space to $C^\infty_1(\R\cup\{\infty\})$ we
call the \emph{distributions on the extended real line} $\R\cup\{\infty\}$.
Note that the test function space $C^\infty_1(\R\cup\{\infty\})$ contains
the well-known Schwartzian test function space $\mathscr{S}(\R)$ as a subspace.
We will write
\begin{equation}
e_n(t):=\e^{\imag\pi nt},\quad \Jop_1^\ast e_n(t)=\e^{-\imag\pi n/t},
\label{eq:HFSbasis1}
\end{equation}
to have a condensed notation for the basis elements of hyperbolic 
Fourier series.

\begin{defn}
\label{defn:distrconv}
For a distribution $f$  on the extended real line, we say that the hyperbolic
Fourier series expansion
\begin{equation}
f(t)
=a_0+\sum_{n\in\Z_{\ne0}}\big(a_n\,\e^{\imag\pi nt}+b_n\,\e^{-\imag\pi n/t}\big)
=a_0+\sum_{n\in\Z_{\ne0}}\big(a_n\,e_n(t)+b_n\,\Jop_1^\ast e_n(t)\big)
\label{eq:HFS0.1}
\end{equation}
\emph{converges in the sense of distribution theory} if for each 
$\varphi\in C^\infty_1(\R\cup\{\infty\})$,
\[
\langle\varphi,f\rangle_\R=a_0\langle \varphi,1\rangle_\R+
\sum_{n\in\Z_{\ne0}}a_n\,\langle \varphi, e_n\rangle_\R+
\sum_{n\in\Z_{\ne0}}b_n\,\langle\Jop_1\varphi, e_n\rangle_\R
\]
holds with absolute convergence. 
\end{defn}

\begin{rem}
The convergence in the above sense of distribution theory 
holds if and only if the coefficients grow at most polynomially: 
\begin{equation}
\exists k:\,\,\,|a_n|,|b_n|=\Ordo(|n|^k) \quad \text{as}\quad |n|\to+\infty.
\label{eq:growthcoeff0.00}
\end{equation}
\end{rem}

\subsection{Existence and uniqueness of hyperbolic Fourier
series expansion}
A first version of the main result on hyperbolic Fourier series runs as
follows.

\begin{thm}
\label{thm:hfs1}
Suppose $f$ is a distribution on the extended real line $\R\cup\{\infty\}$.
Then $f$ may be expanded in a hyperbolic Fourier series
\[
f(t)=a_0+\sum_{n\in\Z_{\ne0}}a_n\,\e^{\imag\pi nt}+b_n\,
\e^{-\imag\pi n/t},
\]
with coefficients that grow at most polynomially, i.e., with 
\eqref{eq:growthcoeff0.00}, so that the series converges in
the sense of distribution theory. 
These coefficients are then uniquely determined 
by $f$: $a_0=a_0(f)$, $a_n=a_n(f)$, $b_n=b_n(f)$.
\end{thm}

The proof of Theorem \ref{thm:hfs1} is supplied in Subsection
\ref{subsec:hfs_distr1} (see, e.g., Remark \ref{rem:equivalent_hfs}).

\begin{rem}
This theorem was anticipated in the 2011 work of Hedenmalm and 
Montes-Rodr\'\i{}guez \cite{hed}, where it was shown that the system of 
unimodular functions $e_n(t)=\e^{\imag\pi nt}$ and 
$\Jop_1^\ast e_n(t)=\e^{-\imag\pi n/t}$ together spanned a weak-star 
dense subspace of $L^\infty(\R)$ as $n$ ranges over the integers. 
It was the additional insight of Viazovska et al which was based on the
need to construct specific magic functions connected with optimal sphere 
packing problems in dimensions $8$ and $24$ which provided the 
framework to attain the result, see \cite{Viaz}, \cite{ckrv}. We should mention 
that the work on Fourier interpolation by Radchenko and Viazovska 
\cite{rad} is quite close to the above theorem, see, for instance, Section 
\ref{sec:radial_FI} below.
A holomorphic version of this theorem also appeared recently in \cite{bon}.
\end{rem}

More precise control may be obtained of the coefficients if 
the distribution we expand is a point mass.

\begin{cor}
\label{cor:hfs1.1}
Suppose $\delta_x$ is the associated unit point 
mass at a point $x\in\R$.
This point mass may be expanded in a hyperbolic Fourier series
\[
\delta_x(t)=A_0(x)+\sum_{n\in\Z_{\ne0}}\big(A_n(x)\,\e^{\imag\pi nt}+
B_n(x)\,\e^{-\imag\pi n/t}\big),
\]
which converges in the sense of distribution theory.
The coefficient functions $A_0,A_n,B_n$ for $n\in\Z_{\ne0}$ are all in
in the test function space $C^\infty_1(\R\cup\{\infty\})$, and
enjoy the growth control $A_0(x)=\Ordo(1+x^2)^{-1}$
and
\[
A_n(x),\,B_n(x)=\Ordo\bigg(\frac{|n|\log^2(|n|+1)}{1+x^2}\bigg),\qquad n\ne0,
\]
uniformly in $n$ and $x\in\R$. 
\end{cor}

The proof of Corollary \ref{cor:hfs1.1} is supplied in Subsection
\ref{subsec:HFS_Cauchy}.

As stated, the coefficients $A_n(x)$ and $B_n(x)$ are smooth functions of
$x$, and we will study them in some detail later on.
See Figure \ref{fig:A07} for the graph of the function $t\mapsto A_n(t)$
with $n=7$. The symmetry properties of the functions $A_n,B_n$ are listed
below.

\begin{prop}
\label{prop:symmetry_hfs_first1}
The functions $A_0$ and $A_n,B_n$ for $n\in\Z_{\ne0}$
have the following symmetry properties:
\[
A_n(x)=\overline{A_{-n}(x)}=A_{-n}(-x),\qquad n\in\Z,\,\,x\in\R,
\]
and
\[
B_n(x)=\overline{B_{-n}(x)}=B_{-n}(-x),\qquad n\in\Z_{\ne0},\,\,x\in\R.
\]  
They also enjoy the properties that
\[
A_0(x)=\Jop_1 A_0(x)=x^{-2}A_0(-1/x),\quad B_n(x)=\Jop_1 A_n(x)=x^{-2}A_n(-1/x),
\]
for all $x\in\R$ and all $n\in\Z_{\ne0}$. 
\end{prop}

Proposition \ref{prop:symmetry_hfs_first1} gets restated in Subsection
\ref{subsec:Cauchy_Poisson} in the form of Proposition \ref{prop:symmetry011}.
and the proof is supplied there.

\begin{figure}
\includegraphics[width=0.7\linewidth]
{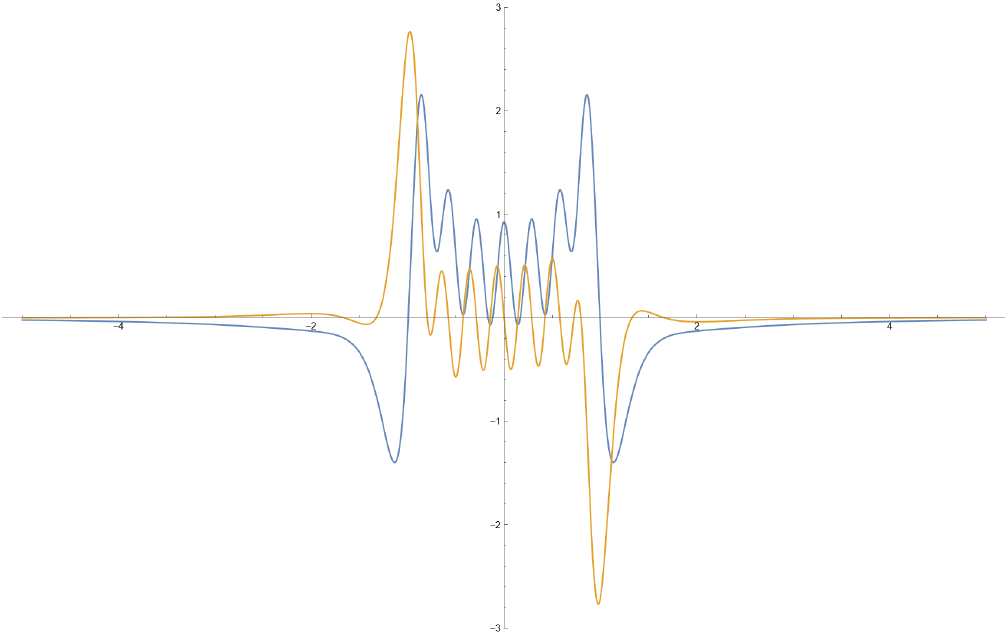}
\caption{Graph of the function $A_n(t)$ for $n=7$, real part in blue, 
imaginary part in yellow.}
\label{fig:A07}
\end{figure}

In particular, it follows that once we find the functions $A_n$ for
$n\in\Z_{\ge0}$, we automatically have the functions $A_0$ and $A_n,B_n$
for all $n\in\Z_{\ne0}$ as well. These functions form a biorthogonal system to
the basis \eqref{eq:HFSbasis1} associated with hyperbolic Fourier series.

\begin{cor}
\label{cor:biorthogonal_2.3.5}
The functions $A_0,A_n,B_n$ for $n\in\Z_{\ne0}$ have the following properties:
\[
\langle  A_0,\Jop_1^\ast e_m\rangle_\R=
\langle A_0,e_m\rangle_\R=\int_\R A_0(x)e_m(x)\diff x=\delta_{m,0},\qquad
m\in\Z,  
\]
while
\[
\langle B_n,\Jop_1^\ast e_m\rangle_\R=\langle A_n,e_m\rangle_\R
=\int_\R A_n(x)e_m(x)\diff x=\delta_{m,n},\qquad
m\in\Z,n\in\Z_{\ne0},  
\]
and
\[
\langle A_n,\Jop_1^\ast e_m\rangle_\R=
\langle B_n,e_m\rangle_\R=\int_\R B_n(x)e_m(x)\diff x=0,\qquad
m\in\Z,\,n\in\Z_{\ne0}.
\]  
Here, $\delta_{m,n}$ stands for the Kronecker delta, which equals $1$ when
$m=n$ and vanishes when $m\ne n$. 
\end{cor}  

The proof of Corollary \ref{cor:biorthogonal_2.3.5} is supplied in Subsection
\ref{subsec:HFS_Cauchy}.
Although the functions $A_0$ and $A_n,B_n$ for $n\in\Z_{\ne0}$ are rather
mysterious, it is possible to evaluate them at the origin in terms of
the counting functions $\mathrm{r}_4(n,\Z)$ and $\mathrm{r}_4(n,\Z+\frac12)$.
Here, $\Z+\frac12$ stands for the set of all half-integers that are not
integers, and, given an $x\in\R_{\ge0}$ and a subset $A\subset\R$,
\[
\mathrm{r}_4(x,A)=\#\big\{(a_1,a_2,a_3,a_4)\in A^4:
\,a_1^2+a_2^2+a_3^2+a_4^2=x\big\}
\]
where $\#$ counts the number of occurrences. 

\begin{prop}
We have that
\[
A_0(0)=\frac{4}{\pi^2}\log2
\]
while for $n\in\Z_{\ne0}$, it holds that
\[
A_n(0)=\frac{\mathrm{r}_4(|n|,\Z+\tfrac12)}{2\pi^2|n|},
\quad
A_n(0)+B_n(0)=\frac{\mathrm{r}_4(|n|,\Z)}{2\pi^2|n|}.
\]  
\label{prop:2.3.6}
\end{prop}

The proof of Proposition \ref{prop:2.3.6} is deferred to a separate
publication \cite{Hed-Jacobi}. We believe the topic deserves to be developed
for its own sake. 

We shall need to take Theorem \ref{thm:hfs1} (which says that distributions
have unique hyperbolic Fourier series expansions)
beyond the realm of distribution theory to a more general context
(ultradistribution theory). To avoid excessive 
technicalities we shall present the expansions in terms of
harmonic extensions to the upper half-plane instead of
introducing narrower classes of test functions. This allows us a direct approach
to hyperbolic Fourier series expansions, and the appropriate classes of test 
functions may be considered later on when necessary.

\subsection{Harmonic extension and distribution theory}
\label{subsec:harmext-gen}
Let $P(\tau,t)$ denote the \emph{Poisson kernel}
\begin{equation}
P(\tau,t):=\pi^{-1}\frac{\im\tau}{|\tau-t|^2},\qquad \tau\in\Hyp,\,\,\,t\in\R,
\label{eq:Poissonkernel1.1}
\end{equation}
where $\Hyp$ denotes the upper half-plane
\[
\Hyp:=\{\tau\in\C:\,\im\tau>0\}.
\]
We observe that for $\tau\in\Hyp$, the function $P(\tau,\cdot)$ belongs to
the space of test functions $C^\infty_1(\R\cup\{\infty\})$, and it depends 
harmonically on $\tau$. It follows that given a distribution $f$ on the
extended  real line $\R\cup\{\infty\}$, the formula
\begin{equation}
\tilde f(\tau):=\langle P(\tau,\cdot),f\rangle_\R,\qquad\tau\in\Hyp,
\label{eq:Poisson0}
\end{equation}
defines a harmonic function, called the \emph{Poisson extension},
with polynomial growth bound
\begin{equation}
|\tilde f(\tau)|=\Ordo\bigg(\frac{1+|\tau|^2}{\im\tau}\bigg)^k
\label{eq:polgrowth0}
\end{equation}
with an implicit constant that is uniform in $\Hyp$, for some finite $k$.
By considering the distributional limit of $\tilde f(\tau)$
as $\tau$ approaches the real line, we recover $f$ from $\tilde f$. 
Given a harmonic function $h:\Hyp\to\C$ with the same growth bound
as $\tilde f$, we get distributional boundary values on $\R\cup\{\infty\}$,
and then the harmonic function $h$ equals the Poisson extension of 
its boundary distribution. In other words, we have the following.

\begin{prop} {}
The distributions $f$ on the extended real line are in a one-to-one 
correspondence with the harmonic functions $\tilde f$ subject to the 
polynomial growth bound \eqref{eq:polgrowth0}.
\label{prop:Schwartz}
\end{prop}

We believe this result goes back to the work of Laurent Schwartz.
However, lacking a direct reference to Schwartz, we found instead
the reference Straube \cite{Straube}. If we use the Cayley transform to move
the upper half-plane to the unit disk, Straube's result applied to the disk
gives the assertion. We mention the paper by Korenblum \cite{Kor-Acta1}
which shows that in the case of the disk the distributional interpretation
of the boundary values was well understood in the holomorphic setting.
These results are to be applied for the unit disk,
and the Cayley transform then moves between the disk and the upper half-plane. 
In the sequel, we shall equate the harmonic extension $\tilde f$ with the 
distribution $f$ on the extended real line, and hence suppress the tilde 
notation.

\begin{rem}
The connection between smoothness of test functions and the
growth on the dual side of holomorphic or harmonic functions
can be very detailed. See, e.g., \cite{BorHed} for a study of
quasianalytic smoothness in terms of almost holomorphic extensions.
\end{rem}

\subsection{Harmonic extension of hyperbolic Fourier series}
\label{subsec:harmext-hfs}
If we have a hyperbolic Fourier series expansion
\[
f(t)=a_0+\sum_{n\in\Z_{\ne0}}\big(a_n\,e^{\imag\pi nt}+b_n\,\e^{-\imag\pi n/t}\big)
=a_0+\sum_{n\in\Z_{\ne0}}\big(a_n\,e_n(t)+b_n\,\Jop_1^\ast e_n(t)\big)
\]
with at most polynomial growth of coefficients \eqref{eq:growthcoeff0.00}, 
then the series converges in the sense of distribution theory to
the distribution $f$ on the extended real line.
The harmonic extension to the upper half-plane $\Hyp$ of this equality
then reads
\begin{equation}
f(\tau)=a_0+\sum_{n\in\Z_{\ne0}}\big(a_n\,e_n(\tau)+
b_n\,(\Jop_1^\ast e_n)(\tau)\big),\qquad\tau\in\Hyp,
\label{eq:harmext1.1}
\end{equation}
where we omit the tildes of the Poisson extensions given by 
\eqref{eq:Poisson0}. In particular, we have that
\begin{equation}
e_n(\tau)=\e^{\imag\pi n\tau} \,\,\,\text{and}\,\,\, (\Jop_1^\ast e_n)(\tau)=
\e^{-\imag\pi n/\tau},\qquad n\in\Z_{\ge0},
\label{eq:harmext2}
\end{equation}
while  
\begin{equation}
e_n(\tau)=\e^{\imag\pi n\bar\tau} \,\,\,\text{and}\,\,\, (\Jop_1^\ast e_n)(\tau)=
\e^{-\imag\pi n/\bar\tau},\qquad n\in\Z_{<0}.
\label{eq:harmext3}
\end{equation}
The representation \eqref{eq:harmext1.1} then naturally splits as 
\begin{equation}
f(\tau)=a_0+f_+(\tau)+f_-(\tau),
\label{eq:harmext4}
\end{equation}
where 
\begin{equation}
f_+(\tau):=\sum_{n\in\Z_{>0}}\big(a_n\,\e^{\imag\pi n\tau}+
b_n\,\e^{-\imag\pi n/\tau}\big),\qquad\tau\in\Hyp,
\label{eq:harmext5}
\end{equation}
is holomorphic, whereas  
\begin{equation}
f_-(\tau):=\sum_{n\in\Z_{<0}}\big(a_n\,\e^{\imag\pi n\bar\tau}+
b_n\,\e^{-\imag\pi n/\bar\tau}\big),\qquad\tau\in\Hyp.
\label{eq:harmext6}
\end{equation}
is conjugate-holomorphic. The imposed growth control on the coefficients
\eqref{eq:growthcoeff0.00} guarantees that each of $f_+,f_-$ 
meets the growth condition \eqref{eq:polgrowth0}.
We call the expansion \eqref{eq:harmext1.1}, or, which is the same,
\eqref{eq:harmext4} combined with \eqref{eq:harmext5} and 
\eqref{eq:harmext6}, the \emph{harmonic hyperbolic Fourier series
expansion} of $f$. The expansion \eqref{eq:harmext5} is referred to as a
\emph{holomorphic hyperbolic Fourier series} without constant term.
Likewise, \eqref{eq:harmext6} is a \emph{conjugate-holomorphic 
hyperbolic Fourier series} without constant term. We are free to add a
constant term to the holomorphic hyperbolic Fourier and still call it that.
The same goes for the conjugate-holomorphic hyperbolic Fourier series.

\subsection{Hyperbolic Fourier series expansion of a
harmonic function}
We begin with the version of Theorem \ref{thm:hfs1} which applies to harmonic
functions in the sense of Proposition \ref{prop:Schwartz}.

\begin{thm}
\label{thm:hfs2}
Suppose $f:\Hyp\to\C$ is harmonic, subject to the growth bound
\begin{equation*}
\exists k:\,\,\,|f(\tau)|=\Ordo\bigg(\frac{1+|\tau|^2}{\im\tau}\bigg)^k
\end{equation*}
with an implicit constant that is uniform in $\Hyp$.
Then $f$ may be expanded in a harmonic hyperbolic Fourier series
\[
f(\tau)=a_0+\sum_{n\in\Z_{>0}}\big(a_n\,\e^{\imag\pi n\tau}+b_n\,
\e^{-\imag\pi n/\tau}\big)
+\sum_{n\in\Z_{<0}}\big(a_n\,\e^{\imag\pi n\bar\tau}+b_n\,
\e^{-\imag\pi n/\bar\tau}\big),\qquad\tau\in\Hyp,
\]
with coefficients that grow at most polynomially, i.e., with 
\eqref{eq:growthcoeff0.00} . 
These coefficients are then uniquely determined 
by $f$: $a_0=a_0(f)$, $a_n=a_n(f)$, $b_n=b_n(f)$.
\end{thm}

The proof of Theorem \ref{thm:hfs2} is supplied in Subsection
\ref{subsec:hfs_distr1}.

\begin{rem}
The series can be seen to converge in the norm 
\[
\|f\|_{k}:=\sup_{\tau\in\Hyp}\bigg(\frac{\im\tau}{1+|\tau|^2}\bigg)^{k}|f(\tau)|
\]
for some finite $k$ (not necessarily the same as the various parameters
$k$ in the theorem). Moreover, since convergence in this norm implies
convergence of the boundary distribution, the assertions of 
Theorems \ref{thm:hfs1} and \ref{thm:hfs2} are equivalent.
\end{rem}

A Taylor series
\[
g(\zeta)=\sum_{n=0}^{+\infty}c_n\zeta^n
\]
converges in the unit disk 
\[
\D:=\{\zeta\in\C:\,\,|\zeta|<1\}
\]
if and only if the coefficients grow at most subexponentially:
\begin{equation}
\forall\epsilon>0:\,\,|c_n|=\Ordo(\e^{\epsilon n})\quad\text{as}\,\,\,n\to+\infty.
\label{eq:fs1.01}
\end{equation}
Writing $\zeta=\e^{\imag\pi \tau}$, it follows that the Fourier series
\[
g(\e^{\imag\pi\tau})=\sum_{n=0}^{+\infty}c_n\e^{\imag\pi n\tau}
\]
converges on $\Hyp$ if and only if \eqref{eq:fs1.01} holds. Analogously,
if we instead write
$\zeta=\e^{-\imag\pi/\tau}$, we again see that 
\[
g(\e^{-\imag\pi/\tau})=\sum_{n=0}^{+\infty}c_n\e^{-\imag\pi n/\tau}
\]
converges on $\Hyp$ if and only if \eqref{eq:fs1.01} holds. 
From these considerations we easily obtain the following proposition.

\begin{prop}
Suppose that the coefficients $a_n,b_n$ are subject to the 
subexponential growth bound
\begin{equation}
\forall\epsilon>0:\,\,|a_n|,|b_n|=\Ordo(\e^{\epsilon |n|})\quad\text{as}\,\,\,
|n|\to+\infty.
\end{equation}
Then the harmonic hyperbolic Fourier series 
\[
f(\tau)=a_0+\sum_{n\in\Z_{>0}}\big(a_n\,\e^{\imag\pi n\tau}+b_n\,
\e^{-\imag\pi n/\tau}\big)
+\sum_{n\in\Z_{<0}}\big(a_n\,\e^{\imag\pi n\bar\tau}+b_n\,
\e^{-\imag\pi n/\bar\tau}\big),\qquad\tau\in\Hyp,
\]
converges absolutely and uniformly on compact subsets to a harmonic
function $f:\Hyp\to\C$.
\end{prop}

The natural question pops up whether the harmonic hyperbolic Fourier
series represent some small subclass of the harmonic functions. 
Pleasantly, it is not so.

\begin{thm}
\label{thm:HFS-general.1.1}
Suppose that $f:\Hyp\to\C$ is a harmonic function. 
Then there exist coefficients $a_n,b_n$ with the growth bound
\begin{equation}
\forall\epsilon>0:\,\,|a_n|,|b_n|=\Ordo(\e^{\epsilon |n|})\quad\text{as}\,\,\,
|n|\to+\infty,
\end{equation}
such that $f$ enjoys the harmonic hyperbolic Fourier series expansion
\[
f(\tau)=a_0+\sum_{n\in\Z_{>0}}\big(a_n\,\e^{\imag\pi n\tau}+b_n\,
\e^{-\imag\pi n/\tau}\big)
+\sum_{n\in\Z_{<0}}\big(a_n\,\e^{\imag\pi n\bar\tau}+b_n\,
\e^{-\imag\pi n/\bar\tau}\big),\qquad\tau\in\Hyp.
\]
\end{thm}

The proof of Theorem \ref{thm:HFS-general.1.1} is supplied in Subsection
\ref{subsec:Solutionscheme-gen}.

We note that when we allow for faster growing coefficients, 
uniqueness is generally speaking lost.

A general reference to elliptic function theory is Chandrasekharan's 
book \cite{cha}.

\begin{example}
Let $\vartheta_{10}(\tau)$ and $\vartheta_{00}(\tau)$ denote the
Jacobi theta functions (with first variable $z=0$, omitted)
\begin{equation}
\vartheta_{10}(\tau):=\sum_{n\in\Z}\e^{\imag\pi(n+\frac12)^2\tau},
\qquad
\vartheta_{00}(\tau):=\sum_{n\in\Z}\e^{\imag\pi n^2\tau},
\label{eq:Jacobitheta00}
\end{equation}
and put
\begin{equation}
\lambda(\tau):=\frac{\vartheta_{10}(\tau)^4}{\vartheta_{00}(\tau)^4},
\qquad \tau\in\Hyp.
\label{eq:lambdaf1}
\end{equation}
This function is called the \emph{modular lambda function}, and since
$\vartheta_{00}$ is zero-free in $\Hyp$, it is holomorphic in $\Hyp$.
It is $2$-periodic and hence it has a Fourier expansion
\begin{equation}
\lambda(\tau)=
\sum_{n=1}^{+\infty}\hat\lambda(n)\,\e^{\imag\pi n\tau},\qquad
\tau\in\Hyp,
\label{eq:lambdaf2}
\end{equation}
where the coefficients turn out to be integers, the first being 
$\hat\lambda(1)=16$,
$\hat\lambda(2)=-128$, and $\hat\lambda(3)=704$. Moreover, it is 
known that the coefficients have the asymptotics
\[
\hat\lambda(n)=(1+\ordo(1))\frac{(-1)^{n-1}}{32}n^{-\frac34}
\exp(2\pi\sqrt{n})\quad\text{as}\,\,\, n\to+\infty.
\]
The lambda function has the important functional properties
\begin{equation}
\lambda(\tau+1)=\frac{\lambda(\tau)}{\lambda(\tau)-1},\quad
\lambda(-{1}/{\tau})=1-\lambda(\tau).
\label{eq:lambdaproperties1.1}
\end{equation}
In particular, the latter property entails that
\[
-1+\lambda(\tau)+\lambda(-1/\tau)=-1+\sum_{n=1}^{+\infty}
\hat\lambda(n)\big(\e^{\imag n\tau}+\e^{-\imag\pi n/\tau}\big)=0,
\qquad \tau\in\Hyp,
\]
a nontrivial representation of the $0$ function in a holomorphic 
hyperbolic Fourier series with coefficients $a_0=-1$ and
$a_n=b_n=\hat\lambda(n)$ for $n\in\Z_{>0}$. This means that
uniqueness fails in the harmonic hyperbolic series expansion 
if we allow the coefficients to grow as quickly as 
$\hat\lambda(n)$.
\label{ex:nonuniqueness}
\end{example}

\begin{thm} {\rm (uniqueness)}
Suppose that the coefficients $a_n,b_n$ are subject to the 
subexponential growth bound
\begin{equation}
\forall\epsilon>0:\,\,|a_n|,|b_n|=\Ordo(\e^{\epsilon |n|})\quad\text{as}\,\,\,
|n|\to+\infty,
\end{equation}
and that associated harmonic hyperbolic Fourier series is trivial: 
\[
a_0+\sum_{n\in\Z_{>0}}\big(a_n\,\e^{\imag\pi n\tau}+b_n\,
\e^{-\imag\pi n/\tau}\big)
+\sum_{n\in\Z_{<0}}\big(a_n\,\e^{\imag\pi n\bar\tau}+b_n\,
\e^{-\imag\pi n/\bar\tau}\big)=0,\qquad\tau\in\Hyp.
\]
Suppose in addition that 
\[
|a_n|=\ordo\big(|n|^{-3/4}\exp\big(2\pi\sqrt{|n|}\big)\big)\quad \text{as}
\,\,\,|n|\to+\infty.
\]
Then $a_0=0$ and $a_n=b_n=0$ for all $n\in\Z_{\ne0}$. 
\label{thm:uniq1}
\end{thm}

The proof of Theorem \ref{thm:uniq1} is supplied in Subsection
\ref{subsec:uniq-beyond}. 

\begin{rem}
\label{rmk:uniq1}
(a)
Theorem \ref{thm:uniq1} would remain true if we make the growth 
assumption on the coefficients $b_n$ in place of the $a_n$. 
In addition we may replace the two-sided assumption on $a_n$ (in the sense
that the assumption concerns the limits as $n\to+\infty$ and as $n\to-\infty$) 
by a one-sided assumption on $a_n$, and a likewise one-sided assumption on 
$b_n$, as long as the two sides are in opposite directions.

(b) Note that the theorem is sharp in the sense that if the little
"o" is replaced by big "O", the assertion fails, by
Example \ref{ex:nonuniqueness}. 
\end{rem}

\subsection{Beyond uniqueness of the harmonic hyperbolic
Fourier series expansion}
We now describe a more general situation where the harmonic hyperbolic
Fourier series fails to be unique due to a finite-dimensional 
space of \emph{exceptional coefficients}.

\begin{defn}
The sequence of coefficients $\{a_0,a_n,b_n\}_{n>0}$ is 
\emph{exceptional of degree}  $\le k$ if there exists a polynomial $P$ of
degree $\le k$ with $P(0)=0$, $P(1)=-a_0$, such that
\[
\sum_{n=1}^{+\infty}a_n\,\e^{\imag\pi n\tau}=P(\lambda(\tau)),\quad
\sum_{n=1}^{+\infty}b_n\,\e^{\imag\pi n\tau}=-a_0-P(1-\lambda(\tau)),
\]
for all $\tau\in\Hyp$. We write $\mathfrak{E}^{\mathrm{hol}}_k$ for
the linear space of all exceptional coefficients of degree $\le k$,
which has dimension $k$.
\label{def:exceptional:1-sided}
\end{defn}

We extend the definition to two-sided sequences.

\begin{defn}
The two-sided sequence of coefficients $\{a_0,a_n,b_n\}_{n\ne0}$ is 
\emph{exceptional of degree}  $\le k$ if for some parameter value 
$c_0\in\C$, the two one-sided sequences 
$\{a_0-c_0,a_n,b_n\}_{n>0}$ and $\{c_0,a_{-n},b_{-n}\}_{n>0}$ 
are both exceptional of degree $\le k$. 
We write $\mathfrak{E}^{\mathrm{harm}}_k$ for
the linear space of all two-sided exceptional coefficients of degree 
$\le k$, which has dimension $2k$.
\label{def:exceptional:2-sided}
\end{defn}

When we let the coefficients grow faster than the range permitted 
in Theorem \ref{thm:uniq1}, we do get nonuniqueness, but at least
for a while we can still control the loss of uniqueness and describe it
in terms of the exceptional coefficient classes.

\begin{thm} {\rm (beyond uniqueness)}
Suppose that the coefficients $a_n,b_n$ are subject to the 
subexponential growth bound
\begin{equation}
\forall\epsilon>0:\,\,|a_n|,|b_n|=\Ordo(\e^{\epsilon |n|})\quad\text{as}\,\,\,
|n|\to+\infty,
\end{equation}
and that associated harmonic hyperbolic Fourier series is trivial: 
\[
a_0+\sum_{n\in\Z_{>0}}\big(a_n\,\e^{\imag\pi n\tau}+b_n\,
\e^{-\imag\pi n/\tau}\big)
+\sum_{n\in\Z_{<0}}\big(a_n\,\e^{\imag\pi n\bar\tau}+b_n\,
\e^{-\imag\pi n/\bar\tau}\big)=0,\qquad\tau\in\Hyp.
\]
Suppose, in addition, that for some integer $k=1,2,3,\ldots$,
\[
|a_n|=\ordo\big(|n|^{-3/4}\exp\big(2\pi\sqrt{k|n|}\big)\big)
\quad \text{as}\,\,\,
|n|\to+\infty.
\]
Then $\{a_0,a_n,b_n\}_{n\ne0}$ belongs to the space of two-sided
exceptional sequences  $\mathfrak{E}^{\mathrm{harm}}_{k-1}$.
\label{thm:uniq2.1}
\end{thm}

The proof of Theorem \ref{thm:uniq2.1} is supplied in Subsection
\ref{subsec:uniq-beyond}. 

\begin{rem}
As in Remark \ref{rmk:uniq1}(a), Theorem \ref{thm:uniq2.1} would 
remain true if  we make the growth assumption on the coefficients 
$b_n$ in  place of the $a_n$. 
In addition we may replace the two-sided assumption on 
$a_n$ (in the sense that the assumption concerns the limits as 
$n\to+\infty$ and as $n\to-\infty$) by a one-sided assumption on 
$a_n$, and a likewise one-sided assumption on 
$b_n$, as long as the two sides are in opposite directions.
\end{rem}

\subsection{Growth of coefficients and the hyperbolic Fourier
series}

Let $\Mfun:\Hyp\to\R_{>0}$ denote the function
\begin{equation}
\Mfun(\tau):=\frac{\max\{1,|\tau|^2\}}{\im\tau}.
\label{eq:defM}
\end{equation}
This function enjoys the following symmetry properties:
\begin{equation}
\Mfun(-1/\tau)=\Mfun(1/\bar\tau)=\Mfun(\tau),\qquad\tau\in\Hyp.
\label{eq:Mprop}
\end{equation}

We first show that if we control the coefficients, we obtain growth
control on the holomorphic hyperbolic Fourier series.

\begin{prop}
Suppose that for some fixed $\beta$ with $0<\beta<+\infty$, we have that
the sequence $\{a_n,b_n\}_{n>0}$ has the growth control
\[
|a_n|,|b_n|=\Ordo\big(n^{-3/4}\exp(2\pi\sqrt{\beta n})\big)\quad\text{as}
\,\,\,n\to+\infty.  
\]
Then, if $f:\Hyp\to\C$ is the associated holomorphic hyperbolic Fourier
series
\[
f(\tau)=\sum_{n=1}^{+\infty}\big(a_n \e^{\imag\pi n\tau}+b_n 
\e^{-\imag\pi n/\tau}\big),
\]
it follows that
\[
|f(\tau)|=\Ordo\big(\exp(\beta\pi\Mfun(\tau))\big)
\]
holds uniformly in $\tau\in\Hyp$. 
\label{prop:coeff->function0.0}
\end{prop}

Proposition \ref{prop:coeff->function0.0} is obtained from a straightforward
but rather tedious calculation, based on Lemma \ref{lem:sumcontrol1} which is
formulated and proved in Subsection \ref{subsec:coeffgrowthharmonic} below.
The instance of harmonic Fourier series
is stated in Proposition \ref{prop:coeff->function} below, and Proposition
\ref{prop:coeff->function0.0} then follows by adding up the two corresponding
estimates. The rather trivial details are left to the reader.

What is less straightforward is that given a holomorphic function of the
given growth, we may control the coefficients in the same way with only
a slight loss.

\begin{thm}
Suppose $f:\Hyp\to\C$ holomorphic with
\[
|f(\tau)|=\Ordo\big(\exp(\beta\pi\,\Mfun(\tau))\big)
\]
with an implicit constant that is uniform on $\Hyp$. Here, $\beta$ is a fixed
parameter with $0<\beta<+\infty$. Then $f$ may be written as a holomorphic
hyperbolic Fourier series
\[
f(\tau)=a_0+
\sum_{n=1}^{+\infty}\big(a_n \e^{\imag\pi n\tau}+b_n \e^{-\imag\pi n/\tau}\big),
\qquad\tau\in\Hyp,  
\]
where if $\beta$ is not an integer, the coefficients satisfy the growth
condition 
\[
|a_n|,|b_n|=\Ordo\big((\log^2n)\exp(2\pi\sqrt{\beta n})\big)\quad\text{as}
\,\,\,n\to+\infty.  
\]
On the other hand, if $\beta$ is an integer, we instead have the slightly weaker
estimate
\[
|a_n|,|b_n|=\Ordo\big(n^{1/2}\exp(2\pi\sqrt{\beta n})\big)\quad\text{as}
\,\,\,n\to+\infty.  
\]
\label{thm:function->coeff}
\end{thm}

The proof of Theorem \ref{thm:function->coeff} is supplied in Subsection
\ref{subsec:exponentially_growing}.

This is all well for holomorphic functions, but what about harmonic functions?
By splitting the given harmonic hyperbolic Fourier series according to
\eqref{eq:harmext4}, we estimate each term in \eqref{eq:harmext5} and
\eqref{eq:harmext6} in accordance with Proposition \ref{prop:coeff->function},
which gives us the harmonic hyperbolic Fourier series version of
Proposition \ref{prop:coeff->function}. 
On the other hand, if we have a harmonic function with the same growth
bound as satisfied by the holomorphic function of Theorem
\ref{thm:function->coeff}, we would like to expand it as a harmonic hyperbolic
Fourier series with control on the coefficients. We can achieve this by
obtaining a splitting of the harmonic function in a holomorphic and a
conjugate-holomorphic piece.

\begin{prop}
Suppose $f:\Hyp\to\C$ harmonic with
\[
|f(\tau)|=\Ordo\big(\exp(\beta\pi\,\Mfun(\tau))\big)
\]
with an implicit constant that is uniform on $\Hyp$. Here, $\beta$ is a fixed
parameter with $0<\beta<+\infty$. Then $f$ splits as
\[
f(\tau)=f(\imag)+f_+(\tau)+f_-(\tau),
\]
where  $f_+(\tau)$ is holomorphic in $\tau$, while
$f_-(\tau)$ is conjugate-holo\-morphic, with $f_+(\imag)=f_-(\imag)=0$,
each having the same growth bound:
\[
|f_+(\tau)|,|f_-(\tau)|=\Ordo\big(\exp(\beta\pi\,\Mfun(\tau))\big)
\]
with an implicit constant that is uniform in $\Hyp$.  
\label{prop:Mbetasplit}
\end{prop}

The proof of Proposition \ref{prop:Mbetasplit} can be found in
Subsection \ref{subsec:splitting}, and is restated in the form of
Proposition \ref{prop:beta1}.

\subsection{Power skewed holomorphic hyperbolic Fourier series}
\label{subsec:powerskewed}
If we begin with a holomorphic hyperbolic Fourier series
\[
f(\tau)=a_0+\sum_{n=1}^{+\infty}\big(a_n\,\e^{\imag\pi n\tau}+
b_n\,\e^{-\imag\pi n/\tau}\big),
\]
such that
\begin{equation}
\label{eq:subexponential1}
\forall\epsilon>0:\,\,|a_n|,|b_n|=\Ordo(\e^{\epsilon n})\quad\text{as}\,\,\,
n\to+\infty,
\end{equation}
so that the series converges in $\Hyp$, we may form the derivative:
\[
f'(\tau)=\sum_{n=1}^{+\infty}\big(\imag\pi n\,a_n\,\e^{\imag\pi n\tau}+
\imag\pi n\tau^{-2}\,b_n\,\e^{-\imag\pi n/\tau}\big),
\]
which has a similar form to the original holomorphic hyperbolic Fourier
series. The main difference is that the constant term disappeared and
that we got the factor $\tau^{-2}$ in the second expression involving the
coefficients $b_n$. This suggests the following more general type of
series, which we may call \emph{power skewed holomorphic hyperbolic
Fourier series}:
\begin{equation}
f(\tau)=a_0+\sum_{n=1}^{+\infty}\big(a_n\,\e^{\imag\pi n\tau}+
b_n\,(\tau/\imag)^{-\beta}\,\e^{-\imag\pi n/\tau}\big),
\label{eq:pskew1}
\end{equation}
where $\beta$ is a real parameter (called the \emph{skew power}).
The instance $\beta=0$ is the holomorphic hyperbolic Fourier series
we have discussed extensively already, whereas $\beta=2$ corresponds
to taking the derivative of holomorphic hyperbolic Fourier series (except 
that we allow the constant term to hang in there).
For $\tau\in\Hyp$, the ratio $\tau/\imag$ has positive real part, so we use
the principal branch of the logarithm to define the power 
\[
(\tau/\imag)^{-\beta}=\exp\big(-\beta\log(\tau/\imag)\big),\qquad\tau\in\Hyp.
\]
The power skewed holomorphic hyperbolic Fourier series 
\eqref{eq:pskew1} appears naturally in the context of Fourier interpolation
of radial functions in $\R^d$, with parameter $\beta=d/2$. More about
this later on. 
We shall now see how to relate power skewed holomorphic hyperbolic
Fourier series to the ordinary holomorphic hyperbolic Fourier series.
This is achieved by invoking the Jacobi theta function $\vartheta_{00}$
given by \eqref{eq:Jacobitheta00}, that is,
\begin{equation}
\vartheta_{00}(\tau):=\sum_{n\in\Z}\e^{\imag\pi n^2\tau},\qquad\tau\in\Hyp.
\label{eq:Jacobitheta01}
\end{equation}
The important fact here is that it is zero-free and enjoys the functional
identity
\begin{equation}
\vartheta_{00}(\tau)=(\tau/\imag)^{-1/2}\vartheta_{00}(-1/\tau),\qquad
\tau\in\Hyp.
\end{equation}
Since $\vartheta_{00}(\imag y)>1$ for $y\in\R_{>0}$, and 
$\vartheta_{00}(\tau)\ne0$ for all $\tau\in\Hyp$, we may choose a suitable 
holomorphic logarithm $\log\vartheta_{00}(\tau)$ which assumes positive 
values on $\imag\R_{>0}$. This logarithm is then also $2$-periodic:
$\log\vartheta_{00}(\tau+2)=\log\vartheta_{00}(\tau)$.
Using this logarithm, we define the power
\begin{equation}
\big(\vartheta_{00}(\tau)\big)^{\gamma}=\exp\big(\gamma\log\vartheta_{00}
(\tau)\big),\qquad\tau\in\Hyp,
\end{equation}
for any $\gamma\in\C$. The following proposition explains the algebraic
relationship between power skewed hyperbolic Fourier series and ordinary
hyperbolic Fourier series.

\begin{prop}  
\label{prop:pskewed1}
Suppose $f,g:\Hyp\to\C$ are holomorphic and connected via the formula
\[
g(\tau)=\vartheta_{00}(\tau)^{-2\beta}f(\tau),
\]
for some $\beta\in\R$. Then, if $g$ has a hyperbolic Fourier series 
expansion
\[
g(\tau)=\tilde a_0+g_1(\tau)+g_2(-1/\tau),
\]
where 
\[
g_1(\tau)=\sum_{n=1}^{+\infty}\tilde a_n\,\e^{\imag\pi n\tau},\quad
g_2(\tau)=\sum_{n=1}^{+\infty}\tilde b_n\,\e^{\imag\pi n\tau},
\]
it follows that $f$ has a power skewed expansion
\[
f(\tau)=a_0+f_1(\tau)+(\tau/\imag)^{-\beta}f_2(-1/\tau)
\]
where
\[
f_1(\tau)=\sum_{n=1}^{+\infty} a_n\,\e^{\imag\pi n\tau},\quad
f_2(\tau)=\sum_{n=1}^{+\infty} b_n\,\e^{\imag\pi n\tau},
\]
and vice versa. Here, the coefficient sequences $a_n,b_n$ and
$\tilde a_n,\tilde b_n$ meet the growth bound \eqref{eq:subexponential1}
simultaneously. Moreover, the relationship between the two expansions is
given by $\tilde a_0=a_0$ and
\[
f_1(\tau)=\big(\vartheta_{00}(\tau)^{2\beta}-1\big)\tilde a_0
+\vartheta_{00}(\tau)^{2\beta}g_1(\tau),\quad
f_2(\tau)=\vartheta_{00}(\tau)^{2\beta}g_2(\tau).
\]
\end{prop}

The proof of Proposition \ref{prop:pskewed1} is supplied in Subsection
\ref{subsec:powerskewed2}.

\begin{rem}
In terms of the Fourier series expansion
\[
\vartheta_{00}(\tau)^{2\beta}=\sum_{k=0}^{+\infty}\widehat{\vartheta_{00}^{2\beta}}(k)
\,\e^{\imag\pi k\tau},  
\]
we obtain the direct relationship between the coefficients in Proposition
\ref{prop:pskewed1}: for $n=1,2,3,\ldots$, we have
\[
a_n=\sum_{k=0}^{n}\tilde a_k\,\widehat{\vartheta_{00}^{2\beta}}(n-k),\quad
b_n=\sum_{k=1}^{n}\tilde b_k\,\widehat{\vartheta_{00}^{2\beta}}(n-k).  
\]
\end{rem}


\begin{cor}
Fix a parameter $\beta\in\R$, and suppose $f:\Hyp\to\C$ is holomorphic.
Then there exist complex coefficients $a_n,b_n$ with the growth bound
\begin{equation}
\forall\epsilon>0:\,\,|a_n|,|b_n|=\Ordo(\e^{\epsilon |n|})\quad\text{as}\,\,\,
|n|\to+\infty,
\end{equation}
such that $f$ has a $\beta$-power skewed holomorphic hyperbolic Fourier
series expansion
\[
f(\tau)=a_0+\sum_{n=1}^{+\infty}
\big(a_n \,\e^{\imag\pi n\tau}+b_n\,(\tau/\imag)^{-\beta}
\e^{-\imag\pi n/\tau}\big),\qquad\tau\in\Hyp.
\]
\end{cor}

\begin{proof}
The assertion is an immediate consequence of the holomorphic version of 
Theorem \ref{thm:HFS-general.1.1} combined with the algebraic 
correspondence of Proposition \ref{prop:pskewed1}.
\end{proof}

\begin{rem}
Naturally, we would like to know when the coefficients $a_n, b_n$
can be chosen to grow at most polynomially, for a given holomorphic 
function $f$. As in the context of Theorem \ref{thm:hfs2}, the appropriate 
growth condition on $f$ 
is
\begin{equation*}
\exists k:\,\,\,|f(\tau)|=\Ordo\bigg(\frac{1+|\tau|^2}{\im\tau}\bigg)^k
\end{equation*}
where the implicit constant is uniform in $\Hyp$. This would require us to
delve a little deeper into the methods that give Theorem 
\ref{thm:HFS-general.1.1}.
\end{rem}

\subsection{Exponentially skewed hyperbolic Fourier series}
\label{subsec:exponentiallyskewed}
Given two real parameters $\omega_1,\omega_2\in\R$, we may consider the 
associated \emph{exponentially skewed} holomorphic hyperbolic Fourier series
\begin{equation}
f(\tau)=f_{\omega_1,\omega_2}(\tau)
=a_0\,\e^{\imag\pi\omega_1\tau}+
\sum_{n=1}^{+\infty}\big(a_n\,\e^{\imag\pi(n+\omega_1)\tau}+
b_n\,\e^{-\imag\pi (n+\omega_2)/\tau}\big), \qquad\tau\in\Hyp,
\label{eq:expskew1}
\end{equation}
where the coefficients are assumed to grow at most subexponentially (this means
that they satisfy condition \eqref{eq:subexponential1}). It is possible to view
this as a shift in the ``frequency domain'' compared with the ordinary
holomorphic hyperbolic Fourier series. Moreover, the shifts
$(\omega_1,\omega_2)\mapsto (\omega_1+k_1,\omega_2+k_2)$ for
$k_1,k_2\in\Z_{\ge0}$ give rise to 
\begin{equation}
f_{\omega_1+k_1,\omega_2+k_2}(\tau)=
\sum_{n=k_1}^{+\infty}a_{n-k_1}\,\e^{\imag\pi(n+\omega_1)\tau}+
\sum_{n=k_2}^{+\infty}b_{n+1-k_2}\,\e^{-\imag\pi (n+\omega_2)/\tau}, \qquad\tau\in\Hyp,
\label{eq:expskew2}
\end{equation}
which is of the same type as \eqref{eq:expskew1}. The only difference is
that the coefficients get shifted to the right, which means that
some of the first coefficients get replaced by $0$. This suggests that we
should fix a fundamental domain $\calQ$ of $\R^2/\Z^2$ and consider
$(\omega_1,\Omega_2)\in\calQ$. A natural first choice is the square
\[
\calQ:=\{(\omega_1,\omega_2):\,0\le\omega_1<1,\,\,-1<\omega_2\le0\}
\]
for then
\[
\im \big((n+\omega_1)\tau\big)\ge0,\qquad \tau\in\Hyp,\,\,n\in\Z_{\ge0}, 
\]
and
\[
\im\big((n+\omega_2)/\tau\big)\ge0,\qquad \tau\in\Hyp,\,\,n\in\Z_{>0}, 
\]
both hold, which entails that each individual term in the series
\eqref{eq:expskew1} is uniformly bounded in $\Hyp$.

We would like, if at all possible, to relate the exponentially skewed
series \eqref{eq:expskew1} with the ordinary holomorphic hyperbolic Fourier
series that we have encountered so far.
What saves the day here is the modular lambda function $\lambda(\tau)$, 
given by \eqref{eq:lambdaf1}, with the expansion \eqref{eq:lambdaf2}.
Since $\hat\lambda(1)=16$, we may write \eqref{eq:lambdaf2} in the form
\[
\lambda(\tau)=16\,\e^{\imag\pi\tau}\bigg(1+\frac{1}{16}\sum_{n=1}^{+\infty}
\hat\lambda(n+1)\,\e^{\imag\pi n\tau}\bigg),\qquad\tau\in\Hyp.  
\]
For a parameter $\alpha\in\C$, we may now define the power
\begin{equation}
\lambda(\tau)^\alpha=16^\alpha\,\e^{\imag\pi\alpha\tau}
\bigg(1+\frac{1}{16}\sum_{n=1}^{+\infty}
\hat\lambda(n+1)\,\e^{\imag\pi n\tau}\bigg)^\alpha,\qquad\tau\in\Hyp, 
\label{eq:lambdapower}
\end{equation}
as a holomorphic function in $\Hyp$. The constant is $16^\alpha=\e^{\alpha\log16}$,
with the standard natural logarithm, and the power is given by the binomial
expansion
\begin{equation}
(1+z)^\alpha=1+\binom{\alpha}{1}z+\binom{\alpha}{2}z^2+\cdots,
\label{eq:binomialexp}
\end{equation}
for $|z|<1$. Since the expression raised to the $\alpha$ power in
\eqref{eq:lambdapower} is zero-free in $\Hyp$, the power extends
holomorphically to all of $\Hyp$, and it follows that the function
\begin{equation}
\e^{-\imag\pi\alpha\tau}\lambda(\tau)^\alpha=16^\alpha\,
\bigg(1+\frac{1}{16}\sum_{n=1}^{+\infty}
\hat\lambda(n+1)\,\e^{\imag\pi n\tau}\bigg)^\alpha,\qquad\tau\in\Hyp,
\label{eq:lambdapower2}
\end{equation}
is a well-defined function which may be expressed as a convergent power
series in the variable $\e^{\imag\pi\tau}$. 
Mutatis mutandis, the same applies if we replace $\tau\mapsto-1/\tau$:
\begin{equation}
\e^{\imag\pi\alpha/\tau}\lambda(-1/\tau)^\alpha=16^\alpha\,
\bigg(1+\frac{1}{16}\sum_{n=1}^{+\infty}
\hat\lambda(n+1)\,\e^{-\imag\pi n/\tau}\bigg)^\alpha,\qquad\tau\in\Hyp,
\label{eq:lambdapower3}
\end{equation}
is a well-defined functions which may be expressed as a convergent power
series in the variable $\e^{-\imag\pi/\tau}$.
We need one more property of the modular lambda function. It does not assume
the value $1$, which means that $1-\lambda(\tau)$ is zero-free.
As such, we may define the associated power
\[
(1-\lambda(\tau))^\alpha
\]
by the same binomial expansion \eqref{eq:binomialexp} for large positive values
of $\im\tau$, and extend it holomorphically to all of $\Hyp$, to obtain a
convergent power series in the variable $\e^{\imag\pi\tau}$. 
Likewise, the function
\[
(1-\lambda(-1/\tau))^\alpha
\]
is a well-defined convergent power series in the variable $\e^{-\imag\pi/\tau}$. 
The following proposition explains the algebraic
relationship between exponentially skewed hyperbolic Fourier series and ordinary
hyperbolic Fourier series.

\begin{prop}  
\label{prop:expskewed1}
Suppose $f,g:\Hyp\to\C$ are holomorphic and connected via the relation
\[
g(\tau)=\lambda(\tau)^{-\omega_1}\lambda(-1/\tau)^{-\omega_2}f(\tau),
\]
for some $\beta\in\R$. Then, if $g$ has a hyperbolic Fourier series 
expansion
\[
g(\tau)=\tilde a_0+g_1(\tau)+g_2(-1/\tau),
\]
where 
\[
g_1(\tau)=\sum_{n=1}^{+\infty}\tilde a_n\,\e^{\imag\pi n\tau},\quad
g_2(\tau)=\sum_{n=1}^{+\infty}\tilde b_n\,\e^{\imag\pi n\tau},
\]
it follows that $f$ has an exponentially skewed hyperbolic Fourier expansion
\[
f(\tau)=\e^{\imag\pi\omega_1\tau}(a_0+f_1(\tau))+
\e^{-\imag\pi\omega_2/\tau}f_2(-1/\tau)
\]
where
\[
f_1(\tau)=\sum_{n=1}^{+\infty} a_n\,\e^{\imag\pi n\tau},\quad
f_2(\tau)=\sum_{n=1}^{+\infty} b_n\,\e^{\imag\pi n\tau},
\]
and vice versa. Here, the coefficient sequences $a_n,b_n$ and
$\tilde a_n,\tilde b_n$ meet the growth bound \eqref{eq:subexponential1}
simultaneously. Moreover, the relationship between the two expansions is
given by
\[
a_0+f_1(\tau)=\big(\e^{-\imag\pi\omega_1\tau}\lambda(\tau)^{\omega_1}\big)
(1-\lambda(\tau))^{\omega_2}(\tilde a_0+g_1(\tau))
\]
and
\[
f_2(\tau)=\big(\e^{-\imag\pi\omega_2\tau}\lambda(\tau)^{\omega_2}\big)
(1-\lambda(\tau))^{\omega_1}g_2(\tau).
\]
In particular, the constants are related via $a_0=16^{\omega_1}\tilde a_0$. 
\end{prop}

The proof of Proposition \ref{prop:expskewed1} is supplied in Subsection 
\ref{subsec:expskewed3}.

\begin{rem}
If we write
\[
\Omega_{\omega_1,\omega_2}(\tau)=
\big(\e^{-\imag\pi\omega_1\tau}\lambda(\tau)^{\omega_1}\big)
(1-\lambda(\tau))^{\omega_2}=\sum_{n=0}^{+\infty}
\widehat{\Omega}_{\omega_1,\omega_2}(n)\,\e^{\imag\pi n\tau},
\]
and $\Omega_{\omega_2,\omega_1}$ when the roles of $\omega_1,\omega_2$ are
reversed, the coefficient relations in Proposition \ref{prop:expskewed1} 
may be written in the following form, for $n=1,2,3,\ldots$:
\[
a_n=\sum_{k=0}^{n}\tilde a_k\,\widehat{\Omega}_{\omega_1,\omega_2}(n-k),
\quad  
b_n=\sum_{k=1}^{n}\tilde b_k\,\widehat{\Omega}_{\omega_2,\omega_1}(n-k).  
\]
\end{rem}

\begin{cor}
Fix two parameter $\omega_1,\omega_2\in\R$, and suppose 
$f:\Hyp\to\C$ is holomorphic.
Then there exist complex coefficients $a_n,b_n$ with the growth bound
\begin{equation}
\forall\epsilon>0:\,\,|a_n|,|b_n|=\Ordo(\e^{\epsilon |n|})\quad\text{as}\,\,\,
|n|\to+\infty,
\end{equation}
such that $f$ has the exponentially skewed holomorphic hyperbolic Fourier
series expansion
\[
f(\tau)=a_0\,\e^{\imag\pi\omega_1\tau}+\sum_{n=1}^{+\infty}
\big(a_n \,\e^{\imag\pi (n+\omega_1)\tau}+b_n\,
\e^{-\imag\pi (n+\omega_2)/\tau}\big),\qquad\tau\in\Hyp.
\]
\end{cor}

\begin{proof}
The assertion is an immediate consequence of the holomorphic version of 
Theorem \ref{thm:HFS-general.1.1} combined with the algebraic 
correspondence of Proposition \ref{prop:expskewed1}.
\end{proof}

\begin{rem}
Naturally, we would like to know when the coefficients $a_n,b_n$ can be
chosen to grow at most polynomially, for a given given holomorphic function 
$f$ with (cf. Theorem \ref{thm:hfs2}) the growth bound
\begin{equation}
\exists k:\,\,\,|f(\tau)|=\Ordo\bigg(\frac{1+|\tau|^2}{\im\tau}\bigg)^k
\label{eq:polgrowth0.11}
\end{equation}
where the implicit constant is uniform in $\Hyp$. After all, this growth bound
is associated with having boundary values in the sense of distribution theory.
This will require us to
delve a little deeper than the above Proposition \ref{prop:expskewed1}.
An immediate inspection suggests that it is natural to impose
the conditions $\omega_1>0$, $\omega_2>-1$ in this context, for otherwise,
even if the functions $f_1,f_2$ meet \eqref{eq:polgrowth0.11} individually,
the combined function
\[
f(\tau)=\e^{\imag\pi\omega_1\tau}(a_0+f_1(\tau))+
\e^{-\imag\pi\omega_2/\tau}f_2(-1/\tau)
\]
need not. 
\end{rem}

\section{Applications to the Klein-Gordon equation}
\label{sec:Klein-Gordon}

\subsection{Oscillatory processes and the Klein-Gordon equation}
Little is known regarding the interpolation of oscillatory processes.
We can visualize a typical problem as the study of pressure fluctuations inside
a natural gas storage facility in terms of a finite number of strategically 
located pressure gauges.
Unfortunately, only some numerical methods based on considering solutions 
with extreme value of entropy are known (see, e.g., \cite{lio}, \cite{ger}).
We will consider a specific configuration of such gauges for solutions to
a particular linear partial differential equation which models oscillatory
behavior.
Since oscillatory processes are associated with hyperbolic partial
differential equations,
it makes sense to begin with the most basic of such equations, which are
the wave equation and the associated eigenvalue problem.
The eigenvalue problem for the wave operator (for eigenvalues $\neq0$)
is called the \emph{Klein-Gordon equation}. In one spatial and one
temporal dimension, the Klein-Gordon equation may be expressed in
suitable alternative coordinates as
\begin{equation}
u''_{xy} + u = 0,\qquad (x,y)\in\R^2.
\label{eq:KleinGordon1.1}
\end{equation}
If we try for pure complex exponential solutions
\begin{equation}
u(x,y)=\e^{\imag\lambda x+\imag\mu y},
\label{eq:puresol0}
\end{equation}
where $\lambda,\mu\in\C$, then 
\[
u''_{xy}+u=(1-\lambda\mu)\,\e^{\imag\lambda x+\imag\mu y}=0
\]
holds if and only if $\lambda\mu=1$. The complex exponential $u$ given by  
\eqref{eq:puresol0} is bounded in $\R^ 2$ if and only if
$(\lambda,\mu)\in\R^2$, so that the
bounded pure complex exponential solutions are parametrized by
$(\lambda,\mu)\in\R^ 2$ with $\lambda\mu=1$. Replacing $\lambda$ by
$t$ and $\mu$ by $1/t$, such solutions are given by
\begin{equation}
u(x,y)=\e^{\imag t x+\imag y/t},
\label{eq:puresol1}
\end{equation}
for $t\in\R_{\ne0}=\R\setminus\{0\}$. Forming linear combinations of such
pure solutions, we have again a bounded smooth solution.
An $L^1$-\emph{mixed bounded exponential solution} is a generalized 
linear combination of the form
\begin{equation}
\label{f0int}
u(x,y)=U[{\varphi}] (x, y) := \int\nolimits_{\R}
\e^{\imag  x t + \imag y / t} \varphi (t) \diff t ,
\end{equation}
where $\varphi \in L^{1}(\R)$.
Such $L^1$-mixed bounded exponential solution are not so exotic after all,
as can be seen from the following argument. First, we note that according to
\cite{hed}, we can think of \eqref{f0int} as saying that $u=U[\varphi]$ is
the Fourier transform of a finite complex-valued Borel measure on $\R^2$
supported  on a certain hyperbola
(corresponding to $\{(\lambda,\mu)\in\R^2:\,\lambda\mu=1\}$ in the above
context) which in addition is absolutely continuous with respect to arc
length measure. 
On the other hand, if $u\in L^\infty(\R^2)$ is given,
its inverse Fourier transform, thought of as a distribution, is a
pseudomeasure on $\R^2$.  
Moreover, if $u$ solves the Klein-Gordon equation \eqref{eq:KleinGordon1.1}
in the sense of distribution theory, then that pseudomeasure would be supported
on the same hyperbola in $\R^2$. Expressed differently, $u$ would be the Fourier
transform of a pseudomeasure whose support is the given hyperbola.
So the additional requirement in \eqref{f0int} is that the pseudomeasure
be an actual finite Borel measure which is absolutely continuous with respect
to arc length measure on the hyperbola. 
A further remark worth mentioning is that on a given compact subset, such 
$L^1$-mixed bounded exponential solutions can approximate any continuous 
solution of the Klein-Gordon equation.

\subsection{The discretized Goursat problem for the
Klein-Gordon equation}
The lines $x=0$ and $y=0$ are characteristics for the equation $u_{xy}+u=0$,
and the Goursat problem would ask for a solution with given data e.g.
along the semi-axes $\{(x,y)\in\R^2:x=0,\,y\le0\}$ and
$\{(x,y)\in\R^2:x\ge0,\,y=0\}$. This problem 
is concerned with the space-like quarter-plane $\{(x,y)\in\R^2:x\ge0,y\le0\}$,
and it was investigated by Riemann. Riemann's solution is often unbounded
in the plane. If we need a bounded solution, some kind of compatibility
property needs to be met for the data along the two characteristics.
For the $L^1$-mixed bounded exponential solutions, we can be more precise.
In terms of notation, we write
\[
\mathrm{J}_{1,0}(x,y):=\sum_{k=0}^{+\infty}\frac{(-1)^k}{k!(k+1)!}x^{k+1}y^k
\]
for a two-variable version of the familiar Bessel function. The following
result was mentioned in \cite{bhm}.

\begin{prop}
If $u(x,y)=U[\varphi](x,y)$ is given by \eqref{f0int} for some
$\varphi\in L^1(\R)$, then
  \begin{equation}
u(0,y)=u(0,0)-\int_0^{+\infty}\mathrm{J}_{1,0}(-y,t)u(t,0)\,\diff t,\qquad y\le0.
\label{eq:J1}
\end{equation}
\end{prop}

In terms of the Goursat problem, this means that the boundary data on
the semi-axis $\{(x,y)\in\R^2:x\ge0,y=0\}$ necessarily induces boundary
data on other semi-axis $\{(x,y):x=0,y\le0\}$ as well
(and, in fact, vice versa).
But we would prefer to work with independent data. One way to achieve this
is to thin out the given data on the two semi-axes by discretizing the axes.

Given that we look for a bounded solution of the form \eqref{f0int}, it makes
sense to reduce the given Goursat boundary data along the two characteristics
to the point that the data is no longer overdetermined. The question is how
to thin out the data.
If we consider the two-sided
\emph{discretized Goursat problem of critical density},
which instead prescribes the values along the discrete subset of the
characteristics consisting of  equidistant points $(\pi m, 0)$ and
$(0, \pi n)$ where $m,n \in \Z$, it turns out that we have
uniqueness (we write that it is two-sided because we consider the full
lattice-cross, which is the union of the discretized boundaries of the two 
spacelike quarter-planes  $\{(x,y)\in\R^2:x\ge0,y\le0\}$ and
$\{(x,y)\in\R^2:x\le0,y\ge0\}$).
In other words, for $L^1$-mixed bounded exponential solutions
$u=U_\varphi$, the values $u(\pi m, 0)$ and $u(0, \pi n)$ determine $u$ uniquely,
provided that $m,n$ range over the integers $\Z$.
This is the main result of \cite{hed}, which may be formulated in the
following manner.

\begin{thm}
 \label{intth0}
{\rm(Hedenmalm, Montes-Rodr\'\i{}guez)}
Suppose $ \varphi  \in L^{1}(\R)$ and that
\begin{equation}
     \int\limits_{\R}  {\rm{e}}^{{{{\rm{i}} \pi n t}}}
    \varphi (t) {\rm{d}} t  =\int\limits_{\R}  {\rm{e}}^{{{{-\rm{i}}
     \pi n / t}}} \varphi (t) {\rm{d}} t=0 \, , \quad  n\in \Z.
\end{equation}
\noindent Then $ \varphi  = 0$.
\end{thm}

The corresponding result which applies to the space-like quarter-planes reads
as follows (see \cite{hed1}, \cite{hed2}).

\begin{thm}
\label{intth0.01}
{\rm(Hedenmalm, Montes-Rodr\'\i{}guez)}
Suppose $ \varphi\in L^{1}(\R)$ and that
\begin{equation}
     \int\limits_{\R}  {\rm{e}}^{{{{\rm{i}} \pi n t}}}
    \varphi (t) {\rm{d}} t  =\int\limits_{\R}  {\rm{e}}^{{{{-\rm{i}}
     \pi n / t}}} \varphi (t) {\rm{d}} t=0 \, , \quad  n\in \Z_{\ge0}.
\end{equation}
\noindent Then $\varphi\in H^1_+(\R)$.
\end{thm}

Here, $H^1_+(\R)$ stands for the Hardy space of functions in $L^1(\R)$ whose
Poisson harmonic extensions (cf. equation \eqref{eq:Poisson0})
to the upper half-plane $\Hyp$ are holomorphic. We will use the notation
$H^1_-(\R)$ for for the analogous Hardy space where the harmonic
extension is instead conjugate-holomorphic.

The method of proof of Theorems \ref{intth0} and \ref{intth0.01} is careful
analysis of the transfer operator associated with the Gauss-type map
$t\mapsto-1/t$ modulo $1$ on the interval $[-1,1]$. 
In terms of the $L^1$-mixed bounded exponential solutions to the Klein-Gordon 
equation, these results entail the following.

\begin{cor}
\label{cor:uniq1.01}
Let $u=U[\varphi]$ be of the form \eqref{f0int} for some $\varphi\in L^1(\R)$.

\noindent{\rm(a)}
If $u(\pi n,0)=u(0,-\pi n)=0$ for each $n\in\Z_{\ge0}$, then
$\varphi\in H^1_+(\R)$, and hence $u(x,y)=0$ in the quarter-plane 
$\{(x,y)\in\R^2:\,x\ge0,y\le0\}$. 

\noindent{\rm(b)}
If instead $u(-\pi n,0)=u(0,\pi n)=0$ for each $n\in\Z_{\ge0}$, then
$\varphi\in H^1_-(\R)$, and hence $u(x,y)=0$ in the quarter-plane 
$\{(x,y)\in\R^2:\,x\le0,y\ge0\}$. 

\noindent{\rm(c)}
If $u(\pi n,0)=u(0,-\pi n)=0$ for each $n\in\Z$, then
$\varphi=0$, and hence $u(x,y)=0$ on all of $\R^2$.
\end{cor}

\begin{rem}
In Corollary \ref{cor:uniq1.01}, parts (a) and (b) are equivalent, in the sense
that (b) follows from (a), and (a) follows from (b).
Moreover, (c) is immediate from (a) and (b), since the assumptions imply that
$\varphi\in H^1_+(\R)\cap H^1_-(\R)=\{0\}$. 
\end{rem}  

Since according to Corollary \ref{cor:uniq1.01}(c), a solution $u=U[\varphi]$
(with $\varphi\in L^1(\R)$) to the
Klein-Gordon equation is uniquely determined by its values along the
\emph{lattice-cross}
\[
\big\{(\pi m,\pi n):\, m,n\in\Z,\,\,\, mn=0\big\}.
\]
it becomes natural to ask whether we can freely prescribe the values there.
Pleasantly, this is answered in the affirmative by Theorem
\ref{thm:interpolation0.01} below,
which in turn builds on the following result. Note that 
the functions $A_0$ and $A_n,B_n$ are all in $L^1(\R)$ for $n\in\Z_{\ne0}$. 
After all, in view of Corollary \ref{cor:hfs1.1}, they are test functions in
$C^\infty_1(\R\cup\{\infty\})$.

\begin{thm}
\label{thm:PBeurling}  
For $n\in\Z$, we put
\[
u_{(n,0)}(x,y):=U[A_n](x,y)=\int_\R \e^{\imag xt+\imag y/t}A_n(t)\,\diff t,
\qquad(x,y)\in\R^2,  
\]
and declare, in addition, that, for $n\in\Z_{\ne0}$,
\[
u_{(0,n)}(x,y):=U[B_{-n}](x,y)=U[A_n](y,x)=
\int_\R \e^{\imag yt+\imag x/t}A_n(t)\,\diff t,
\qquad(x,y)\in\R^2.  
\]
Then if we use Kronecker delta notation, we have
\[
u_{(n,0)}(\pi m,0)=\delta_{m,n},\quad u_{(n,0)}(0,\pi m)=\delta_{0,m}\delta_{m,0},
\]
for each $m\in\Z$. Likewise, we find that for $n\in\Z_{\ne0}$ and $m\in\Z$,
\[
u_{(0,n)}(\pi m,0)=0,\quad u_{(0,n)}(0,\pi m)=\delta_{m,n}.
\]  
\end{thm}

The proof of Theorem \ref{thm:PBeurling} is supplied in Subsection
\ref{subsec:HFS_Cauchy}.

\begin{rem}
The functions $u_{(n,0)}$ and $u_{(0,n)}$ are the analogous functions in the
context of the Klein-Gordon equation \eqref{eq:KleinGordon1.1} to the
classical cardinal sine functions for bandlimited signals, modelled by the
Paley-Wiener space of entire functions (which we may think of as the Fourier
transform applied to the Hilbert space $L^2([-1,1])$).
Although defined in terms of the complex-valued functions $A_n$ and $B_n$,
the functions $u_{(n,0)}(x,y)$ and $u_{(0,n)}(x,y)$ are necessarily real-valued.
This is related with the fact that the Klein-Gordon equation
\eqref{eq:KleinGordon1.1} which these functions solve has real coefficients.
Another consequence of Theorem \ref{thm:PBeurling} and Corollary 
\ref{cor:uniq1.01} is that $A_n,B_n\in H^1_-(\R)$ for $n>0$ while $A_n,B_n\in 
H^1_+(\R)$ for $n<0$.
\end{rem}

\begin{figure}
\centering
\begin{tikzpicture}[scale=3]
\fill[lightgray] (0, 0.1) rectangle (1.4, 1.5);
            \draw[-latex] (-1.5,1.5) -- (1.6,1.5)node[below]{$x$};
\draw[-latex] (0,0) -- (0,1.7)node[right]{$y$};
\path(0,1.5) --
node[below, pos=1]{$(2\pi,0)$} (1,1.5)
    (0,1.5)node[below]{$(0,0)$};
\draw(0,1.5)--(-.5,1.5)node[below]{}(.5,1.5)--(.5,1.5)
node[below]{$(\pi,0)$};
\draw(0,1.0)--(0,1.0)node[below]{}(0,1.0)--(0,1.0)
node[below]{$(0,-\pi)$};
\draw(0,0.5)--(0,0.5)node[below]{}(0,0.5)--(0,0.5)
node[below]{$(0,-2\pi)$};
\draw(1,0.0)--(1,0.0)node[below]{}(0.6,0.7)--(0.6,0.7)
node[right]{}
;
\filldraw [black] (0,1.5) circle (0.5pt);
\filldraw [black] (0.5,1.5) circle (0.5pt);
\filldraw [black] (1.0,1.5) circle (0.5pt);
\filldraw [black] (0.0,1.0) circle (0.5pt);
\filldraw [black] (0.0,0.5) circle (0.5pt);
\end{tikzpicture}
\caption{The discretized lattice-cross for a spacelike quarter-plane.}
\label{fig:lattice-cross}
\end{figure}
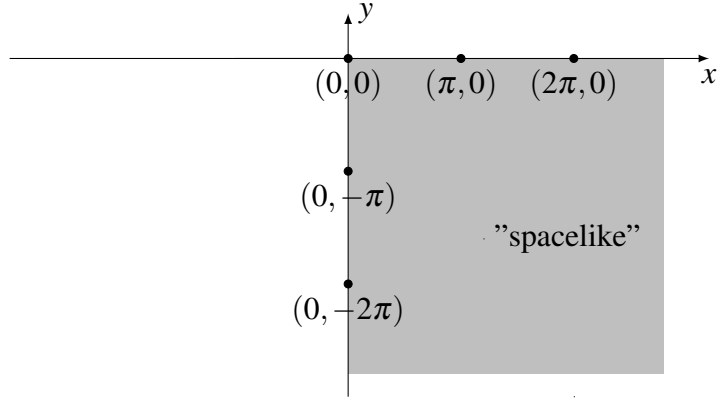

\begin{thm}
\label{thm:interpolation0.01}
Suppose the sequence of complex numbers $\alpha_0$ and $\alpha_n,\beta_n$
is given for $n\in\Z_{\ne0}$, with
\[
|\alpha_0|+\sum_{n\in\Z_{\ne0}}\big(|\alpha_n|+|\beta_n|\big)(\log^3(|n|+1))
<+\infty.
\]
Then if $\varphi$ is given by
\[
\varphi(t)=\alpha_0A_0(t)
+\sum_{n\in\Z_{\ne0}}\big(\alpha_nA_n(t)+\beta_n B_{-n}(t)\big),
\qquad t\in\R,  
\]
we have that $\varphi\in L^1(\R)$, and, in addition, the function
\[
u(x,y)=U[\varphi](x,y)=\alpha_0\, u_{(0,0)}(x,y)+
\sum_{n\in\Z_{\ne0}}\big(\alpha_n\, u_{(n,0)}(x,y)+\beta_n\, u_{(0,n)}(x,y)\big)  
\]
solves the interpolation problem
\[
u(\pi m,0)=\alpha_m,\quad u(0,\pi n)=\beta_n,
\]
for all $m\in\Z$ and $n\in\Z_{\ne0}$. 
\end{thm}

The proof of Theorem \ref{thm:interpolation0.01} is supplied in Subsection
\ref{subsec:HFS_Cauchy}.

\begin{rem}
According to part (a) or (b) of Corollary \ref{cor:uniq1.01}, we may consider
the interpolation problem along the discretized boundary of a spacelike
quarterplane, and obtain a unique solution on the corresponding quarterplane.
This interpolation problem may be solved in a completely analogous fashion
to that of Theorem \ref{thm:interpolation0.01}. We omit the necessary details.
\end{rem}  

\begin{figure}
\includegraphics[width=0.5\linewidth]{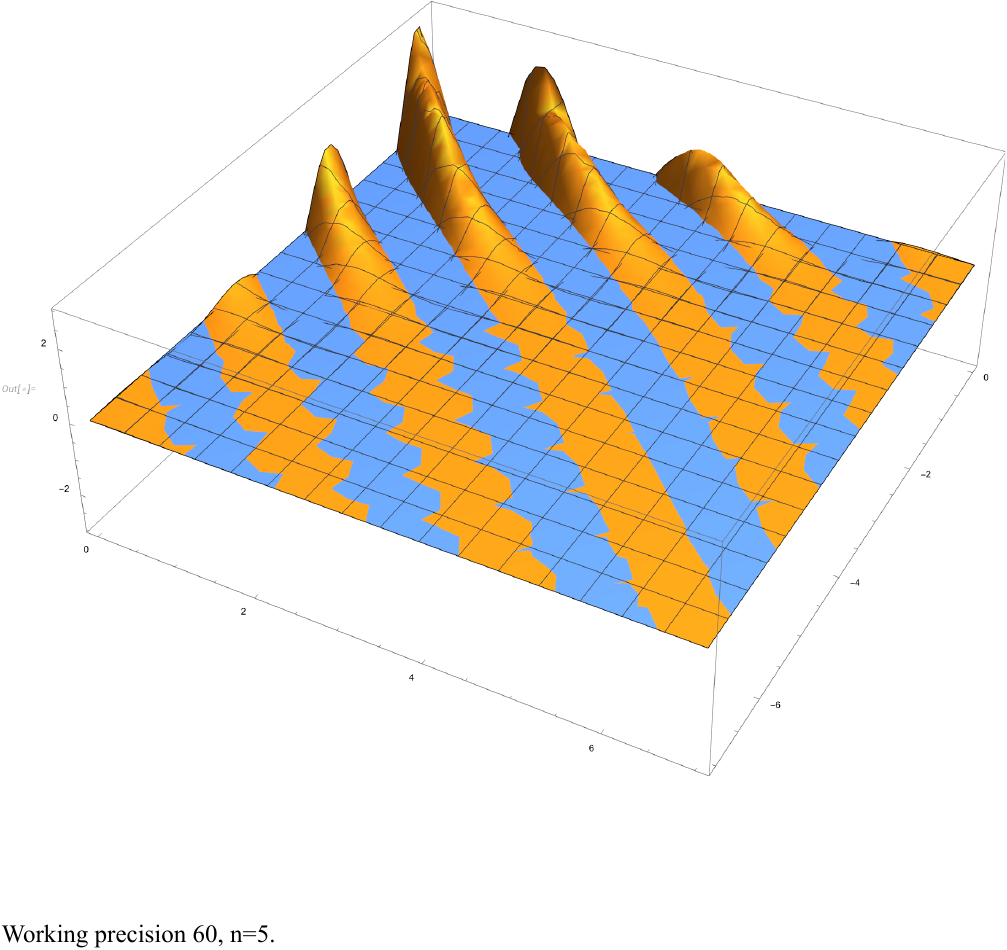}
\caption{Graph of the positive part of the function $u_{(n,0)}(x,y)$ for $n=5$ 
  in the quarter-plane where $x\ge0$ and $y\le0$. The computation is based
on Corollary \ref{cor:PBeurlingformula} below.}
\label{fig:GraphXY}
\end{figure}

\subsection{Alternative discretization variant: normal
derivatives}

The vanishing of two-sided Goursat boundary data along the two
axes amounts to having
\begin{equation}
\label{eq:Goursatbrydata1}
u(x,0)=u(0,y)=0,\qquad (x,y)\in\R^2.
\end{equation}
What happens if we replace one of the conditions, say the condition along the
$y$-axis, by a condition on the derivative or even a higher derivative?
For instance, if we consider the vanishing boundary data
\begin{equation}
\label{eq:Goursatbrydata2}  
u(x,0)=(\partial_x^k u)(0,y)=0, \qquad (x,y)\in\R^2,
\end{equation}
it follows immediately from the Klein-Gordon equation that, at least if $u$
is smooth, we have that
\[
0=\partial_y^k\big((\partial_x^k u)(0,y)\big)
=(\partial_y^k\partial_x^k u)(0,y)=(-1)^k u(0,y),\qquad y\in\R.
\]
In other words, the condition \eqref{eq:Goursatbrydata2} appears to be
slightly stronger than \eqref{eq:Goursatbrydata1}. 
As written, the two conditions apparently have inequivalent
discretizations. We are already familiar
with the discretization of \eqref{eq:Goursatbrydata1}. The corresponding
discretization of \eqref{eq:Goursatbrydata2} is
\begin{equation}
\label{eq:Goursatbrydata2.disk}  
u(\pi m,0)=(\partial_x^k u)(0,\pi n)=0, \qquad m,n\in\Z.
\end{equation}
If we as before take as our point of departure the formula \eqref{f0int} for
$u=U[\varphi]$, we see that
\[
\partial_x^k u(0,y)=\imag^k\int_\R \e^{\imag y/t}\varphi(t)\,t^k\,\diff t,
\]
provided that the integral makes sense. The condition on $\varphi$ that
\begin{equation}
\int_\R |\varphi(t)|(1+|t|^{k})\diff t<+\infty
\label{eq:extraintegrability}
\end{equation}
will guarantee that $u(x,y)$ and $\partial_x^k u(x,y)$ are bounded and 
continuous functions on $\R^2$.

\begin{prop}
Suppose $\varphi\in L^1(\R)$ has the extra integrability
\eqref{eq:extraintegrability}
for some $k\in\Z_{\ge0}$. Then if $u=U[\varphi]$ is of the form 
\eqref{f0int}, and 
\[
\int_\R \e^{\imag\pi m t}\varphi(t)\,\diff t=\int_\R\e^{-\imag \pi n/t}
\varphi(t)\,t^k\,\diff t=0,\qquad m,n\in\Z,
\]
holds, it follows that $\varphi=0$.
\label{prop:extraintegrability}
\end{prop}

The proof of Proposition \ref{prop:extraintegrability} is supplied in
Subsection \ref{subsec:normalderivatives}.

\begin{cor}
Suppose $\varphi\in L^1(\R)$ has the extra integrability
\eqref{eq:extraintegrability}
for some $k\in\Z_{\ge0}$, and let $u=U[\varphi]$ be given by 
\eqref{f0int}. Suppose in addition that
\[
u(\pi m,0)=\partial_x^k u(0,\pi n)=0, \qquad m,n\in\Z,
\]
holds. It then follows that $\varphi=0$  and hence that $u=0$.
\end{cor}

This connects with the notion of
power skewed hyperbolic Fourier series. The interpolation problem
\begin{equation}
u(\pi m,0)=\alpha_m,\quad \partial_y^k u(0,\pi n)=\beta_n,
\label{eq:skewedinterpolation1.01}
\end{equation}
for $m,n\in\Z$ and $n\in\Z_{\ne0}$ amounts to looking for a suitable $\varphi$
such that
\[
u(\pi m,0)=\int_\R \e^{\imag\pi m t}\varphi(t)\diff t=\alpha_n,\qquad m\in\Z,
\]
while at the same time
\[
(\partial^k_x u)(0,\pi n)=\imag^k\int_\R
\e^{\imag\pi n/ t}\varphi(t)\,t^k\,\diff t=\beta_n,\qquad n\in\Z.
\]
It is of course entirely possible that the coefficients $\alpha_n,\beta_n$
in the interpolation problem cannot be chosen completely freely, but it
is very likely to be the case up to a finite-dimensional defect. 
Following the same scheme as in Theorem \ref{thm:interpolation0.01},
which is based on Theorem \ref{thm:PBeurling}, we should try to find,
if at all possible, a biorthogonal system to the system of functions
\[
\e^{\imag\pi mt},\quad t^{k}\e^{-\imag\pi n/t},
\]
for $m,n\in\Z$. 
This automatically hints at the study of power skewed hyperbolic
Fourier series, as defined in Subsection \ref{subsec:powerskewed}
with $\beta=-k$. 

\subsection{Alternative discretization: translation
 along the axes}

An alternative app\-roach to the discretization of the two axes is that
we consider, for $u=U[\varphi]$, a discretization of the two-sided Goursat
boundary data \eqref{eq:Goursatbrydata1} with equal density of the form
\begin{equation}
\label{eq:Goursatbrydata3.disk}  
u(\pi (m+\omega_1),0)=u(0,\pi (n+\omega_2))=0, \qquad m,n\in\Z,
\end{equation}
where $\omega_1,\omega_2\in\R$ are frequency shifts. 

\begin{prop}
For $\varphi\in L^1(\R)$, let $u=U[\varphi]$ be given by \eqref{f0int}.
If \eqref{eq:Goursatbrydata3.disk} holds for some $\omega_1,\omega_2\in\R$,   
then $\varphi=0$ and hence $u=0$. 
\label{prop:uniq1.05}
\end{prop}

The proof of Proposition \ref{prop:uniq1.05} is supplied in Subsection
\ref{subsec:translatedlatticecross}.

The  interpolation problem associated with \eqref{eq:Goursatbrydata3.disk}
reads
\begin{equation}
u(\pi (m+\omega_1),0)=\alpha_m,\quad u(0,\pi (n+\omega_2))=\beta_n,
\label{eq:skewedinterpolation1.02}
\end{equation}
and in view of Proposition \ref{prop:uniq1.05}, the solution $u=U[\varphi]$
with $\varphi\in L^1(\R)$, is necessarily unique, provided that it exists.
Since
\[
u(\pi(m+\omega_1),0)=\int_\R \e^{\imag\pi(m+\omega_1)t}\varphi(t)\,\diff t
\]
and
\[
u(0,\pi(n+\omega_2))=\int_\R \e^{-\imag\pi(n+\omega_2)/t}\varphi(t)\,\diff t
\]
we see the relevance of exponentially skewed hyperbolic Fourier series
(see Subsection \ref{subsec:exponentiallyskewed})
to the study of the interpolation problem.


\section{Preparatory work}

\subsection{The splitting of a harmonic function}
\label{subsec:splitting}
It is well-known that a harmonic function in a simply connected domain
is the sum of a holomorphic function and a conjugate-holomorphic function.
This decomposition is unique up to an additive constant. 
Let us do this for a harmonic function $h:\Hyp\to\C$ with the growth bound
\begin{equation}
\label{eq:growthboundh1}
|h(\tau)|\le C\,\exp(\beta\pi\, \Mfun(\tau)),\qquad \tau\in\Hyp,
\end{equation}
for a positive constant $C$, where we recall that $\Mfun(\tau)$ is
given by \eqref{eq:defM}. The parameter $\beta$ is assumed positive.
The gradient of $h$ can be estimated using the Poisson kernel representation
\[
h(\tau+\epsilon\eta)=\int_{-\pi}^{\pi}
\frac{1-|\eta|^2}{|\eta-\e^{\imag t}|^2}\,h(\tau+\epsilon\,\e^{\imag t})
\,\frac{\diff t}{2\pi},
\]
valid for $\tau\in\Hyp$ and $\epsilon$ with $0<\epsilon<\im\tau$.  
We obtain that
\[
\partial_\tau h(\tau)=\frac{1}{\epsilon}\int_{-\pi}^{\pi}\e^{-\imag t}
h(\tau+\epsilon\,\e^{\imag t})\frac{\diff t}{2\pi},\quad
\bar\partial_\tau h(\tau)=\frac{1}{\epsilon}\int_{-\pi}^{\pi}\e^{\imag t}
h(\tau+\epsilon\,\e^{\imag t})\frac{\diff t}{2\pi},
\]
and since
\begin{equation}
\Mfun(\tau+\epsilon\,\e^{\imag t})\le
\frac{\max\{1,|\tau|^2-\epsilon^2\}}{\im\tau-\epsilon}
\le\frac{\max\{1,|\tau|^2\}}{\im\tau-\epsilon},
\qquad t\in\R,
\end{equation}
it is an immediate consequence of \eqref{eq:growthboundh1} that
\begin{equation}
\max\big\{|\partial_\tau h(\tau)|,|\bar\partial_\tau h(\tau)|\big\}
\le \frac{C}{\epsilon}\,
\exp\bigg(\beta\pi\frac{\max\{1,|\tau|^2\}}{\im\tau-\epsilon}
\bigg)
\end{equation}
holds. By suitably optimizing over $\epsilon$ with $0<\epsilon<\im\,\tau$,
it can be shown that
\begin{equation}
\label{eq:gradh1}  
\max\big\{|\partial_\tau h(\tau)|,|\bar\partial_\tau h(\tau)|\big\}
\le \frac{C\,\e}{(\im\tau)^2}\big(\beta\pi\max\{1,|\tau|^2\}+\im\tau\big)\,
\exp\big(\beta\Mfun(\tau)\big)
\end{equation}
holds. The function $\tau\mapsto \partial_\tau h(\tau)$ is holomorphic, whereas
$\tau\mapsto\bar\partial_\tau h(\tau)$ is conjugate-holomorphic.
If we put
\[
h_+(\tau):=\int_{\imag}^\tau \partial_\eta h(\eta)\,\diff\eta,
\quad
h_-(\tau):=\int_{\imag}^\tau \bar\partial_\eta h(\eta)\,\diff\bar\eta,
\]  
we obtain a holomorphic function $h_+$ and a conjugate-holomorphic function
$h_-$ with $h_+(\imag)=h_-(\imag)=0$ such that
\[
h(\tau)=h_+(\tau)+h_-(\tau)+h(\imag).
\]
Using the estimate \eqref{eq:gradh1}, we find that
\begin{multline}
\label{eq:gradh+1}
|h_+(\tau)|\le \int_{\gamma(\imag,\tau)}|\partial_\eta h(\eta)|\,|\diff\eta|
\\
\le C\,\e\int_{\gamma(\imag,\tau)}
\frac{1}{(\im\eta)^2}\big(\beta\pi\max\{1,|\eta|^2\}+\im\,\eta\big)\,
\exp\big(\beta\pi\Mfun(\eta)\big)\,|\diff\eta|
\end{multline}
where $\gamma(\imag,\tau)$ is a suitable path in $\Hyp$ from
$\imag$ to $\tau$.
Moreover, if we apply the change-of-variables $\eta=1/\bar\xi$ (reflection in
the unit circle) on the right-hand side integral in \eqref{eq:gradh+1}, 
it follows, after some simplification, that
\begin{equation}
\label{eq:gradh+2}
|h_+(\tau)|
\le\int_{\gamma'(\imag,1/\bar\tau)}
\frac{C\,\e}{(\im\xi)^2}\big(\beta\pi\max\{1,|\xi|^2\}+\im\,\xi\big)\,
\exp\big(\beta\pi\Mfun(\xi)\big)\,|\diff\xi|
\end{equation}
where $\gamma'(\imag,1/\bar\tau)$ is any path in $\Hyp$ which connects $\imag$
to $1/\bar\tau$. This estimate is precisely the same as in \eqref{eq:gradh+1},
only the endpoint $\tau\in\Hyp$ is replaced by $1/\bar\tau\in\Hyp$. 
This means that it is enough to estimate $h_+(\tau)$ using \eqref{eq:gradh+1}
for $|\tau|\le1$, as the remaining case $|\tau|>1$ can be handled using
\eqref{eq:gradh+2}, since then the endpoint has $|1/\bar\tau|<1$. 
So, we consider $\tau\in\Hyp$ with $|\tau|\le1$, and by symmetry, we may assume
that $\re\,\tau >0$ as well. We choose the path $\gamma(\imag,\tau)$
which passes along the unit circle to the point with the same real part
as $\tau$, and then straight down to $\tau$ along the line segment parallel
to the imaginary axis. As we implement suitable parametrizations,
the integral on the right-hand side of \eqref{eq:gradh+2}
obtains the estimate
\begin{multline}
\int_{\gamma(\imag,\tau)}
\frac{1}{(\im\eta)^2}\big(\beta\pi\,\max\{1,|\eta|^2\}+\im\,\eta\big)\,
\exp\big(\beta\pi\Mfun(\eta)\big)\,|\diff\eta|
\\
\le
\int_{\alpha}^{\pi/2}\bigg(\frac{\beta\pi}{\sin^2t}+\frac{1}{\sin t}\bigg)
\exp\bigg(\frac{\beta\pi}{\sin t}\bigg)\diff t+
\int_{\im\tau}^{\sqrt{1-(\re\tau)^2}}
\bigg(\frac{\beta\pi}{t^2}+\frac{1}{t}\bigg)
\exp\bigg(\frac{\beta\pi}{t}\bigg)\,\diff t  
\end{multline}
where $\cos\alpha=\re\tau$. The first integral on the right-hand side
is handled by by splitting the interval, from $\alpha$ to $\pi/4$, and then
from $\pi/4$ to $\pi/2$. The second part is estimated using a trivial
estimate, while the first part is handled using integration by
parts. The second integral is analogously analyzed using integration by parts.
Omitting the necessary details, we find that
\begin{multline}
\int_{\gamma(\imag,\tau)}
\frac{1}{(\im\eta)^2}\big(\beta\pi\max\{1,|\eta|^2\}+\im\,\eta\big)\,
\exp\big(\beta\pi\pi\,\Mfun(\eta)\big)\,|\diff\eta|
\\
\le
2(1+\beta\pi)\exp(\beta\pi\sqrt{2})
+\bigg(\sqrt{2}+\frac{3}{\beta\pi}\bigg)\exp\bigg(\frac{\beta\pi}{\im\,\tau}
\bigg).  
\end{multline}
We now conclude from \eqref{eq:gradh+1} and \eqref{eq:gradh+2} that 
\begin{equation}
|h_+(\tau)|
\le
2C\e\,(1+\beta)\exp(\beta\pi\sqrt{2})
+C\e\,\bigg(\sqrt{2}+\frac{3}{\beta\pi}\bigg)
\exp\big(\beta\pi\,\Mfun(\tau)\big).  
\end{equation}
The same kind of estimates applied to the conjugate-holomorphic function $h_-$
as well. We formulate this splitting less precisely in the form of a
proposition.

\begin{prop}
\label{prop:beta1}
Fix a $\beta$ with $0<\beta<+\infty$. Suppose $h:\Hyp\to\C$ is harmonic with
\[
|h(\tau)|=\Ordo\big(\exp\big(\beta\pi\,\Mfun(\tau)\big)\big)
\]  
uniformly in $\Hyp$. Then there exist unique functions $h_+,h_-:\Hyp\to\C$
such that $h_+$ is holomorphic, $h_-$ is conjugate-holomorphic,
with $h_+(\imag)=h_-(\imag)=0$, and the splitting
\[
h(\tau)=h_+(\tau)+h_-(\tau)+h(\imag),\qquad\tau\in\Hyp,
\]
holds. These functions have the same growth estimate:
\[
|h_+(\tau)|,|h_-(\tau)|=\Ordo\big(\exp\big(\beta\pi\,\Mfun(\tau)\big)\big),
\]
uniformly in $\Hyp$. 
\end{prop}

The instance of $\beta=0$ in Proposition \ref{prop:beta1} is a little
different.

\begin{prop}
\label{prop:beta2}
Suppose $h:\Hyp\to\C$ is harmonic and bounded, i.e., 
\[
|h(\tau)|=\Ordo(1)
\]  
holds uniformly in $\Hyp$. Then there exist unique functions 
$h_+,h_-:\Hyp\to\C$
such that $h_+$ is holomorphic, $h_-$ is conjugate-holomorphic,
with $h_+(\imag)=h_-(\imag)=0$, and the splitting
\[
h(\tau)=h_+(\tau)+h_-(\tau)+h(\imag),\qquad\tau\in\Hyp,
\]
holds. These functions enjoy the logarithmic growth estimate:
\[
|h_+(\tau)|,|h_-(\tau)|=\Ordo\bigg(\log\frac{1+|\tau|^2}{\im\tau}\bigg),
\]
uniformly in $\Hyp$. 
\end{prop}
 
This statement is well-known, and indeed much more is known, since for bounded
$h$ the function $h_+$ is in BMOA of the upper half-plane (the bounded
mean oscillation space on $\R$ whose Poisson extensions are 
holomorphic).  

Another instance of a growth condition appeared earlier in the context of
harmonic extension of distributions on the extended real line
$\R\cup\{\infty\}$.
A direct proof can be constructed along the lines of the proof of 
Proposition \ref{prop:beta1}.
  
\begin{prop}
\label{prop:beta3}
Fix a real $\alpha$ with $0<\alpha<+\infty$.
Suppose $h:\Hyp\to\C$ is harmonic and that 
\[
|h(\tau)|=\Ordo\bigg(\frac{1+|\tau|^2}{\im\,\tau}\bigg)^\alpha
\]  
holds uniformly in $\Hyp$, Then there exist unique functions 
$h_+,h_-:\Hyp\to\C$
such that $h_+$ is holomorphic, $h_-$ is conjugate-holomorphic,
with $h_+(\imag)=h_-(\imag)=0$, and the splitting
\[
h(\tau)=h_+(\tau)+h_-(\tau)+h(\imag),\qquad\tau\in\Hyp,
\]
holds. These functions enjoy the same growth estimate:
\[
|h_+(\tau)|,|h_-(\tau)|=\Ordo\bigg(\frac{1+|\tau|^2}{\im\tau}\bigg)^\alpha,
\]
uniformly in $\Hyp$. 
\end{prop} 
  
\subsection{Coefficient growth and function growth in
harmonic Fourier series} 
\label{subsec:coeffgrowthharmonic}
If we have a distribution $h$ on the circle $\R/2\Z$, i.e., a
$2$-periodic distribution, it has a Fourier series expansion 
\begin{equation}
\label{eq:Fourierseries1.00}
h(t)=\sum_{n\in\Z}c_n\,e_n(t),\quad \text{where}\quad e_n(t)=\e^{\imag\pi nt},
\end{equation}
and the coefficients grow at most polynomially: $|c_n|=\Ordo(|n|^k)$ as 
$|n|\to+\infty$ for some finite $k$. The convergence of the series is
not understood pointwise on $\R$, but only in the sense of distribution theory
with convergence against a test function (cf. Definition \ref{defn:distrconv}).
We extend $h$ to a harmonic function in $\Hyp$ via the Poisson kernel as in  
\eqref{eq:Poisson0} and call the extension $h=\tilde h$ as well. 
Then $h$ meets the growth bound \eqref{eq:polgrowth0}, and since $h$ is 
$2$-periodic, we can do better. Indeed, by applying the Poisson extension 
term by term, we find that
\begin{equation}
h(\tau)=\sum_{n\in\Z}c_n\,e_n(\tau),\qquad \tau\in\Hyp,
\label{eq:Fourierseries1.01}
\end{equation}
where the functions $e_n(\tau)$ are given by \eqref{eq:harmext2} and 
\eqref{eq:harmext3}, depending on whether $n\ge0$ or $n<0$.
This gives the estimate
\begin{equation}
|h(\tau)|\le\sum_{n\in\Z}|c_n|\,|e_n(\tau)|=\Ordo\bigg(\sum_{n\in\Z}
(1+|n|)^k\,\e^{-\pi |n|\im\tau}\bigg)=\Ordo\bigg(1+\frac{1}{(\im\tau)^{k+1}}\bigg)
\end{equation}
uniformly in $\tau\in\Hyp$. We write this down as a proposition. 

\begin{prop}
Fix $k$ with $0<k<+\infty$.
If the coefficients $c_n\in\C$ have $|c_n|=\Ordo(|n|^k)$ as $|n|\to+\infty$,
it follows that the harmonically extended Fourier series
\eqref{eq:Fourierseries1.01} has the growth bound
\begin{equation}
|h(\tau)|\le\sum_{n\in\Z}|c_n|\,|e_n(\tau)|
=\Ordo\bigg(1+\frac{1}{(\im\tau)^{k+1}}\bigg)
\end{equation}
uniformly in $\tau\in\Hyp$.
\end{prop}

This proposition has a converse.

\begin{prop}
\label{prop:function->coeff0}
Fix $k$ with $0<k<+\infty$. Suppose $h:\Hyp\to\C$ is harmonic and that
\begin{equation}
|h(\tau)|=\Ordo\bigg(1+\frac{1}{(\im\tau)^{k}}\bigg)
\end{equation}
holds uniformly in $\tau\in\Hyp$. Then if $h$ is $2$-periodic, i.e.,
$h(\tau+2)=h(\tau)$ holds, it follows that $h$ is uniquely
represented by a harmonic Fourier series of the form
\eqref{eq:Fourierseries1.01},
where the coefficients have the growth $|c_n|=\Ordo(|n|^k)$ as $|n|\to+\infty$.
\end{prop}

\begin{proof}
Given that $f$ is $2$-periodic, we can use standard Fourier expansion on the
slice function $h(\cdot+\imag y)$ for fixed $y>0$ to the effect that
\begin{equation}
h(\tau+\imag y)=\sum_{n\in\Z}c_n(y)\,e_n(\tau),
\label{eq:harmonicFourier1}
\end{equation}
holds first for $\tau\in\R$ with the coefficients $c_n(y)$ in $\ell^2$,
and then in a second step for $\tau\in\Hyp$, by the generalized maximum
principle.
The coefficients are then given by the formula
\[
c_n(y)=\frac12\int_{-1}^{1}\e^{-\imag\pi n t}h(t+\imag y)\,\diff t,
\]
so that in view of the growth bound on $h$,
\begin{equation}
\label{eq:coeffbound2}
|c_n(y)|=\Ordo(1+y^{-k})
\end{equation}
holds uniformly in $y>0$. It remains to analyze the dependence of the
coefficients $c_n(y)$ on the parameter $y>0$. If we apply
\eqref{eq:harmonicFourier1} with $\tau=t+\imag s$ for $t\in\R$ and some
$s>0$, it is immediate that, since for $n\ge0$,
\[
e_n(t+\imag s)=\e^{\imag\pi n(t+\imag s)}
=\e^{-\pi n s}\e^{\imag\pi n t}=\e^{-\pi n s}e_n(t),
\]  
whereas for $n<0$,
\[
e_n(t+\imag s)=\e^{\imag\pi n(t-\imag s}
=\e^{\pi n s}\e^{\imag\pi n t}=\e^{\pi n s}e_n(t),
\] 
We derive that
\[
h(t+\imag s+\imag y)=\sum_{n\in\Z}c_n(s+y)\,e_n(t)=\sum_{n\in\Z}c_n(y)\,
e_n(t+\imag s)=\sum_{n\in\Z}c_n(y)\,\e^{-\pi |n|s}\,e_n(t),
\]
and equating coefficients, that
\[
c_n(s+y)=\e^{-\pi|n|s}c_n(y). 
\]
As this holds for all combinations of $s,y>0$, we conclude that the limit
\[
c_n:=\lim_{y\to0^+}c_n(y)
\]
exists, and in the process also that the $c_n(y)$ may be recovered from $c_n$:
\[
c_n(y)=\e^{-\pi |n|y}c_n.
\]
In view of \eqref{eq:coeffbound2}, we find that
\[
|c_n|=\Ordo\big((1+y^{-k})\,\e^{\pi|n|y}\big)
\]
and we are free to optimize over $y>0$. For $n\ne0$, we choose $y=1/|n|$,
which gives the desired estimate.  
\end{proof}

We need to consider also wilder growth of the coefficients or of
the periodic harmonic function. 

\begin{prop}
\label{prop:coeff->function}  
Suppose, for $\alpha\in\R$ and $0<\beta<+\infty$, that the coefficients
$c_n$ have the growth control
\[  
|c_n|=\Ordo\big(|n|^{-\frac34}\exp\big(2\pi\sqrt{\beta|n|}\big)\big)
\quad\text{as}\quad|n|\to+\infty.
\]  
Then the associated harmonic Fourier series \eqref{eq:Fourierseries1.01}
enjoys the growth bound
\[
|h(\tau)|\le\sum_{n\in\Z}|c_n|\,\e^{-\pi|n|\im\tau}=\Ordo(\exp(\pi\beta/\im\tau))
\]
uniformly in $\Hyp$. Moreover, under the corresponding little ``o'' condition
on the coefficients, the conclusion holds with little ``o'' in the limit as
$\im\,\tau\to0^+$.
\end{prop}

The proposition is immediate from the $\alpha=\frac32$ instance of the
following lemma.

\begin{lem}
For positive $\alpha,\beta, y$, and an integer
$N>\alpha^2/(4\pi^2\beta)$, we have that
\begin{equation}
\sum_{n\ge N}n^{-\alpha/2}\exp\big(2\pi\sqrt{\beta n}-\pi n y\big)
=\Ordo\big(\ybeta^{\alpha-\frac32}\,\e^{\pi\beta /y}\big)
\end{equation}
as $y\to0^+$. 
\label{lem:sumcontrol1}
\end{lem}

\begin{proof}
We consider the function
\[
\varphi(t):=-\pi \ybeta t^2+2\pi\sqrt{\beta} t-\alpha\log t,
\]
for positive real parameters $\alpha,\beta,\ybeta$, and observe that 
\begin{equation}
n^{-\alpha/2}\exp(2\pi\sqrt{\beta n}-\pi n\ybeta)=\exp(\varphi(\sqrt{n})).
\label{eq:exponentiate1}
\end{equation}
As the derivate equals
\[
\varphi'(t)=-2\pi \ybeta t+2\pi\sqrt{\beta}-\frac{\alpha}{t},
\]
we obtain two critical points,
\[
t_0(\ybeta)=\frac{\sqrt{\beta}}{2\ybeta}
\bigg(1-\sqrt{1-\frac{2\alpha\ybeta}{\pi\beta}}\bigg)
=\frac{\alpha}{2\pi\sqrt{\beta}}+\Ordo(\ybeta)
\]
and
\[
t_1(\ybeta)=\frac{\sqrt{\beta}}{2\ybeta}
\bigg(1+\sqrt{1-\frac{2\alpha\ybeta}{\pi\beta}}\bigg)
=\frac{\sqrt{\beta}}{\ybeta}-\frac{\alpha}{2\pi\sqrt{\beta}}
+\Ordo(\ybeta)
\]
as $\ybeta\to0^+$. For small $\ybeta>0$, the first critical point
$t_0(\ybeta)$ is smaller than $\sqrt{N}$,
so it lies outside of the given interval, and we may disregard it.
On the interval $[\sqrt{N},t_1(\ybeta)]$, the function
$\varphi$ is strictly increasing, whereas on $[t_1(\ybeta),+\infty[$,
$\varphi$ is instead strictly decreasing. The value
$\varphi(t_1(\ybeta))$ is then the unique
absolute maximum on $[\sqrt{N},+\infty[$, and we calculate that
\begin{equation}
\varphi(t_1(\ybeta))=\frac{\pi\beta}{\ybeta}
-\alpha\log \frac{\sqrt{\beta}}{\ybeta}
+\Ordo(\ybeta)
\label{eq:maxphi1}
\end{equation}
as $\ybeta\to0^+$. The second derivative equals
\[
\varphi''(t)=-2\pi \ybeta+\frac{\alpha}{t^2}.
\]
Let $\gamma_0$ be the positive real parameter given by
\[
\gamma_0^2=\frac{\alpha}{2(\pi-\eta_0)},
\]  
where $0<\eta_0<\pi$,
and observe that
\begin{equation}
\varphi''(t)=-2\pi \ybeta+\frac{\alpha}{t^2}\le
\bigg(-2\pi+\frac{\alpha}{\gamma_0^2}\bigg)\ybeta=-2\eta_0 \ybeta,
\qquad t\in[\gamma_0 \ybeta^{-\frac12},+\infty[. 
\end{equation}
By Taylor's formula, we have that
\begin{equation}
\varphi(t)=\varphi(t_1(\ybeta))+\int_{t_1(\ybeta)}^{t}(t-s)\varphi''(s)\diff s
=\varphi(t_1(\ybeta))+\frac12(t-t_1(\ybeta))^2\varphi''(\xi),
\end{equation}
for some point $\xi$ between $t_1(\ybeta)$ and $t$, and consequently,
we obtain
\begin{multline}
\varphi(t)=\varphi(t_1(\ybeta))+\frac12(t-t_1(\ybeta))^2\varphi''(\xi)\le
\varphi(t_1(\ybeta))-\eta_0 \ybeta(t-t_1(\ybeta))^2
\\
=\frac{\pi\beta}{\ybeta}-\alpha\log \frac{\sqrt{\beta}}{\ybeta}
-\eta_0 \ybeta(t-t_1(\ybeta))^2
+\Ordo(\ybeta),\qquad t\in[\gamma_0 \ybeta^{-\frac12},+\infty[, 
\label{eq:longintest1}
\end{multline}
since $t_1(\ybeta)\ge\gamma_0 \ybeta^{-\frac12}$ holds as $\ybeta\to0^+$.
Note that in \eqref{eq:longintest1}, the $\Ordo(\ybeta)$ term is independent
of $t$. 
As for the remaining interval, we use monotonicity:
\begin{multline}
\varphi(t)\le \varphi(\gamma_0 \ybeta^{-\frac12})
=-\pi\gamma_0^2-\alpha\log\gamma_0
+\frac{\alpha}{2}\log\ybeta+2\pi\sqrt{\beta}\gamma_0 \ybeta^{-\frac12},\qquad
t\in[\sqrt{N},\gamma_0\ybeta^{-\frac12}].
\\
\label{eq:shortintest1}
\end{multline}
Using the monotonicity properties of $\varphi$, we find that
\begin{equation}
\sum_{N\le n<t_1(\ybeta)^2}\exp(\varphi(\sqrt{n}))\le\int_{N}^{t_1(\ybeta)^2}
\exp(\varphi(\sqrt{s}))\diff s+\exp(\varphi(t_1(\ybeta)))
\end{equation}
and
\begin{equation}
\sum_{n\ge t_1(\ybeta)^2}\exp(\varphi(\sqrt{n}))\le\int_{t_1(\ybeta)^2}^{+\infty}
\exp(\varphi(\sqrt{s}))\diff s+\exp(\varphi(t_1(\ybeta))).
\end{equation}
By adding up these estimates, it follows that
\begin{multline}
\sum_{n\ge N}\exp(\varphi(\sqrt{n}))\le\int_{N}^{+\infty}
\exp(\varphi(\sqrt{s}))\diff s+2\exp(\varphi(t_1(\ybeta)))
\\
=2\int_{\sqrt{N}}^{+\infty}
\exp(\varphi(t))t\diff t+2\exp(\varphi(t_1(\ybeta))).
\label{eq:sumestimate1}
\end{multline}
Next, from the estimate \eqref{eq:shortintest1}, it follows that
\begin{equation}
2\int_{\sqrt{N}}^{\gamma_0 \ybeta^{-\frac12}}\exp(\varphi(t))\,t\diff t
\le\gamma_0^2\ybeta^{-1}
\exp(\varphi(\gamma_0 y^{-\frac12}))=
\gamma_0^{2-\alpha}\e^{-\pi\gamma_0^2}\ybeta^{-1+\frac{\alpha}{2}}
\exp(2\pi\sqrt{\beta}\gamma_0 \ybeta^{-\frac12}). 
\label{eq:shortintest2}
\end{equation}
On the other hand, from \eqref{eq:longintest1}, we find that
\begin{multline}
2\int_{\gamma_0 \ybeta^{-\frac12}}^{+\infty}
\exp(\varphi(t))t\diff t
\le (1+\Ordo(\ybeta))\,\beta^{-\alpha/2}\ybeta^{\alpha}\e^{\pi\beta/\ybeta}
\int_{\gamma_0 \ybeta^{-1/2}}^{+\infty}\exp(-\eta_0\ybeta(t-t_1(\ybeta))^2)2t\diff t
\\
\le (1+\Ordo(\ybeta))\,\beta^{-\alpha/2}\ybeta^{\alpha}\e^{\pi\beta/\ybeta}
\int_{0}^{+\infty}\exp(-\eta_0\ybeta(t-t_1(\ybeta))^2)2t\diff t
\\
\le (1+\Ordo(\ybeta))\,\beta^{-\alpha/2}\ybeta^{\alpha}\e^{\pi\beta/\ybeta}
\big(\eta_0^{-1}\ybeta^{-1}\e^{-\eta_0 y t_1(\ybeta)^2}+
2\pi^{\frac12}\eta_0^{-\frac12}\ybeta^{-\frac12}t_1(\ybeta)\big)
\\
\le2\sqrt{\pi}\eta_0^{-\frac12}\beta^{(1-\alpha)/2}y^{\alpha-\frac32}(1+\Ordo(\ybeta))
\big(1+\Ordo(\eta_0^{-\frac12}\ybeta^{\frac12})\big)
\e^{\pi\beta/\ybeta}
\label{eq:shortintest2.001}
\end{multline}
as $\ybeta\to0^+$. We may now fix $\eta_0:=\pi/2$ and hence
$\gamma_0^2=\frac{\alpha}{\pi}$. This permits us to not worry about the
dependence on $\eta_0$ and $\gamma_0$ in the above estimates.
Putting things together, we find that the inequality \eqref{eq:sumestimate1}
leads to 
\begin{multline}
\sum_{n\ge N}\exp(\varphi(\sqrt{n}))\le\int_{N}^{+\infty}
\exp(\varphi(\sqrt{s}))\diff s+2\exp(\varphi(t_1(\ybeta)))
\\
=\Ordo\Big(\ybeta^{-1+\frac{\alpha}{2}}
\exp\big(2\sqrt{\pi\alpha\beta}\ybeta^{-\frac12}\big)\Big)
\\
+\sqrt{2}\beta^{(1-\alpha)/2}\ybeta^{\alpha-\frac32}
(1+\Ordo(\ybeta^{\frac12}))\e^{\pi\beta/\ybeta}+2\ybeta^{\alpha}\e^{\pi\beta/\ybeta}
(1+\Ordo(\ybeta))
\\
=\sqrt{2}\beta^{(1-\alpha)/2}
(1+\Ordo(\ybeta^{\frac12}))\ybeta^{\alpha-\frac32}\e^{\pi\beta/\ybeta}
\label{eq:sumestimate2}
\end{multline}
as $\ybeta\to0^+$.
In view of the identity \eqref{eq:exponentiate1}, the claimed assertion now
follows.
\end{proof}

The converse to Proposition \ref{prop:coeff->function} reads as follows.

\begin{prop}
\label{prop:function->coeff1.0}  
Suppose, for $0<\beta<+\infty$, $h:\Hyp\to\C$ is a harmonic function
which is $2$-periodic, i.e., $h(\tau+2)=h(\tau)$ holds, and enjoys the estimate
\begin{equation}
|h(\tau)|=\Ordo\big(\exp(\pi\beta/\im\tau)\big)
\end{equation}
uniformly in $\tau\in\Hyp$. It then follows that $h$ is uniquely
represented by a harmonic Fourier series of the form
\eqref{eq:Fourierseries1.01}, where the coefficients have the growth bound
\[
|c_n|=\Ordo\big(\exp\big(2\pi\sqrt{\beta|n|}\big)\big)
\]
as $|n|\to+\infty$.
\end{prop}

\begin{proof}[Proof sketch]
Following along the same lines as the proof of Proposition
\ref{prop:function->coeff0}, we may use the following formula to recover
the Fourier coefficients:
\begin{equation}
c_n=\frac{\e^{\pi|n|s}}{2}\int_{-1}^{1}\e^{-\imag\pi nt}h(t+\imag s)\,\diff t,
\label{eq:get_coeff_cn}
\end{equation}
provided that $s>0$. 
The growth estimate of the function $f$ now gives that 
\[
|c_n|\le C\,\exp(\pi|n|s)\,\exp(\pi\beta/s),
\]
where $C$ is the implied constant.
Now optimizing over $s>0$ gives the indicated coefficient estimate.
\end{proof}

\subsection{Solving the Dirichlet problem in the half-plane with
possibly rapid growth}

We consider the Dirichlet problem of finding a harmonic function
$u:\Hyp\to\C$ which extends continuously to the closure $\Hyp\cup\R$,
with $u=\varphi$ on $\R$, where $\varphi:\R\to\C$ is a given continuous
function. If $\varphi$ has finite integral
\begin{equation}
\int_\R \frac{|\varphi(t)|}{1+t^2}\,\diff t<+\infty,
\label{eq:Poissonintegrability0.001}
\end{equation}
we can use the Poisson extension \eqref{eq:Poisson0} to define such a solution
\begin{equation}
\label{eq:harmext1.01}
u(\tau)=\int_\R P(\tau,t)\,\varphi(t)\,\diff t=\frac{1}{\pi}\int_{\R}
\frac{\im\tau}{|t-\tau|^2}\,\varphi(t)\,\diff t,\qquad \tau\in\Hyp.  
\end{equation}
This is the case in the following instance.

\begin{prop}
\label{prop:loggrowth}  
Suppose that $\varphi:\R\to\C$ is continuous with logarithmic power growth:
\[  
|\varphi(t)|=\Ordo(\log(2+|t|))^k\quad \text{as}\,\,\,|t|\to+\infty,
\]
for some integer $k$, $0<k<+\infty$. Then the Poisson extension $u$ given by
\eqref{eq:harmext1.01} has the same growth rate
\[  
|u(\tau)|=\Ordo(\log(2+|\tau|))^k,
\]
uniformly in $\Hyp$. 
\end{prop}  

\begin{proof}
The growth bound \eqref{eq:Poissonintegrability0.001} on $\varphi$ holds.
We consider the holomorphic function
\[
V(\tau)=\log(\tau+2\imag)=\log|\tau+2\imag|+\imag\arg(\tau+2\imag),\qquad
\tau\in\Hyp,
\]
where the argument is chosen in the interval from $0$ to $\pi$. The $k$-power
of $V$ has
\[
V(\tau)^k=\log^k|\tau+2\imag|+\Ordo(\log^{k-1}|\tau+2\imag|),\qquad
\tau\in\Hyp,
\]
so that the same applies to its real part:
\[
\re\, V(\tau)^k=\log^k |\tau+2\imag|+\Ordo(\log^{k-1}|\tau+2\imag|),\qquad
\tau\in\Hyp.
\]
Next, we observe that without loss of generality we may assume that $\varphi$
is real-valued. Indeed, if the assertion of the lemma holds for the real and
imaginary parts of $\varphi$ separately, then it holds also for $\varphi$.
In this case, the Poisson extension $u$ is harmonic and real-valued, and
\[
\varphi(t)-A\,\re\,V(t)^k\le B,\qquad t\in\R,
\]
holds for some positive constants $A,B$, by the growth assumption on $\varphi$.
This then carries over to the Poisson extension of the left-hand side:
\[
u(\tau)-A\,\re\,V(\tau)^k\le B,\qquad \tau\in\Hyp.
\]
This establishes the estimate from above on $u$. The analogous estimate from
below follows by the same argument applied to $-u$ in place of $u$ and
$-\varphi$ in place of $\varphi$. In view of this, the assertion now follows.
\end{proof}

The maximum principle can also help in cases of slower rates of growth. 

\begin{prop}
\label{prop:logloggrowth}  
Suppose that $\varphi:\R\to\C$ is continuous with double logarithmic growth:
\[  
|\varphi(t)|=\Ordo(\log\log(10+|t|))\quad \text{as}\,\,\,|t|\to+\infty.
\]
Then the Poisson extension $u$ given by \eqref{eq:harmext1.01} has the same
growth rate
\[  
|u(\tau)|=\Ordo(\log\log(10+|\tau|)),
\]
uniformly in $\Hyp$. 
\end{prop} 

\begin{proof}[Proof sketch]
If we let $V$ be the holomorphic function
\[
V(\tau)=\log\log(\tau+10\imag),\qquad\tau\in\Hyp,
\]
we may use its real part as a comparison function, as in the proof of
Proposition \ref{prop:loggrowth}. The details of the argument
are left to the reader. 
\end{proof}  

Although the Poisson extension works very well under the integrability
condition \eqref{eq:Poissonintegrability0.001}, we are left to wonder about the
uniqueness. For instance, the function $\tau\mapsto\im\,\tau$ is harmonic in
$\Hyp$ and vanishes on the line $\R$, so we may add it to any given solution
$u$ to the Dirichlet problem and still solve the boundary value problem.
The following proposition explains the degree of nonuniquess if we do not
impose any growth condition on the solution $u$. The result is well-known,
but the proof is short and we include it for completeness.

\begin{prop}
Suppose that $u:\Hyp\to\C$ is harmonic and extends continuously to the closed
half-plane $\Hyp\cup\R$, and that $u(t)=0$ holds for each $t\in\R$. Then
there exist two entire functions $F_1,F_2:\C\to\C$ with real Taylor
coefficients at the point $0$, such that
\[
u(\tau)=\im\,F_1(\tau)+\imag\,\im\,F_2(\tau),\qquad \tau\in \Hyp.
\]  
\end{prop}

\begin{proof}
By the Schwarz reflection principle \cite{rud} (Theorem 11.17), 
if we extend $u$ to the lower half-plane
by setting
\[
u(\bar\tau):=-u(\tau),\qquad \tau\in\Hyp,
\]
we obtain a harmonic function in the whole plane $\C$ which vanishes on $\R$.
Considering the real and imaginary parts of $u$ separately, we find entire
functions $F_1,F_2$ such that $\re \,u=\im\, F_1$ and $\im\,u=\im\, F_2$.
These two entire functions take real values along $\R$, i.e., 
they have real Taylor coefficients at $0$.   
\end{proof}  

As for the solvability of the of the Dirichlet problem for general continuous
$\varphi:\R\to\C$, it has been considered and the answer is affirmative
(see the work of Finkelstein and Scheinberg \cite{FinkSchein}). We have found
an alternative short proof which we also supply.

\begin{prop}
\label{prop:Dirichlet-general}  
Suppose $\varphi:\R\to\C$ is continuous. Then there exists a harmonic function
$u:\Hyp\to\C$ which extends continuously to the closed half-plane $\Hyp\cup\R$,
such that $u(t)=\varphi(t)$ holds for all $t\in\R$.
\end{prop}

\begin{proof}
It is an exercise to find an even entire function $\Psi$ with real Taylor
coefficients at the point $0$, which grows at least as quickly along the real
line as any given real-valued continuous function. In particular,
it is possible to find such a $\Psi$ with
\[
\int _\R |\varphi(t)|\,\exp(-\re\,\Psi(t))\,\diff t<+\infty.
\]
Then it can be shown that the formula
\[
u(\tau)=\frac{1}{\pi}\int_\R\im
\bigg(\frac{\exp(\Psi(\tau)-\Psi(t))}{t-\tau}\bigg)\,\varphi(t)\,\diff t,
\qquad\tau\in\Hyp,
\]
defines a harmonic function $u$ whose boundary values equal $\varphi$.
\end{proof}
  
This settles the matter of the Dirichlet problem with arbitrary continuous
boundary values, at the cost of allowing arbitrary growth of the solution $u$.
But what about an intermediate situation, when we have growth control on
$\varphi$, can we get analogous growth control of the solution $u$?

\begin{prop} Fix an $\alpha\in\R_{>0}$,
and suppose $\varphi:\R\to\C$ is continuous with the growth bound
\[
|\varphi(t)|=\Ordo\big((1+|t|)^\alpha\big)\quad\text{as}\,\,\,|t|\to+\infty.
\]
Then if $\alpha\notin\Z$, there exists a harmonic function
$u=u_\alpha:\Hyp\to\C$ which extends continuously to the closed half-plane
$\Hyp\cup\R$ with $u=\varphi$ on the line $\R$, such that
\[
|u(\tau)|=\Ordo\big((1+|\tau|)^{\alpha}\big)
\]
holds uniformly in $\tau\in\Hyp$. Moreover, if $\alpha=k\in\Z_{>0}$,
there still is such a solution $u$, but the growth bound is slightly worse:
\[
|u(\tau)|=\Ordo\big((1+|\tau|)^{k}\log(2+|\tau|)\big),
\]
uniformly in $\tau\in\Hyp$. 
\label{prop:powergrowth}
\end{prop}

The proof of the proposition is deferred to the separate publication
\cite{HedWenn}. The reason is that it uses integration kernel methods and
in fact its inclusion would make the present work even longer than it is.

\begin{rem}
(a) For $0<\alpha<1$, the solution $u$ offered by Proposition
\ref{prop:powergrowth} coincides with the Poisson extension. However, for
$\alpha\ge1$ alternative harmonic extension kernels are needed.

(b) If we consider for $k\in\Z_{>0}$ the harmonic function
\[
u(\tau)=\im(\tau^k\log\tau),\qquad \tau\in\Hyp,
\]
where we use the principal branch of the logarithm, we find that the boundary
values are $u(t)=\pi t^k$ for $t<0$ while $u(t)=0$ for $t\ge0$. 
In particular, $u(t)=\Ordo(|t|^k)$ while the growth of $u(\tau)$ in the upper
half-plane $\Hyp$ involves a logarithm. This explains why the integer case
$\alpha=k\in\Z_{>0}$ of Proposition \ref{prop:powergrowth} is different from
the non-integer case.
\end{rem}

\section{Uniqueness issues for the Klein-Gordon equation}
\label{sec:KG-uniq}

\subsection{Translated lattice-crosses}
\label{subsec:translatedlatticecross}

The following lemma uses the idea of periodization in Fourier
analysis. It was key in, e.g., \cite{hed}, \cite{hed1}. 

\begin{lem}
\label{lem:periodization}
Let $\psi\in L^1(\R)$. Then
\[
\forall n\in\Z
:\,\,\,\int_\R \e^{\imag\pi nt}\psi(t)\,\diff t=0
\]
holds if and only if
\[
\sum_{j\in\Z}\psi(t+2j)=0\quad \text{a.e. on}\quad\R.
\]
\end{lem}  


\begin{proof}
The Fourier coefficients of the $2$-periodic periodization
\[
\Psi(t):=\sum_{j\in\Z}\psi(t+2j)
\]
are given by
\begin{equation}
\frac12\int_{[-1,1]}\e^{-\imag\pi m t}\Psi(t)\,\diff t=
\frac12\int_{[-1,1]}\e^{-\imag\pi m t}\sum_{j\in\Z}\psi(t+2j)\,\diff t=  
\frac12\int_\R \e^{-\imag\pi mt}\psi(t)\,\diff t
\label{eq:FouriercoeffPsi1}
\end{equation}
for $m\in\Z$. If the Fourier coefficients all vanish, then the
function $\Psi$ must vanish as well.
On the other hand, if $\Psi$ vanishes
almost everywhere, then the Fourier coefficients 
must vanish too.
\end{proof}

\begin{proof}[Proof of Proposition \ref{prop:uniq1.05}]
We begin with the setting of Proposition \ref{prop:uniq1.05}.
This means that $\varphi\in L^1(\R)$ is supposed to have the following two
properties, for some fixed $\omega_1,\omega_2\in\R$:
\begin{equation}
\label{eq:translatedlatticecond1}
\forall m\in\Z:\,\,\, \int_\R \e^{\imag\pi(m+\omega_1)t}\varphi(t)\,\diff t=0
\end{equation}
and
\begin{equation}
\label{eq:translatedlatticecond2}
\forall n\in\Z:\,\,\, \int_\R \e^{\imag\pi(n+\omega_2)/t}\varphi(t)\,\diff t
=\int_\R\e^{-\imag\pi (n+\omega_2)s}\varphi(-1/s)\,s^{-2}\diff s =0.
\end{equation}
We apply Lemma \ref{lem:periodization} to the functions
\[
\psi_1(t):=\e^{\imag\pi\omega_1 t}\varphi(t),\quad
\psi_2(t):=t^{-2}\e^{-\imag\pi\omega_2 t}\varphi(-1/t),  
\]
and find that 
the conditions \eqref{eq:translatedlatticecond1}
and \eqref{eq:translatedlatticecond2} amount to having
\begin{equation}
\sum_{j\in\Z}\psi_l(t+2j)=0,\qquad l=1,2,
\end{equation}
almost everywhere on $\R$, respectively. The condition for $l=1$ means that
\begin{equation}
\label{eq:firstcond1}
\varphi(t)=-\sum_{j\in\Z_{\neq0}}\e^{\imag2\pi j\omega_1}\varphi(t+2j)
\end{equation}
almost everywhere on $\R$, while the condition for $l=2$ is the same as
\begin{equation}
\varphi(t)=-\sum_{j\in\Z_{\ne0}}\frac{\e^{-\imag2\pi j\omega_2}}{(1-2jt)^2}
\,\varphi\Big(\frac{t}{1-2jt}\Big),  
\label{eq:secondcond1}
\end{equation}
again almost everywhere on $\R$. We may combine the conditions
\eqref{eq:firstcond1} and \eqref{eq:secondcond1} in two different ways,
by inserting \eqref{eq:firstcond1} into the right-hand side of
\eqref{eq:secondcond1} or, alternatively, by inserting \eqref{eq:secondcond1}
into the right-hand side of \eqref{eq:firstcond1}. We choose the second
alternative, and obtain
\begin{equation}
\varphi(t)=\sum_{j,k\in\Z_{\ne0}}
\frac{\e^{\imag2\pi (k\omega_1-j\omega_2)}}{(1-2j(t+2k))^2}
\,\varphi\Big(\frac{t+2k}{1-2j(t+2k)}\Big),  \quad\text{a.e.}\,\,\,t\in\R.
\label{eq:secondcond2}
\end{equation}
If we put, for a given $\omega\in\R$,
\begin{equation}
\Tope_\omega f(t):=\sum_{j\in\Z_{\ne0}}\frac{\e^{\imag2\pi j\omega}}{(2j-t)^2}
\,f\Big(\frac{1}{2j-t}\Big),  
\label{eq:Topeomega}
\end{equation}
the equation \eqref{eq:secondcond2} may be expressed in the form
\begin{equation}
\varphi(t)=\Tope_{\omega_2}\Tope_{-\omega_1}\varphi(t), \quad\text{a.e.}\,\,\,
t\in\R.
\label{eq:secondcond3}
\end{equation}
In view of the triangle inequality, we find that
\begin{multline}
\label{eq:triangleineq1}
|\Tope_\omega f(t)|=\bigg|\sum_{j\in\Z_{\ne0}}\frac{\e^{\imag2\pi j\omega}}{(2j-t)^2}
\,f\Big(\frac{1}{2j-t}\Big)\bigg|
\le\sum_{j\in\Z_{\ne0}}\frac{1}{(2j-t)^2}
\,|f|\Big(\frac{1}{2j-t}\Big)=\Tope_0 |f|(t),\qquad t\in\R.
\end{multline}
In particular, by iteration, if $\varphi$ solves \eqref{eq:secondcond3},
then, by taking absolute values, we find that
\begin{equation}
|\varphi(t)|=|\Tope_{\omega_2}\Tope_{-\omega_1}\varphi(t)|
\le\Tope_0(|\Tope_{-\omega_1}\varphi|)\le\Tope_0^2|\varphi|(t),
\quad\text{a.e.}\,\,\,t\in\R.
\label{eq:secondcond4}
\end{equation}
Here, we used the rather obvious property that $\Tope_0$ is positive,
in the sense that if $f,g:\R\to\R_{\ge0}$ and $f\le g$ holds holds almost 
everywhere, then $\Tope_0 f\le\Tope_0 g$ holds almost everywhere as well. 
Another important
relevant observation is that in view of the definition \eqref{eq:Topeomega}
of $\Tope_0$,  we have the following \emph{localization principle}:
the restriction of $f$ to $[-1,1]$ determines the restriction of $\Tope_0 f$
to $[-1,1]$.
Next, according to Proposition 3.4.1 in \cite{hed1}, 
\[
\int_{[-1,1]}\Tope_0^2|\varphi|(t)\,\diff t=
\int_{[-1,1]}\Tope_0|\varphi|(t)\,\diff t=\int_{[-1,1]}|\varphi(t)|\,\diff t,
\]
which means that \eqref{eq:secondcond4} is only possible if in fact
\begin{equation}
|\varphi(t)|=\Tope_0^2|\varphi|(t),
\quad\text{a.e.}\,\,\,t\in[-1,1].
\label{eq:secondcond5}
\end{equation}
This can be thought of as asking that $|\varphi(t)|\diff t$ be an absolutely
continuous invariant measure for the Gauss-type map $t\mapsto -1/t$ (mod 1)
on the interval $[-1,1]$, which is not possible unless
$\varphi(t)=0$ almost everywhere on $[-1,1]$. This was shown using the
Birkhoff ergodic theorem in the infinite ergodic setting in \cite{hed}.
A more direct proof follows from Proposition 3.13.3 in \cite{hed1},
which asserts that the iterates $\Tope_0^{2n}|\varphi|$ tend to $0$ in
the $L^1$ norm on subintervals $[-1+\epsilon,1-\epsilon]$, for fixed $\epsilon$
with $0<\epsilon<1$. Since by the localization principle applied to
\eqref{eq:secondcond5}, we must have that
\[
|\varphi(t)|=\Tope_0^{2n}|\varphi|(t),
\quad\text{a.e.}\,\,\,t\in[-1,1], 
\]
for $n=1,2,3,\ldots$, it is immediate that $\varphi(t)=0$ for almost every
$t\in[-1,1]$. It remains to show that $\varphi(t)=0$ holds for almost
every $t\in\R$. To this end, we can use the identity \eqref{eq:secondcond1},
since we already know that $\varphi(t)=0$ almost everywhere on $[-1,1]$.
For $t\in\R$ with $|t|>1$ and $j\in\Z_{\ne0}$, 
\[
\frac{t}{1-2jt}\in[-1,1],
\]
so that in view of \eqref{eq:secondcond1}, the information about $\varphi$ on
$[-1,1]$ carries over to the complement $\R\setminus[-1,1]$, and we
conclude that $\varphi=0$ as an element of $L^1(\R)$. 
\end{proof}

\subsection{Normal derivatives on the lattice-cross}
\label{subsec:normalderivatives}
We turn to the proof of Proposition \ref{prop:extraintegrability}.

\begin{proof}[Proof of Proposition \ref{prop:extraintegrability}]
If $k=0$, the assertion coincides with that of Theorem \ref{intth0},
so we may assume in the rest of the proof that $k>0$.  
It is given that $\varphi\in L^1(\R)$ has the extra integrability
\eqref{eq:extraintegrability}, and that
\[
\int_\R \e^{\imag\pi m t}\varphi(t)\,\diff t=\int_\R\e^{-\imag \pi n/t}
\varphi(t)\,t^k\,\diff t=0,\qquad m,n\in\Z,
\]
holds. After a change of variables in the latter integral, the conditions read
\begin{equation}
\int_\R \e^{\imag\pi m t}\varphi(t)\,\diff t=\int_\R\e^{\imag \pi n t}
\varphi(-1/t)\,t^{-2-k}\, \diff t=0,\qquad m,n\in\Z,
\label{eq:condition.101}
\end{equation}
We apply Lemma \ref{lem:periodization} to the functions
\[
\psi_1(t):=\varphi(t),\quad \psi_2(t):=t^{-2-k}\varphi(-1/t),
\]
which are in $L^1(\R)$ because of \eqref{eq:extraintegrability},
and find that the condition \eqref{eq:condition.101} amounts to having
\[
\sum_{j\in\Z}\psi_l(t+2j)=0,\qquad l=1,2,
\]
almost everywhere on $\R$, respectively.
The condition for $l=1$ is
\begin{equation}
\label{eq:firstcond1.ny1}
\varphi(t)=-\sum_{j\in\Z_{\neq0}}\varphi(t+2j)
\end{equation}
almost everywhere on $\R$, while the condition for $l=2$ may be expressed in
the form 
\begin{equation}
\varphi(t)=
-\sum_{j\in\Z_{\ne0}}(1-2jt)^{-2-k}\,\varphi\Big(\frac{t}{1-2jt}\Big),  
\label{eq:secondcond1.ny1}
\end{equation}
again almost everywhere on $\R$.
We may combine the conditions
\eqref{eq:firstcond1.ny1} and \eqref{eq:secondcond1.ny1} in two different ways,
by inserting \eqref{eq:firstcond1.ny1} into the right-hand side of
\eqref{eq:secondcond1.ny1} or, alternatively, by inserting
\eqref{eq:secondcond1.ny1} into the right-hand side of
\eqref{eq:firstcond1.ny1}. We choose the second alternative, and obtain
\begin{equation}
\varphi(t)=
\sum_{j,j'\in\Z_{\ne0}}
(1-2j(t+2j'))^{-2-k}
\,\varphi\Big(\frac{t+2j'}{1-2j(t+2j')}\Big),  \quad\text{a.e.}\,\,\,t\in\R.
\label{eq:secondcond2.ny1}
\end{equation}
If we consider the ``transfer operator''
\begin{equation}
\Tope_k f(t):=\sum_{j\in\Z_{\ne0}}(2j-t)^{-2-k}
\,f\Big(\frac{1}{2j-t}\Big),  
\label{eq:Topeomega.ny1}
\end{equation}
the equation \eqref{eq:secondcond2.ny1} may be expressed in the form
\begin{equation}
\varphi(t)=(-1)^k\,\Tope_k^2\varphi(t),
\quad\text{a.e.}\,\,\,
t\in\R.
\label{eq:secondcond3.ny1}
\end{equation}
In view of the triangle inequality, we find that
\begin{multline}
\label{eq:triangleineq1.ny1}
|\Tope_k f(t)|=\bigg|\sum_{j\in\Z_{\ne0}}(2j-t)^{-2-k}
\,f\Big(\frac{1}{2j-t}\Big)\bigg|
\le\sum_{j\in\Z_{\ne0}}|2j-t|^{-2-k}
\,|f|\Big(\frac{1}{2j-t}\Big)=\Tope_{\langle k\rangle} |f|(t),\qquad t\in\R,
\end{multline}
where $\Tope_{\langle k\rangle}$ denotes the operator
\begin{equation}
\Tope_{\langle k\rangle} f(t):=\sum_{j\in\Z_{\ne0}}|2j-t|^{-2-k}
\,f\Big(\frac{1}{2j-t}\Big).  
\label{eq:Topeomega.ny10}
\end{equation}

In particular, by iteration and by monotonicity, if $\varphi$ solves
\eqref{eq:firstcond1.ny1}
and \eqref{eq:secondcond1.ny1}, then, by taking absolute values, we find that
\begin{equation}
|\varphi(t)|=|\Tope_{k}^2\varphi(t)|
\le\Tope_{\langle k\rangle}(|\Tope_{k}\varphi|)\le\Tope_{\langle k\rangle}^2
|\varphi|(t),
\quad\text{a.e.}\,\,\,\,t\in\R.
\label{eq:secondcond4.ny1}
\end{equation}
The operators $\Tope_k$ and $\Tope_{\langle k\rangle}$ meet the following
\emph{localization principle}: the restriction of $f$ to $[-1,1]$ determines
the restriction of both $\Tope_k f$ and $\Tope_{\langle k\rangle}f$ to $[-1,1]$.
We find that for measurable $f$, a suitable change of variables yields
\begin{multline}
\label{eq:Tkoperator1}  
\int_{[-1,1]}|\Tope_k f(t)|\,\diff t\le \int_{[-1,1]}\Tope_{\langle k\rangle}
|f|(t)\,\diff t\\
=\sum_{j\in\Z_{\ne0}}\int_{[-1,1]}|2j-t|^{-2-k}
\,|f|\Big(\frac{1}{2j-t}\Big)\,\diff t
=\sum_{j\in\Z_{\ne0}}\int_{[1/(2j+1),1/(2j-1)]}
\,|f(t)|\,|t|^k\diff t
\\
=\int_{[-1,1]}|f(t)|\,|t|^k\diff t\le\int_{[-1,1]}|f(t)|\,
\diff t,
\end{multline}
provided that the right-hand side integral is finite. Next, by iteration
of \eqref{eq:Tkoperator1},
\begin{equation}
\label{eq:Tkoperator2}
\int_{[-1,1]}|\Tope_k^2\varphi (t)|\,\diff t
\le\int_{[-1,1]}|\Tope_k\varphi (t)|\,\diff t
\le \int_{[-1,1]}|\varphi (t)|\,|t|^k\diff t\le \int_{[-1,1]}
|\varphi (t)|\,\diff t.
\end{equation}
We proceed to show that
for $\varphi\in L^1(\R)$, the only solution to \eqref{eq:secondcond3.ny1}
is $\varphi=0$. By \eqref{eq:secondcond3.ny1} and \eqref{eq:Tkoperator2},
\begin{equation}
\label{eq:Tkoperator2.001}
\int_{[-1,1]}|\varphi(t)|\diff t=\int_{[-1,1]}|\Tope_k^2\varphi (t)|\,\diff t
\\
\le \int_{[-1,1]}|\varphi (t)|\,|t|^k\diff t\le \int_{[-1,1]}
|\varphi (t)|\,\diff t.
\end{equation}
The left-hand and right-hand sides of \eqref{eq:Tkoperator2.001} coincide,
so we must have equality throughout in \eqref{eq:Tkoperator2.001}.
But since $k>0$, it follows that $|t|^k<1$ for $|t|<1$, which entails that
the last inequality in \eqref{eq:Tkoperator2} is strict unless $\varphi=0$
almost everywhere on $[-1,1]$. Hence we conclude from \eqref{eq:Tkoperator2}
that $\varphi(t)=0$ holds almost everywhere on $[-1,1]$. It remains to show
that $\varphi$ vanishes almost everywhere on $\R\setminus[-1,1]$ as well.
To this end we can use \eqref{eq:secondcond1.ny1}, 
since the fact that
\[
\frac{t}{1-2jt}\in[-1,1]\quad \text{for}\quad t\in\R\setminus[-1,1],\,\,\,
j\in\Z_{\ne0}  
\]
holds shows that
\[
\varphi(t)=-\sum_{j\in\Z_{\ne0}}(1-2jt)^{-2-k}
\,\varphi\Big(\frac{t}{1-2jt}\Big)=0,
\qquad\text{a.e.}\,\,\, t\in\R\setminus[-1,1].
\]
The conclusion that $\varphi(t)=0$ almost everywhere on $\R$ is now immediate.
\end{proof}

\section{Uniqueness and beyond for hyperbolic
Fourier series}
\label{sec:uniq-beyond}

\subsection{Proofs of the theorems on uniqueness}
\label{subsec:uniq-beyond}

We turn to the uniqueness of the harmonic hyperbolic Fourier expansion.

\begin{proof}[Proof of Theorems \ref{thm:uniq1} and \ref{thm:uniq2.1}]
We first note that Theorem \ref{thm:uniq1} may be viewed as the special case
of $k=1$ in Theorem \ref{thm:uniq2.1}, so we proceed to obtain Theorem
\ref{thm:uniq2.1} for general $k=1,2,3,\ldots$, with the understanding that
$\mathfrak{E}^{\mathrm{harm}}_{k-1}=\{0\}$ for $k=1$.
We note that under the given growth control on the coefficients, the
series \eqref{eq:harmext1.1} converges absolutely and uniformly on compact
subsets of $\Hyp$. In terms of the functions $e_n$ given by
\eqref{eq:harmext2} (for $n\ge0$) and \eqref{eq:harmext3} (for $n<0$),
the vanishing condition reads
\begin{equation}
a_0+\sum_{n\ne0}\big(a_n\,e_n(\tau)+b_n\,e_n(-1/\tau)\big)=0,\qquad \tau\in\Hyp,
\label{eq:pointwiserep2}
\end{equation}
which we rewrite as
\begin{equation}
a_0+\sum_{n>0}\big(a_n\,e_n(\tau)+b_n\,e_n(-1/\tau)\big)
=-\sum_{n<0}\big(a_n\,e_n(\tau)+b_n\,e_n(-1/\tau)\big),
\qquad \tau\in\Hyp.
\label{eq:pointwiserep3}
\end{equation}
We observe that the left-hand side is holomorphic, whereas the right-hand
side is conjugate-holomorphic. This is possible only if each of the sides
is constant:
\begin{equation}
\sum_{n\ge0}\big(a_n\,e_n(\tau)+b_n\,e_n(-1/\tau)\big)=C,
\qquad \tau\in\Hyp.
\label{eq:pointwiserep4}
\end{equation}
Our next task is to combine this with the growth control on the
coefficients $a_n$ to show that the combined sequence of coefficients
$\{a_0,a_n,b_n\}_{n>0}$ is exceptional of degree $\le k$ in the sense of
Definition \ref{def:exceptional:1-sided}.
To this end, we introduce the functions
\[
F_+(\tau)=:a_0+
\sum_{n>0}a_n\,e_n(\tau),\qquad
G_+(\tau):=\sum_{n>0}b_n\,e_n(\tau),
\]
which are holomorphic $2$-periodic functions in $\Hyp$ with
$F_+(+\imag\infty)=a_0$ and $G_+(+\imag\infty)=0$.
In view of our growth assumption on the coefficients $a_n$, 
the little ``o'' statement of Proposition \ref{prop:coeff->function} gives that
\begin{equation}
\limsup_{\im \,\tau\to0^+}\,\e^{-\pi k/\im \tau}\,|F_+(\tau)|=0.
\label{eq:basicest1.04}
\end{equation}
Moreover, the equality \eqref{eq:pointwiserep4} asserts that
\begin{equation}
F_+(\tau)+G_+(-1/\tau)=C,\qquad \tau\in\Hyp,
\label{eq:pointwiserep5}
\end{equation}
which we write in the form
\begin{equation}
F_+(\tau)=C-G_+(-1/\tau),\qquad \tau\in\Hyp.
\label{eq:pointwiserep6}
\end{equation}
As the functions $F_+,G_+$ are $2$-periodic, that is,
$F_+(\tau+2)=F_+(\tau)$ and $G_+(\tau+2)=G_+(\tau)$, which we can express as
$F_+\circ\Ttrans^2=F_+$ and $G_+\circ\Ttrans^2=G_+$, where
$\Ttrans(\tau):=\tau+1$. If we write
$\Strans(\tau):=-1/\tau$, we find that according to \eqref{eq:pointwiserep6},
$F_+(\tau)$ is must be both $2$-periodic and invariant under the transformation
$\Strans\Ttrans^2\Strans:\,\tau\mapsto \tau/(1-2\tau)$. Then $F_+,G_+$ are
invariant under the group $\Gamma_{2\Theta}$ consisting of all M\"obius
automorphisms of $\Hyp$ generated by the elements $\Ttrans^2$ and
$\Strans\Ttrans^2\Strans$, so we may think of them as holomorphic functions on
the Riemann surface $\Hyp/\Gamma_{2\Theta}$. The fundamental domain of
$\Hyp/\Gamma_{2\Theta}$ is
\[
\calD_{2\Theta}:=\big\{\tau\in\Hyp:\,|\re\tau|<1,\,\,\,
|\tau+\tfrac12|>\tfrac12,\,\,\,|\tau-\tfrac12|>\tfrac12\big\}.
\]
To get the Riemann surface from the fundamental domain, we glue together the
two linear boundary segments where the real part is $\pm1$, and we also glue
the semicircles together. The result is that the Riemann surface
$\Hyp/\Gamma_{2\Theta}$ is a sphere with three punctures (corresponding to
the points $0$, $\pm1$, and $+\imag\infty$). Moreover, we know that $F_+,G_+$
extend holomorphically across the puncture at $+\imag\infty$, and that
$G_+$ vanishes at that point. In addition, the identity
\eqref{eq:pointwiserep6} asserts that both $F_+,G_+$ extend holomorphically
across the puncture at $0$ as well. This means that there only remains a single
possible singularity at the remaining puncture $\pm1$, and the given growth
control there determines whether we have a pole or essential singularity. 
In the given setting, there is at most a pole of order $\le k-1$ at the
puncture $\pm1$. This situation 
is actually outlined in detail in Lemma 1.1 of \cite{bh2}.  
With the parameter values $n_0=n_\infty=0$ and $n_1=k-1$, the conditions (1),
(2), and (3) of Lemma 1.1  \cite{bh2} are all met. This follows from
\eqref{eq:basicest1.04} and the fact that $F_+(\tau)$ has limit $0$
at $+\imag\infty$.
It is now a consequence of the assertion of Lemma 1.1 of \cite{bh2}
that $F_+(\tau)$ is of the form
\begin{equation}
F_+(\tau)=a_0+P(\lambda(\tau)),\qquad \tau\in\Hyp,
\label{eq:Fplus1.1}
\end{equation}
where $P$ stands for a polynomial of degree $\le k-1$ with $P(0)=0$. 
But then \eqref{eq:pointwiserep5} tells us that
$G_+$ takes the form
\begin{equation}
G_+(\tau)=C-F_+(-1/\tau)=C-a_0-P(\lambda(-1/\tau))=C-a_0-P(1-\lambda(\tau)),
\label{eq:Gplus1.1}
\end{equation}
where we used the functional properties \eqref{eq:lambdaproperties1.1}.
We note that by plugging in the value $\tau=+\imag\infty$, while using that
$G_+(+\imag\infty)=\lambda(+\imag\infty)=0$, we obtain from
\eqref{eq:Gplus1.1} that
\begin{equation}
P(1)=C-a_0.  
\end{equation}
We note that \eqref{eq:Fplus1.1} and \eqref{eq:Gplus1.1}
assert that \emph{the sequence $\{a_0-C,a_n,b_n\}_{n>0}$ is exceptional of degree
$\le k-1$}. 
From \eqref{eq:pointwiserep3} and \eqref{eq:pointwiserep4} we conclude that 
\begin{equation}
\sum_{n<0}\big(a_n\, e_n(\tau)+b_n\,e_n(-1/\tau)\big)=-C,\qquad \tau\in\Hyp.
\label{eq:pointwiserep7}
\end{equation}
Since for $n>0$, $e_{-n}(\tau)={e_{n}(-\bar\tau)}$ holds, it is natural
to introduce the holomorphic functions
\[
F_-(\tau):=\sum_{n>0}a_{-n}\,e_{n}(\tau),\qquad
G_-(\tau):=\sum_{n>0} b_{-n} \,e_{n}(\tau),
\]
and to observe that \eqref{eq:pointwiserep7} asserts that
\begin{equation}
F_-(\tau)+G_-(-1/\tau)=-C,\qquad \tau\in\Hyp.
\label{eq:pointwiserep8}
\end{equation}
We are now in an analogous situation as we were in with $F_+,G_+$, and
may conclude from a repetition of the argument that $F_-$ and $G_-$ take
the form
\begin{equation}
F_-(\tau)=Q(\lambda(\tau)),\quad G_-(\tau)=-C-Q(1-\lambda(\tau)),
\label{eq:pointwiserep8.01}
\end{equation}
where $Q$ is a polynomial of degree $\le k-1$, with $Q(0)=0$. Moreover,
from the second relation in \eqref{eq:pointwiserep8.01} we find that
$Q(1)=-\bar C$. It follows from the relation \eqref{eq:pointwiserep8.01} that
\emph{the sequence $\{C,a_{-n},b_{-n}\}_{n>0}$ is an exceptional sequence
of degree $\le k-1$}. We have now established that the two-sided sequence
$\{a_0,a_n,b_n\}_{n\ne0}$ is exceptional of degree $\le k$, as claimed.
\end{proof}

\subsection{Mapping properties of the modular $\lambda$
function}
\label{subsec:lambda-map}
The proof of Proposition \ref{prop:shin2} would get more streamlined
if we would use the mapping properties of the modular $\lambda$ function.
Let us first get acquainted with the following transformations on the
upper half-plane $\Hyp$:
\begin{equation}
\Ttrans(\tau):=\tau+1,\quad \Strans(\tau):=-1/\tau,\quad
\Strans^\ast:=1/\bar\tau.  
\label{eq:TSS*}  
\end{equation}
Of these, $\Strans$ and $\Ttrans$ are holomorphic automorphisms, while
$\Strans^\ast$ is the conjugate-holo\-morphic automorphism corresponding to
reflection in the circle $\Te$.
Starting from the fundamental domain $\calD_\Theta$ given by
\begin{equation}
\calD_\Theta:=\{\tau\in\Hyp:\,|\re\,\tau|<1,\,\,\,|\tau|>1\}, 
\label{eq:shin1.5}
\end{equation}
we form its ``double copy''
\begin{equation}
\calD_{2\Theta}=\calD_\Theta\cup\Strans^\ast(\calD_\Theta)\cup\Te_+=
\calD_\Theta\cup\Strans(\calD_\Theta)\cup\Te_+,
\end{equation}
which is a fundamental domain in its own right. It can be described by the
formula
\[
\calD_{2\Theta}=\big\{\tau\in\Hyp:\,|\re\tau|<1,\,\,\,
|\tau+\tfrac12|>\tfrac12,\,\,\,|\tau-\tfrac12|>\tfrac12\big\}.
\]
We introduce the notation
\[
\C_{\re<\frac12}:=\{z\in\C:\,\re\, z<\tfrac12\},\quad
\C_{\re>\frac12}:=\{z\in\C:\,\re\, z>\tfrac12\},
\]
and, for a given real number $\alpha$, we put
\begin{equation}
\calL_\alpha:=\{z\in\C:\,\re\, z=\alpha\},\quad
\calL_\alpha^+:=\{z\in\Hyp:\,\re\, z=\alpha\},
\label{eq:half-lines1}
\end{equation}
for the associated vertical line and half-line, respectively.

\begin{prop}
\label{prop:lambda-classical}  
The modular $\lambda$ function maps conformally the domain
$\calD_{2\Theta}$ onto the slit plane
$\C\setminus(]-\infty,0]\cup[1,+\infty[)$. In addition, it maps the
subdomain $\calD_\Theta$ onto the slit half-plane
$\C_{\re<\frac12}\setminus]-\infty,0]$, while it maps $\Strans(\calD_{\Theta})$ onto
the slit half-plane $\C_{\re>\frac12}\setminus [0,+\infty[$. The semicircle
$\Te_+$ is mapped onto the line $\calL_{\frac12}$ in such a way that as
$\tau\in\Te_+$ increases its argument, the point $\lambda(\tau)\in\calL_{\frac12}$
moves downward in the negative imaginary direction. The half-lines
$\calL_{-1}^+$ and $\calL_{1}^+$ are both mapped onto the slit
$]-\infty,0[$, while the edges $\pm1$ are mapped to infinity and the point
$+\imag\infty$ is mapped to $0$. Similarly, the semicircles
$\Strans(\calL_{-1}^+)$ and $\Strans(\calL_{1}^+)$ are both mapped onto the slit
$]1,+\infty[$. Finally, the imaginary half-line $\calL_0^+$ is mapped by $\lambda$ 
onto the interval $]0,1[$, where the endpoints $+\imag\,\infty$ and $0$ are 
mapped onto $0$ and $1$, respectively.
\end{prop}

The mapping properties of the modular $\lambda$ function presented in
Proposition \ref{prop:lambda-classical} are classical.
An elementary exposition may be found in \cite{bh2}.

We shall also need to approximate $\lambda(\tau)$ for
$\tau\in \calD_{2\Theta}$ near the cusps at $\pm1$ and $+\imag\infty$.
We recall the Fourier representation formula \eqref{eq:lambdaf2}
\begin{equation}
\lambda(\tau)=
\sum_{n=1}^{+\infty}\hat\lambda(n)\,\e^{\imag\pi n\tau},\qquad
\tau\in\Hyp,
\label{eq:lambdaf2.01}
\end{equation}
where the coefficients are integers, the first three being 
$\hat\lambda(1)=16$,
$\hat\lambda(2)=-128$, and $\hat\lambda(3)=704$. In particular, we have that
\begin{equation}
\lambda(\tau)=16\,\e^{\imag\pi\tau}-128\,\e^{\imag2\pi\tau}
+\Ordo(\e^{-3\pi\im\,\tau})
\label{eq:lambdaf2.02}
\end{equation}
holds uniformly provided that $\im\,\tau\ge1$. In view of the $2$-periodicity
and the functional properties \eqref{eq:lambdaproperties1.1}, it follows
that
\begin{equation}
\lambda(\tau)=1-\frac{1}{\lambda(-1/(\tau+1))}
=1-\frac{1}{\lambda(-1/(\tau-1))},\qquad\tau\in\Hyp.
\label{eq:lambdaproperty2}
\end{equation}
In view of \eqref{eq:lambdaf2.02}, the functional property
\eqref{eq:lambdaproperty2} gives that
\begin{multline}
\lambda(\tau)=1-\frac{1}{\lambda(-1/(\tau+1))}
\\
=1-\frac{1}{16\,\e^{-\imag\pi/(\tau+1)}-128\,\e^{-\imag2\pi/(\tau+1)}
  +\Ordo(\e^{-3\pi\im\,\tau/|\tau+1|^2})}
\\
=1-\frac{\e^{\imag\pi/(\tau+1)}}{16(1-8\,\e^{-\imag\pi/(\tau+1)}
  +\Ordo(\e^{-2\pi\im\,\tau/|\tau+1|^2}))}
\\=-\frac{1}{16}\,\e^{\imag\pi/(\tau+1)}+\frac12
+\Ordo(\e^{-\pi\im\,\tau/|\tau+1|^2}),
\label{eq:lambdaproperty3}
\end{multline}
uniformly provided that $\tau\in\Hyp$ lies in the horocyclic disk
$|\tau+1-\frac{\imag}{2}|\le\frac12$, which amounts to the condition
$\im\,\tau\ge|\tau+1|^2$. Moreover, by $2$-periodicity, the same applies near
the cusp at $1$:
\begin{equation}
\lambda(\tau)=\lambda(\tau-2)=1-\frac{1}{\lambda(-1/(\tau-1))}
=-\frac{1}{16}\,\e^{\imag\pi/(\tau-1)}+\frac12
+\Ordo(\e^{-\pi\im\,\tau/|\tau-1|^2}),
\label{eq:lambdaproperty4}
\end{equation}
uniformly provided that $\tau\in\Hyp$ lies in the horocyclic disk
$|\tau-1-\frac{\imag}{2}|\le\frac12$. 
A similar argument applies to the derivative $\lambda'$, and gives that
\begin{multline}
\lambda'(\tau)=\frac{\lambda'(-1/(\tau+1))}{(\tau+1)^2(\lambda(-1/(\tau+1)))^2}
=\frac{\imag\pi}{16(\tau+1)^2}\,\e^{\imag\pi/(\tau+1)}
+\Ordo(|\tau+1|^{-2}\e^{-\pi\im\,\tau/|\tau+1|^2}),
\label{eq:lambda'property3}
\end{multline}
again uniformly provided that $\tau\in\Hyp$ lies in the horocyclic disk
$|\tau+1-\frac{\imag}{2}|\le\frac12$, and this extends by $2$-periodicity
to the translated horocyclic disk $|\tau-1-\frac{\imag}{2}|\le\frac12$ as well.

Since the modular $\lambda$ function
maps conformally $\calD_{2\Theta}$ onto the slit plane
$\C\setminus(]-\infty,0]\cup[1,+\infty[)$, we can use the approximations
\eqref{eq:lambdaproperty3} and \eqref{eq:lambdaproperty4} to estimate the
inverse mapping
\[
\lambda^{-1}:\,\C\setminus(]-\infty,0]\cup[1,+\infty[)\to\calD_{2\Theta}.
\]
If $\zeta\in\C\setminus(]-\infty,0]\cup[1,+\infty[)$ has large modulus
$|\zeta|$, then we try to solve the equation $\lambda(\tau)=\zeta$ with
$\tau\in\calD_{2\Theta}$, and realize that only $\tau$ near the cusps at
$\pm1$ are possible. The inverse formulae are, for
$\zeta\in\C\setminus(]-\infty,0]\cup[1,+\infty[)$ with big $|\zeta|$,
\begin{equation}
\label{eq:inverse_lambda1}  
\lambda^{-1}(\zeta)=-1+\frac{\imag\pi}{\log(8-16\zeta)+\Ordo(|\zeta|^{-1})},
\qquad \im\,\zeta\ll-1,
\end{equation}
and
\begin{equation}
\lambda^{-1}(\zeta)=1+\frac{\imag\pi}{\log(8-16\zeta)+\Ordo(|\zeta|^{-1})},
\qquad \im\,\zeta\gg 1,
\label{eq:inverse_lambda2}
\end{equation}
respectively.

\section{The Schwarz transform and its consequences}
\label{sec:Schwarz}

\subsection{The Schwarz transform based on reflection
in the circle}
We introduce a mapping from the $2$-periodic harmonic functions in $\Hyp$
to the holomorphic functions in $\Hyp$, which we will call the
\emph{Schwarz transform} $\shin$ as it is based on reflection in the
semi-circle
\[
\Te_+:=\{\tau\in\Hyp:\,|\tau|=1\}. 
\]

\begin{defn}
Suppose $h:\,\Hyp\to\C$ is harmonic, so that the function can be split
uniquely as
\[
h(\tau)=h(\imag)+h_+(\tau)+h_-(\tau),
\]
where $h_+:\,\Hyp\to\C$ is holomorphic with $h_+(\imag)=0$ and
$h_-:\,\Hyp\to\C$ is conjugate-holomorphic with $h_-(\imag)=0$.
The \emph{Schwarz transform} of $h$, written $\shin h$, is the
holomorphic function
\[
\shin\, h(\tau):=h(\imag)+h_+(\tau)+h_-\Big(\frac{1}{\bar\tau}\Big),
\qquad\tau\in\Hyp.
\]  
\end{defn}

It is natural to ask when we can recover the harmonic function $h$ from
knowing its Schwarz transform. We may immediately recover partial information
about $h$. 

\begin{prop}
If $h:\Hyp\to\C$ is harmonic, then the Schwarz transform $\shin\, h$
is a well-defined holomorphic function on $\Hyp$, and
\[
\shin\,h(\tau)=h(\tau),\qquad\tau\in\Te_+.
\]
\label{prop:shin1}
\end{prop}  

\begin{proof}
The mapping $\tau\mapsto1/\bar\tau$ is a conjugate-holomorphic automorphism of
$\Hyp$, which makes $\shin\,h$ well-defined and holomorphic, since the
composition of two conjugate-holomorphic functions is holomorphic. Moreover,
since $1/\bar\tau=\tau$ holds for $\tau\in\Te_+$, $\shin\,h$ and $h$
coincide on $\Te_+$.
\end{proof}

Clearly some additional information is needed,
as can be seen from some elementary examples. The following proposition gives
appropriate conditions for uniqueness. The formulation involves the domain
$\calD_\Theta$ 
which may be considered to be the fundamental domain for the so-called Theta
group. See Figure \ref{fig:Dtheta} for an illustration.

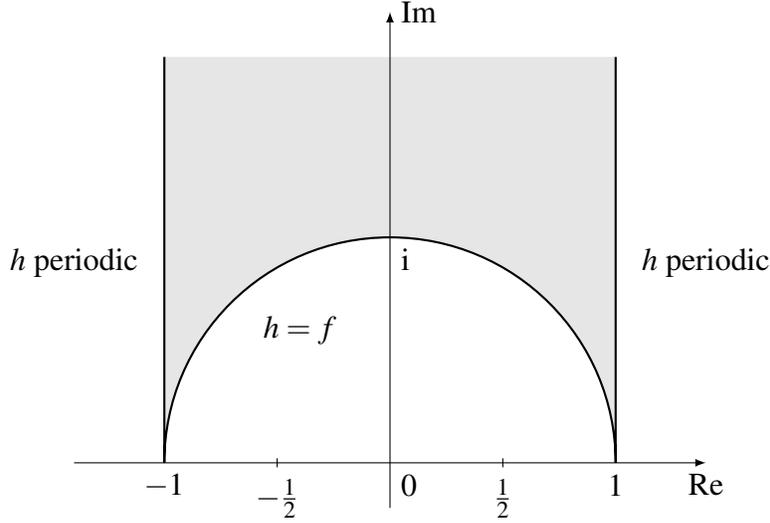
\begin{figure}
\begin{tikzpicture}[scale=3]
\draw[thick, fill=gray!20] (1,1.8) -- (0:1) arc (0:180:1)
 -- (-1,1.8);
\draw(-1.4,1) node[below]{$h$ periodic};
\draw(1.4,1) node[below]{$h$ periodic};
\draw(-0.4,.7) node[below]{$h=f$};
\draw[-latex] (-1.4,0) -- (1.4,0)node[below]{Re};
\draw[-latex] (0,-.2) -- (0,2.0)node[right]{Im};
\path(-1,0) --node[below, pos=0]{$-1$}
node[below right, pos=.5]{0}node[below, pos=1]{1} (1,0)
(0,1)node[below right]{$\imag$};
\draw(-.5,.02)--(-.5,-.02)
node[below]{$-\frac{1}{2}$}(.5,.02)--(.5,-.02)node[below]{$\frac{1}{2}$};
\end{tikzpicture}
\caption{The domain $\mathcal{D}_\Theta$ (in gray) and the corresponding
Dirichlet problem.}
\label{fig:Dtheta}
\end{figure}

\begin{prop}
Suppose $h:\Hyp\to\C$ is harmonic and $2$-periodic, so that
$h(\tau+2)=h(\tau)$ holds. If in addition
\[
|h(\tau)|=\ordo(\im\tau)\quad\text{as}
\,\,\,\,\im\tau\to+\infty,
\]
and if
\[
|h(\tau)|=\ordo\big(\exp(\pi/\im\tau)\big)\quad\text{as}
\,\,\,\,\calD_\Theta\ni\tau\to\pm1,
\]
the function $h$ is uniquely determined by its Schwarz transform $\shin\, h$.
\label{prop:shin2}
\end{prop}  


\begin{proof}
We consider the Riemann surface $\mathscr{S}_\Theta$ obtained by gluing
together the vertical linear boundary segments of the fundamental domain
$\calD_\Theta$ to form a cylindrical surface whose boundary consists of the
circle $\T_+$, the points $\pm1$ which are identified with one another,
and the isolated boundary point $+\imag\infty$.
The assumptions on $h$ entail that it is a harmonic function on
$\mathscr{S}_\Theta$.
The growth condition $|h(\tau)|=\ordo(\im\tau)$ as $\tau$ approaches
$+\imag\infty$ means that $h$ has a removable singularity at $+\imag\infty$,
and hence it extends harmonically across that point. 

We need to show that $h$ is uniquely determined by its Schwarz transform
$\shin\,h$. By linearity, it is enough to show that if $\shin\,h=0$ holds, then
we must have $h=0$ as well. By Proposition \ref{prop:shin1}, the assumption
$\shin\,h=0$ entails that $h$ vanishes on $\Te_+$. On the Riemann surface
$\mathscr{S}_\Theta\cup\{+\imag\infty\}$, we then know that $h$ is harmonic
and extends continuously to vanish along $\Te_+$, which constitutes most of
the boundary. The only boundary point not included is the cusp point $\pm1$.
In order to model the cusp in $\mathscr{S}_\Theta$ better, we consider the
alternative realization of in the form of
\begin{equation}
\calD_\Theta^{+}:=\big\{\tau\in\Hyp:\,0<\re\,\tau<2,\,\,|\tau|>1,\,\,
|\tau-2|>1\big\},  
\end{equation}
where the surface $\mathscr{S}_\Theta$ again results from gluing together
the vertical linear segments. On $\calD_\Theta^{+}$, the cusp is located at
the boundary point $\tau=1$. Let $\psi:\Hyp\to\Hyp$ be the involutive
M\"obius mapping
\[
\psi(\tau):=\frac{\tau-2}{\tau-1},
\]
which maps $\calD_\Theta^{+}$ onto itself, and the cusp at $1$ gets sent to
$+\imag\infty$, while $+\imag\infty$ gets mapped to $1$. We note in addition
that the circular parts where $|\tau|=1$ and where $|\tau-2|=1$ get sent to the
vertical linear segments of the boundary of $\calD_\Theta^{+}$.  
This means that the vanishing of $h$ on $\Te_+$ entails that the harmonic
function $h\circ\psi(\tau)$ vanishes for $\tau\in\partial\calD_\Theta^+$
with $\re\,\tau\in\{0,2\}$. Moreover, the growth bound on $h$ near the cusp
amounts to having
\begin{equation}
h\circ\psi(\tau)=\ordo\big(\exp(\pi\im\,\tau)\big)\quad \calD_\Theta^{+}\ni
\tau\to+\imag\infty.   
\end{equation}
By a version of the classical Phragm\'en-Lindel\"of principle, the permitted
growth in the strip is insufficient to allow the function to be unbounded near
the boundary point $+\imag\infty$. In other words, $h$ must be bounded in
a neighborhood of the cusp at $1$ within the domain $\calD_\Theta^+$. By
periodicity, then, $h$ is bounded in a neighborhood of the cusp at $\pm1$
inside $\calD_\Theta$. Since $h$ is continuous at all other boundary points of  
$\mathscr{S}_\Theta$, including the removable singularity at $+\infty$, it
follows that we may apply the maximum principle to the real and imaginary
parts of $h$ separately (see, e.g., Tsuji \cite{tsuji}). The result is that
$h=0$ on the surface $\mathscr{S}_\Theta\cup\{+\imag\infty\}$, in particular on
$\mathscr{D}_\Theta$. But $h$ is assumed harmonic in $\Hyp$, and vanishing
on the open set $\calD_\Theta$ forces $h$ to vanish throughout $\Hyp$.
\end{proof}

The harmonic $2$-periodic functions on $\Hyp$ have a characterization in terms
of harmonic Fourier series, provided we have the appropriate growth bound
in the imaginary direction. With some rather specific growth bounds, this
topic has been covered back in Propositions \ref{prop:function->coeff0} and
\ref{prop:function->coeff1.0}. 

\begin{prop}
\label{prop:function->coeff2.0}
Suppose $h:\Hyp\to\C$ is harmonic and $2$-periodic, i.e., $h(\tau+2)=h(\tau)$
holds, and enjoys the growth estimate
\[
|h(\tau)|=\ordo(\im\tau)\quad\text{as}
\,\,\,\,\im\tau\to+\infty.
\]
Then $h$ has a harmonic Fourier series expansion
\[
h(\tau)=\sum_{n\in\Z}c_n\,e_n(\tau),\qquad\tau\in\Hyp,
\]
where the functions $e_n$ are given by \eqref{eq:harmext2} for $n\ge0$ and
\eqref{eq:harmext3} for $n<0$, and the
coefficients grow subexponentially: $|c_n|=\Ordo(\exp({\epsilon|n|}))$ as
$|n|\to+\infty$.
\end{prop}

\begin{proof}
The periodicity and the growth bound on $h$ entail that $h(\tau)=\Ordo(1)$
as $\im\tau\to+\infty$, by the Phragm\'en-Lindel\"of principle for harmonic
functions. But the
by periodicity and compactness, $h(\cdot+\imag\epsilon)$ is uniformly bounded
in $\Hyp$ for any $\epsilon>0$. Proposition \ref{prop:function->coeff0}
gives that there are coefficients $c_n(\epsilon)$ with
$c_n(\epsilon)=\Ordo(1)$ as $|n|\to+\infty$ such that
\[
h(\tau+\imag\epsilon)=\sum_{n\in\Z}c_n(\epsilon)\,e_n(\tau),\qquad\tau\in\Hyp,
\]
and the proof of Proposition \ref{prop:function->coeff0} also shows that
$c_n(\epsilon)=\e^{-\pi\epsilon |n|}c_n$ holds, where $c_n$ are the coefficients
in the expansion of $h$. Since $c_n(\epsilon)$ is bounded for each
$\epsilon>0$, it now follows that the coefficients $c_n$ grow subexponentially.
\end{proof}

\begin{prop}
Suppose $h:\Hyp\to\C$ is harmonic and represented by a harmonic Fourier
series
\[
h(\tau)=\sum_{n\in\Z}c_n\,e_n(\tau), \qquad \tau\in\Hyp,
\]
where the coefficients $c_n$ grow subexponentially. Then the Schwarz transform
of $h$ is a holomorphic hyperbolic Fourier series given by
\[
\shin\,h(\tau)=\sum_{n\in\Z_{\ge0}}c_n\,e_n(\tau)+\sum_{n\in\Z_{<0}}
c_n\,e_n(1/\bar\tau)  ,\qquad\tau\in\Hyp.
\]
\label{prop:shinexpansion1}
\end{prop}

\begin{proof}
We put
\[
h_+(\tau)=\sum_{n\in\Z_{\ge0}}c_n\,e_n(\tau)-\sum_{n\in\Z_{\ge0}}c_n\,e_n(\imag)
\]
and
\[
h_-(\tau)=\sum_{n\in\Z_{<0}}c_n\,e_n(\tau)-\sum_{n\in\Z_{<0}}c_n\,e_n(\imag),
\]
so that $h_+:\Hyp\to\C$ is holomorphic while $h_-:\Hyp\to\C$ is
conjugate-holomorphic, with $h_+(\imag)=h_-(\imag)=0$. The decomposition
\[
h(\tau)=h(\imag)+h_+(\tau)+h_-(\tau),\qquad\tau\in\Hyp,
\]
then holds, and by the definition of the Schwarz transform, we have
\[
\shin\,h(\tau)=h(\imag)+h_+(\tau)+h_-(1/\bar\tau)=
\sum_{n\in\Z_{\ge0}}c_n\,e_n(\tau)+\sum_{n\in\Z_{<0}}
c_n\,e_n(1/\bar\tau)  ,\qquad\tau\in\Hyp,
\]
as claimed.
\end{proof}

\subsection{Solution scheme for the inverse Schwarz
transform problem}
\label{subsec:Solutionscheme-gen}
It remains to decide when we can solve the inverse problem:
\emph{To find for a given holomorphic function $f:\Hyp\to\C$
an associated harmonic function $h:\Hyp\to\C$ which is $2$-periodic and
has $\shin\,h=f$}. As it turns out, we can always do this. But it is not
always possible to do this in the uniqueness regime of Proposition
\ref{prop:shin2}, simply because any solution must have $h=f$ on $\Te_+$,
and the given $f$ need not obey the boundary growth bound at $\pm1$, which
means that we cannot expect $h$ to do so either. We follow an algorithm to
define $h$, where the first step is the only one where there is some freedom. 
We call it the \emph{harmonic extension algorithm}.

\medskip

\noindent{\sc Step I}: \emph{The Dirichlet problem in $\calD_\Theta$}.
We find a harmonic function $h_0:\calD_\Theta\to\C$
which solves a Dirichlet boundary problem along $\Te_+$, namely $h_0=f$ on
$\Te_+$, and has periodic boundary conditions along the linear boundary of
$\calD_\Theta$:
\[
\forall y>0:\,\,\,h_0(-1+\imag y)=h_0(1+\imag y)\quad\text{and}\quad
\partial_x h_0(-1+\imag y)=\partial_x h_0(1+\imag y).  
\]  
In addition, $h_0(\tau)=\Ordo(1)$ as $\im\,\tau\to+\infty$ is required.
In terms of the alternative nome coordinate $q=\e^{\imag\pi \tau}$, 
this amounts to an ordinary Dirichlet problem for a bounded simply connected 
domain $\hat\calD_\Theta$ with real-analytic Jordan boundary curve, except
for a single cusp at $q=-1$ corresponding to the points $\tau=\pm1$.
After all, the periodic boundary condition means that in the $q$-domain
$\hat\calD_\Theta$, $h_0$ extends harmonically across the segment $]-1,0[$, so
we may as well include the segment in $\hat\calD_\Theta$. Being simply connected,
the domain $\hat\calD_\Theta$ is conformally equivalent to the upper half-plane
$\Hyp$, and the cusp point $q=-1$ may be sent to the boundary point at
infinity of $\Hyp$. \emph{Proposition \ref{prop:Dirichlet-general} now
guarantees that such a harmonic function $h_0$ exists, no matter how quickly
the function $f$ grows near the cusp $\pm1$ along $\Te_+$}.
If we want growth control on $h_0$ for some type of growth bound on $f$,
we can use for instance Propositions \ref{prop:loggrowth} and
\ref{prop:powergrowth}.

\medskip

\noindent{\sc Step II}: \emph{Periodic extension}.
We extend $h_0$ periodically from $\calD_\Theta$ to the domain
\begin{equation}
\Omega_0:=\Hyp\setminus\bigcup_{j\in\Z}\bar\D(2j,1).
\label{eq:Omega_0.01}
\end{equation}
Then $h_0(\tau+2)=h_0(\tau)$ holds for all $\tau\in\Omega_0$, and the extended
function, still called  $h_0$, is harmonic in $\Omega_0$ and bounded at
$+\imag\infty$. These properties are immediate if we express the Dirichlet
problem in terms of the nome variable $q$.

\medskip

\noindent{\sc Step III}: \emph{Schwarz reflection}. We now use the fact that
the Dirichlet boundary data on $\Te_+$ comes from a holomorphic function $f:
\Hyp\to\C$. Indeed, we consider the function $g_0:=h_0-f$ which by the above
Step II is harmonic in the periodic domain $\Omega_0$. The periodicity of
$h_0$ does not carry over to $g_0$, because the function $f$ need not be
periodic, but a good thing is that $g_0=0$ holds on $\Te_+$.
This is good because it puts us in the setting of the Schwarz reflection
principle (see, e.g., \cite{Ahlfors}, pp. 172-173, in the instance of a
linear boundary), which asserts that $g_0$ extends harmonically across
$\Te_+$, and the extension has the form 
\begin{equation}
g_0(\tau)=-g_0(1/\bar\tau),
\label{eq:Schwarzreflection}
\end{equation}
provided that one of the sides makes sense. So, if $1/\bar\tau\in \Omega_0$,
the left-hand side of \eqref{eq:Schwarzreflection} gets defined by 
the right-hand side.
If we put (note that $\Omega_0\cap\D_\e=\Omega_0$ trivially)
\begin{equation}
\Omega_{1}':=(\Omega_0\cap\D_\e)\cup\Te_+\cup\,\Strans^\ast(\Omega_0\cap\D_\e),
\label{eq:Omega0'}
\end{equation}
where we recall the notation $\Strans^\ast(\tau)=1/\bar\tau$,  it is clear that
$\Omega_1'$ is open and connected, and that $\Omega_0\subset\Omega_1'$.
Moreover, the function $g_0$ extends harmonically to $\Omega'_1$. But then 
$h_0=g_0+f$ gets extended as a harmonic function to $\Omega'_1$ as well. 

\medskip

\noindent{\sc Step IV}: \emph{Iterative application of the steps II and III}. 
The domain $\Omega_0$ is periodic, and it was extended by forming $\Omega'_1$. 
We note that the enlargement from $\Omega_0$ to $\Omega_1'$ was localized
to the semidisk $\bar\D\cap\Hyp$. 
We may now use the periodicity of $h_0$ to further extend the function $h_0$
to the periodized domain
\begin{equation}
\Omega_{1}:=\big\{\tau\in\Hyp:\,\,\tau-2k\in\Omega'_{1}
\,\,\,\,\text{holds for some }\,\,k\in\Z \big\}.
\label{eq:Omega1}
\end{equation}
Then clearly $\Omega_0\subset\Omega_1'\subset\Omega_1$ holds, and we let 
$h_1$ denote the resulting $2$-periodic harmonic function 
$h_1:\Omega_1\to\C$ with restriction $h_1=h_0$ on $\Omega_1'$. 
After that, we again apply Step III based on Schwarz reflection to extend
$h_1$ harmonically to the bigger domain 
\[
\Omega_{2}':=(\Omega_1\cap\D_\e)
\cup\Te_+\cup\,\Strans^\ast(\Omega_1\cap\D_\e),
\]
After all, we proceed as before and write $g_1:=h_1-f$, which is harmonic in 
$\Omega_1$ since both $h_1$ and $f$ are, and the Schwarz reflection formula
\eqref{eq:Schwarzreflection} remains valid:
\begin{equation}
g_1(\tau)=-g_1(1/\bar\tau),
\label{eq:Schwarzreflection2}
\end{equation}
In view of
\eqref{eq:Schwarzreflection2}, $g_1$ extends harmonically to $\Omega_2'$,
and hence $h_1=g_1+f$ extends harmonically to $\Omega_2'$ as well.
At the next level, we use the periodicity of $h_1$ to further extend
the function $h_1$ to the periodized domain
\[
\Omega_{2}:=\big\{\tau\in\Hyp:\,\,\tau-2k\in\Omega'_{2}
\,\,\,\,\text{holds for some }\,\,k\in\Z \big\}.
\]
We let $h_2$ denote the resulting $2$-periodic function $h_2:\Omega_2\to\C$
with restriction $h_2=h_1$ on $\Omega_2'$.
These are just the first few iteration steps of a general iteration procedure.
Suppose that we have, for a given $j=1,2,3,\ldots$, a $2$-periodic harmonic
function $h_{j}$ on a $2$-periodic domain $\Omega_{j}$. We then extend it
to a $2$-periodic harmonic function $h_{j+1}$ on a larger $2$-periodic domain
$\Omega_{j+1}$ in the following two-step fashion. 
We form $g_{j}:=h_j-f$, which is harmonic in $\Omega_j$. This function vanishes
on the semicircle $\Te_+$, so that by Schwarz reflection,
\begin{equation}
g_j(\tau)=-g_j(1/\bar\tau),
\label{eq:Schwarzreflectionj}
\end{equation}
which extends $g_j$ harmonically to the domain
\begin{equation}
\Omega_{j+1}':=(\Omega_j\cap\D_\e)\cup\Te_+\cup\,\Strans^\ast(\Omega_j\cap\D_\e).
\label{eq:Omegaj'}
\end{equation}
Consequently, $h_j=g_j+f$ extends harmonically to $\Omega_{j+1}'$ as well. In a
second step, the harmonic function $h_j$ gets extended by $2$-periodicity to
the larger domain
\begin{equation}
\Omega_{j+1}:=\big\{\tau\in\Hyp:\,\,\tau-2k\in\Omega'_{j+1}
\,\,\,\,\text{holds for some }\,\,k\in\Z \big\},
\label{eq:Omegaj+1}
\end{equation}
and the extension gets to be called $h_{j+1}$, and it is harmonic and
$2$-periodic on $\Omega_{j+1}$. 

\medskip

We need to make some remarks concerning the structure of the sets $\Omega_j$
for $j=0,1,2,\ldots$. The statement involves the concept of a 
cuspidal hyperbolic polygon.

\begin{defn}
An open subset $\Omega$ of the upper half-plane $\Hyp$ is a 
\emph{cuspidal hyperbolic polygon} if  the complement $\Hyp\setminus\Omega$ 
is a finite or countable union of disjoint relatively closed disks or
half-planes, provided that the boundary of each such disk or half-plane is 
hyperbolic geodesic. 
\end{defn}

We remark that the hyperbolic geodesics are either semicircles centered    
along the real line, or vertical half-lines.  

\begin{figure}[t!]
\centering  
\includegraphics[width=1\linewidth]{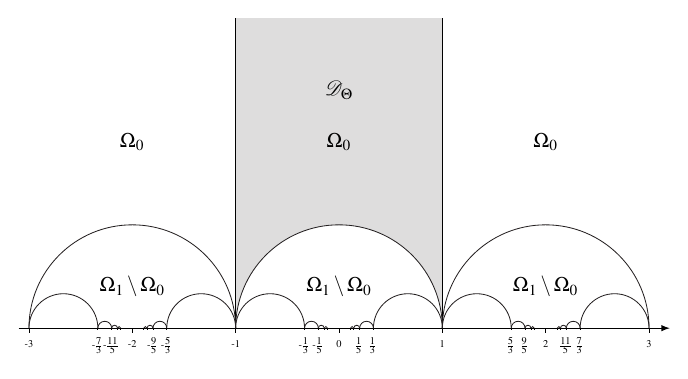}
\caption{The domain $\Omega_0$ is marked as well as $\Omega_1\setminus
\Omega_0$. The fundamental domain $\calD_\Theta$ is shaded in  gray.}  
\label{fig-2.1}
\end{figure}

\begin{prop}
For each $j=0,1,2,\ldots$, the domains $\Omega_j$ and $\Omega_{j+1}'$ are 
cuspidal hyperbolic polygons, and we have the containments
\[
\Omega_j\subset\Omega_{j+1}'\subset\Omega_{j+1}.
\]  
\end{prop}

\begin{proof}
We read off from the definition of $\Omega_{j+1}$ that the containment 
$\Omega_{j+1}'\subset\Omega_{j+1}$ holds, so we need only be concerned with
the remaining containment $\Omega_j\subset\Omega_{j+1}'$.
We resort to induction to deal with the outstanding issues.
We first devote attention to the containment $\Omega_j\subset\Omega_{j+1}'$.
For $j=0$ the assertion is clearly true. Indeed, inspection of
\eqref{eq:Omega0'}
gives that $\Omega_1'\cap\D_\e=\Omega_0\cap\D_\e$, while clearly 
$\Omega_0\cap\D=\emptyset$. To proceed with the induction, we assume that
the containment $\Omega_j\subset\Omega_{j+1}'$ holds for all $j=0,\ldots,j_0$,
and aim to show that the containment holds for $j=j_0+1$ as well.
We show that $\Omega_{j_0+1}\subset\Omega_{j_0+2}'$ holds by establishing
separately that (a) $\Omega_{j_0+1}\cap\D_\e\subset\Omega_{j_0+2}'\cap\D_\e$,
(b) $\Omega_{j_0+1}\cap\D\subset\Omega_{j_0+2}'\cap\D$, and that
(c) $\Omega_{j_0+1}\cap\Te_+\subset\Omega_{j_0+2}'\cap\Te_+$.

To arrive at the inclusion (a), we assume that $\tau\in\Omega_{j_0+1}\cap\D_\e$.
 By \eqref{eq:Omegaj'} for $j=j_0+1$,
it follows that $\tau\in\Omega_{j_0+2}'\cap\D_\e$, so that the containment holds
within the exterior disk $\D_\e$.  This finishes the verification of part (a).

To arrive at the inclusion (b), we assume that 
$\tau\in\Omega_{j_0+1}\cap\D$. Then in view of \eqref{eq:Omegaj+1} with 
$j=j_0$, we have that there exists a $k=k(\tau)\in\Z$ such that 
$\tau-2k\in\Omega'_{j_0+1}$. If $k=0$, then $\tau\in\Omega'_{j_0+1}\cap\D$
and by \eqref{eq:Omegaj'} with $j=j_0$ it follows that
\[
\Omega'_{j_0+1}\cap\D=\Strans^\ast(\Omega_{j_0}\cap\D_\e)
\subset\Strans^\ast(\Omega_{j_0+1}'\cap\D_\e)\subset
\Strans^\ast(\Omega_{j_0+1}\cap\D_\e),
\] 
where the first containment is a consequence of the induction hypothesis.
In particular, it follows that if $\tau\in\Omega_{j_0+1}\cap\D$ with associated
$k(\tau)=0$, then $\tau\in \Omega_{j_0+2}'\cap\D$, by \eqref{eq:Omegaj'}
for $j=j_0+1$. As for the remaining case when  $k\ne0$, then by the triangle 
inequality,
\[
|\tau-2k|\ge 2|k|-|\tau|>2k-1\ge1,
\]
so that by \eqref{eq:Omegaj'} with $j=j_0$, 
\[
\tau-2k\in\Omega_{j_0+1}'\cap\D_\e=\Omega_{j_0}\cap\D_\e.
\]
By the $2$-periodicity of $\Omega_{j_0}$, it follows that $\tau\in\Omega_{j_0}$,
and since $\tau\in\D$ was assumed, we obtain that 
\[
\tau\in\Omega_{j_0}\cap\D\subset\Omega_{j_0+1}'\cap\D\subset
\Omega_{j_0+1}\cap\D,
 \]
where the first containment follows from the induction hypothesis. In view
of the definition \eqref{eq:Omegaj'} with $j=j_0+1$, we find that 
$\tau\in\Omega_{j_0+2}'\cap\D$, as claimed. This finishes the verification of
(b). 
 
Finally, we turn to the inclusion (c).
To this end, we assume that $\tau\in\Omega_{j_0+1}\cap\Te_+$.
Since by \eqref{eq:Omegaj'} with $j=j_0+1$ we know that
$\Omega_{j_0+1}'\cap\Te_+=\Te_+$, the conclusion that $\tau\in
\Omega_{j_0+1}\cap\Te_+$ is immediate. This finishes the verification of part 
(c).
 
We turn to the remaining assertion that the domains $\Omega_j$ and 
$\Omega_{j+1}'$ are cuspidal hyperbolic polygons. First, this is true about 
$\Omega_0$ by inspection, and about $\Omega_1'$ by reflection.
The general case can be obtained by another induction argument, which is
left to the interested reader. The important matter is that each of the
two domains ($\Omega_j$ and $\Omega_{j+1}'$) has the points $1$ 
and $-1$ as endpoints of geodesic complementary disks, and this property
survives along the induction process, since all we do is Schwarz reflection
in $\Te_+$ and exploitation of the $2$-periodicity. 
\end{proof}  

We need to understand some symmetry properties of the domains $\Omega_j$,
and how they grow with $j$.  Let $\Rtrans^\ast(\tau):=-\bar\tau$ denote the 
reflection in the imaginary axis.

\begin{prop}
For $j=0,1,2,\ldots$, the domains $\Omega_j$ are $2$-periodic and symmetric with
respect to the imaginary axis: $\Rtrans^\ast(\Omega_j)=\Omega_j$ and 
$\Ttrans^2(\Omega_j)=\Omega_j$. Moreover, we have that
we have that 
\[
\sup\{\im\,\tau:\,\,\tau\in\Hyp\setminus\Omega_j\}=\frac{1}{2j+1},
\]
so that in particular, the union of all the domains $\Omega_j$ over
$j\in\Z_{\ge0}$ equals all of $\Hyp$.
\label{prop:Omega_j}
\end{prop}

\begin{proof}
The domain $\Omega_0$ is symmetric with respect to the imaginary axis. This
property gets inherited by $\Omega_1'$, since $\Strans^\ast(\Omega_0\cap\D_\e)$ 
is symmetric too, which we see from the commutation relation 
$\Rtrans^\ast\Strans^\ast=\Strans^\ast\Rtrans^\ast$. The periodicity operation
which produces $\Omega_1$ must then also preserve the symmetry, since 
$\Rtrans^\ast\Ttrans^2=\Ttrans^2\Rtrans^\ast$. The same arguments apply under
the two-step iteration which gives $\Omega_{j+1}$ from the smaller domain 
$\Omega_j$, and we find that all the domains $\Omega_j$ are symmetric under 
reflection in the imaginary axis. The fact that they are all $2$-periodic
follows from the definition of \eqref{eq:Omegaj+1}.

We turn to the assertion about the maximal imaginary part of a point from the 
complement $\Hyp\setminus\Omega_j$. For $j=0$,  the complement $\Hyp\setminus
\Omega_0$ contains the half-disk $\bar\D\cap\Hyp$ which contains the point
$\imag$ with imaginary part $1$. For $j>0$,  $\Hyp\setminus\Omega_j$ contains
a relatively closed half-disk whose boundary semicircle is a hyperbolic
geodesic and passes through  the points $1$ and $(2j-1)/(2j+1)$. It is a
circle centered at $2j/(2j+1)$ of radius $1/(2j+1)$, and it contains the point 
\[
\frac{2j}{2j+1}+\frac{\imag}{2j+1}.
\]
This point has the claimed maximal imaginary part $1/(2j+1)$. The fact that
this is the largest possible imaginary part can be obtained as a consequence
of Lemma 2 on p. 66 of \cite{rad}.
\end{proof}

\begin{cor}
The harmonic function $h_0$ which is defined in Step I as one solution to
the Dirichlet problem on $\calD_\Theta$ with boundary datum $f$ on $\Te_+$
and periodic conditions along the linear segments, extends harmonically to a
function $h_\infty:\Hyp\to\C$, which is $2$-periodic and has $h_\infty=f$
on $\Te_+$. Moreover, $h_\infty(\tau)=\Ordo(1)$ as $\im\,\tau\to+\infty$. 
\label{cor:Hyperext1}
\end{cor}

\begin{proof}
According to the iteration in Step IV, $h_0$ extends harmonically to a
function $h_j:\Omega_j\to\C$. Letting $j$ tend to infinity we find that
$h_j$ extends to a function $h_\infty$ which is defined and harmonic in
$\Hyp$, by Proposition \ref{prop:Omega_j}.
\end{proof}

\begin{cor}
\label{cor:HFSthm1}
The harmonic function $h_\infty:\Hyp\to\C$ of Corollary \ref{cor:Hyperext1}
has a harmonic Fourier series expansion
\[
h_\infty(\tau)=\sum_{n\in\Z}c_n\,e_n(\tau),\qquad \tau\in\Hyp,
\]
where the coefficients $c_n$ grow subexponentially. Moreover, its Schwarz
transform is $\shin\,h_\infty=f$, and consequently, the holomorphic function
$f:\Hyp\to\C$ has the holomorphic Fourier series expansion
\[
f(\tau)=\shin\,h_\infty(\tau)=c_0+\sum_{n\in\Z_{>0}}c_n\,e_n(\tau)+\sum_{n\in\Z_{<0}}
c_n\,e_n(1/\bar\tau)  ,\qquad\tau\in\Hyp.
\]
\end{cor}

\begin{proof}
In view of Corollary \ref{cor:Hyperext1} and Proposition
\ref{prop:function->coeff2.0}, the function $h_\infty$ has claimed harmonic
Fourier series expansion. The Schwarz transform $\shin\,h_\infty$ is
automatically holomorphic on $\Hyp$, and by Proposition \ref{prop:shin1} 
and Corollary \ref{cor:Hyperext1}, we obtain that $\shin\,h_\infty=h_\infty=f$
holds on $\Te_+$. But then the difference $\shin\,h_\infty-f$ is holomorphic on
$\Hyp$ and vanishes along the semicircle $\Te_+$, so by the uniqueness theorem
for holomorphic functions, $\shin\,h_\infty-f$ vanishes throughout $\Hyp$.
The claimed hyperbolic Fourier series expansion now follows from
Proposition \ref{prop:shinexpansion1}.
\end{proof}  

We may now write down the proof of Theorem \ref{thm:HFS-general.1.1}. 

\begin{proof}[Proof of Theorem \ref{thm:HFS-general.1.1}]
We first split the harmonic function $f:\Hyp\to\C$ as
\[
f(\tau)=f(\imag)+f_+(\tau)+f_-(\tau),\qquad\tau\in\Hyp,
\]
where $f_+:\Hyp\to\C$ is holomorphic while $f_-:\Hyp\to\C$ is
conjugate-holomorphic, both with $f_+(\imag)=f_-(\imag)=0$. We then apply the
holomorphic hyperbolic Fourier series expansion of Corollary \ref{cor:HFSthm1}
first to $f_+(\tau)$ and then to $f_-\circ\Rtrans^\ast(\tau)=f_-(-\bar\tau)$.
The result is the claimed harmonic hyperbolic Fourier series expansion of $f$. 
\end{proof}
  
In association with the solution algorithm which extends the harmonic function
$h_0$ uniquely to a harmonic function $h_\infty:\Hyp\to\C$, there are the
domains $\Omega_j$, which in a natural fashion suggests the introduction of
the concept of the height (or level) of a point $\tau\in\Hyp$.

\begin{defn}
\label{def:heightoftau}
The \emph{height} $\mathbbm{n}(\tau)$ of a point $\tau\in\Hyp$ is the
smallest value $j\in\Z_{\ge0}$ such that $\tau\in\Omega_j$. 
\end{defn}

\begin{cor}
We have that
\[
\mathbbm{n}(\tau)\le\frac{1}{2}+\frac{1}{2\,\im\,\tau},\qquad \tau\in\Hyp.
\]
\label{cor:n_estimate}
\end{cor}  

\begin{proof}
This is an immediate consequence of Proposition \ref{prop:Omega_j}.
\end{proof}

The average height is, however, much smaller. The following asymptotics was
found by Bondarenko, Radchenko, and Seip in Proposition 6.7 of \cite{bon}. 

\begin{prop}
We have that
\[
\frac12\int_{-1}^{1}\mathbbm{n}(t+\imag y)\,\diff t=
\frac{1}{\pi^2}\log^2\frac1y+\Ordo\Big(\log\frac1y\Big),
\]
as $y\to0^+$. 
\label{prop:n_estimate_ave}
\end{prop}



\subsection{The even Gauss map lifted to the upper half-plane}

In the algorithm presented in Subsection \ref{subsec:Solutionscheme-gen},
the following sequences of domains appear, $\Omega_j$ and $\Omega'_{j+1}$,
for $j=0,1,2,\ldots$, nested according to
\begin{equation}
\label{eq:}
\Omega_0\subset\Omega_1'\subset\Omega_1\subset\Omega_2'\subset\Omega_2\subset
\cdots.
\end{equation}
We recall the transformations $\Strans,\Strans^\ast:\Hyp\to\Hyp$, inversion
and reflection in $\Te_+$, respectively, given by $\Strans(\tau)=-1/\tau$
and $\Strans^\ast(\tau)=1/\bar\tau$.
We shall also need the mapping $\mathrm{mod}_2$, given by
\[
\mathrm{mod}_2(\tau):=\tau-2k,\qquad \tau\in\C,
\]
where the integer $k\in\Z$ is chosen such that $\tau'=\tau-2k$ lies in the
closed vertical strip
\[
\bar\calV:=\{\tau'\in\C:\,\,-1\le\re\,\tau'\le 1\}.
\]
This defines the value of $k$ uniquely unless $\re\,\tau'=\pm1$. In the
exceptional case that $\re\,\tau'=\pm1$, we are free to choose $k\in\Z$ such
that, e.g.,
$\re\,\tau'=-1$. 
We observe that 
then
\begin{equation}
\tau\in\Omega_0 \,\,\,\,
\Longleftrightarrow
\,\,\,\,
\mathrm{mod}_2(\tau)\in \calD_\Theta\cup\calL^{+}_{-1},
\label{eq:Omega1.00}
\end{equation}
where we recall the half-line notation $\calL^+_{-1}$ from
\eqref{eq:half-lines1}. The periodic extension \eqref{eq:Omegaj+1}
which defines $\Omega_{j+1}$ for $j=0,1,2,\ldots$ amounts to
\begin{equation}
\tau\in\Omega_{j+1} \,\,\,
\Longleftrightarrow
\,\,\,\,
\mathrm{mod}_2(\tau)\in \Omega_{j+1}'.
\label{eq:mod2.00}
\end{equation}
At the same time, it is clear from \eqref{eq:Omegaj'} that for
$j=0,1,2,\ldots$,
\begin{equation}
\tau\in\Omega_{j+1}'\setminus\Omega_j \,\,\,
\Longleftrightarrow
\,\,\,
\tau\in(\Hyp\cap\bar\D)\setminus\Omega_j\,\,\,\text{ and }\,\,\,
\Strans^\ast(\tau)\in(\Omega_j\cap\D_\e)\cup\Te_+.
\label{eq:Omega2}
\end{equation}
Moreover, by the reflection symmetry of $\Omega_j$ stated in Proposition
\ref{prop:Omega_j}, we know that
\begin{equation}
\Strans^\ast(\tau)\in(\Omega_j\cap\D_\e)\cup\Te_+\,\,\,
\Longleftrightarrow\,\,\,
\Strans(\tau)\in(\Omega_j\cap\D_\e)\cup\Te_+.
\label{eq:SorS'}
\end{equation}
If we view the $2$-periodicity of the domain $\Omega_j$ in the form
\[
\tau\in\Omega_j\,\,\,\Longleftrightarrow\,\,\,
\mathrm{mod}_2(\tau)\in\Omega_j,
\]
we may combine \eqref{eq:mod2.00} with \eqref{eq:Omega2} and arrive at
\begin{multline}
\tau\in\Omega_{j+1}\setminus\Omega_j \,\,\,
\Longleftrightarrow
\,\,\,
\mathrm{mod}_2(\tau)\in(\Hyp\cap\bar\D)\setminus\Omega_j
\,\,\,\text{ and }\,\,\,
\Strans^\ast\big(\mathrm{mod}_2(\tau)\big)\in(\Omega_j\cap\D_\e)\cup\Te_+.
\label{eq:Omega2.01}
\end{multline}
If we write
\[
\calV_+:=\{\tau\in\Hyp:\,\,-1<\re\,\tau< 1\}
\]
for the open vertical half-strip, then
$\Hyp\cap\bar\D=\calV_+\setminus\calD_\Theta$,
and this allows us to express the equivalence \eqref{eq:Omega2.01} in the form
\begin{equation}
\tau\in\Omega_{j+1}\setminus\Omega_j \,\,\,
\Longleftrightarrow
\,\,\,
\mathrm{mod}_2(\tau)\in\calV_+\setminus\Omega_j\,\,\,\text{ and }\,\,\,
\Strans^\ast\big(\mathrm{mod}_2(\tau)\big)\in(\Omega_j\cap\D_\e)\cup\Te_+.
\label{eq:Omega2.02}
\end{equation}
We now connect with the 
\emph{hyperbolically lifted conjugate Gauss-type map}
$\mathbbm{g}^\ast_2:\Hyp\to\calV_+\cup
\calL_{-1}^+\subset\Hyp$ given by
\begin{equation}
\mathbbm{g}^\ast_2(\tau)=\mathrm{mod}_2\circ\Strans^\ast
(\tau),\qquad\tau\in\Hyp.
\label{eq:g2map1*}
\end{equation}
Given a point $\tau\in\Hyp$, we consider the following \emph{fly-catcher
algorithm}, which
associates an algorithmic height number
$\mathbbm{n}^\ast_{\mathrm{alg}}(\tau)\in\Z_{\ge0}$ with each point $\tau\in\Hyp$.
As for notation, we write $\bar\calD_\Theta$ for the closure of the fundamental
domain $\calD_\Theta$.

\begin{defn} (The fly-catcher algorithm)
We begin with a point $\tau\in\Hyp$. 
  
\noindent{\sc Step I}: \emph{Apply $\mathrm{mod}_2$}. We form the point
$\tau_0:=\mathrm{mod}_2(\tau)\in\calV_+\cup\calL_{-1}^+$, and if $\tau_0\in
\bar\calD_\Theta$, we declare that $\mathbbm{n}^\ast_{\mathrm{alg}}(\tau)=0$,
and we stop the algorithm. On the other hand, in the remaining case
$\tau_0\in(\calV_+\cup\calL_{-1}^+)\setminus\bar\calD_\Theta=\Hyp\cap\D$, we
declare that
$\mathbbm{n}^\ast_{\mathrm{alg}}(\tau)>0$ and proceed to evaluate this number by
passing to the next step.

\noindent{\sc Step II}: \emph{Apply $\mathbbm{g}^\ast_2$}. We form the point
$\tau_1:=\mathbbm{g}^\ast_2(\tau_0)\in\calV_+\cup\calL_{-1}^+$, and if $\tau_1\in
\bar\calD_\Theta$, we declare that $\mathbbm{n}^\ast_{\mathrm{alg}}(\tau)=1$
and stop the algorithm.
In the remaining case
$\tau_1\in(\calV_+\cup\calL_{-1}^+)\setminus\bar\calD_\Theta=\Hyp\cap\D$,
we declare that $\mathbbm{n}^\ast_{\mathrm{alg}}(\tau)>1$ and proceed to evaluate
this number by passing to the next general iteration step.

\noindent{\sc Step III}: \emph{Apply $\mathbbm{g}^\ast_2$ iteratively}.
For $j\ge1$ we suppose that we have points $\tau_0,\ldots,\tau_{j}\in\Hyp\cap\D$
which we obtained earlier from $\tau\in\Hyp$,
and that it is determined that $\mathbbm{n}^\ast_{\mathrm{alg}}(\tau)>j$ holds.
We then proceed to form the point
$\tau_{j+1}:=\mathbbm{g}^\ast_2(\tau_j)\in\calV_+\cup\calL_{-1}^+$. Now, if
$\tau_{j+1}\in\bar\calD_\Theta$ holds, we declare that
$\mathbbm{n}^\ast_{\mathrm{alg}}(\tau)=j+1$, and stop the algorithm.
In the remaining case we must have instead
$\tau_{j+1}\in(\calV_+\cup\calL_{-1}^+)\setminus\bar\calD_\Theta
=\Hyp\cap\D$, and we then declare that $\mathbbm{n}^\ast_{\mathrm{alg}}(\tau)>j+1$
must hold. In this instance, we run Step III again. Finally, if Step III
never terminates, we declare that $\mathbbm{n}^\ast_{\mathrm{alg}}(\tau)=+\infty$.
\label{def:flycatcher}
\end{defn}

\begin{rem}
We note that in each iteration in Step III, we have that the points move 
ever further away from the real line:  
\[
\im\,\tau_{j+1}= \im\,\mathbbm{g}^\ast_2(\tau_j)=
\im\,\mathrm{mod}_2\circ\Strans^\ast(\tau_j)=\im\,\Strans^\ast(\tau_j)
=\im\frac{1}{\bar\tau_j}=\frac{\im\,\tau_j}{|\tau_j|^2}>\im\,\tau_j,
\]
since $\tau_j\in\Hyp\cap\D$ holds as long as the algorithm has not terminated.
\label{rem:im_part_growing}
\end{rem}

\begin{defn}
In the context of the above algorithm, suppose that for some $j\ge0$, 
$\mathbbm{n}^\ast_{\mathrm{alg}}(\tau)=j$ holds for
a point $\tau\in\Hyp$, which means that $\tau_k\in\Hyp\cap\D$ for all $k<j$ 
while $\tau_j\in\bar\calD_\Theta$ holds, we say that $\tau$ is a 
\emph{mesh point} if $\tau_j\in\partial\calD_\Theta$, where the boundary is
understood in the relative sense (inside $\Hyp$).
\label{def:setkapoint}
\end{defn}

\begin{rem}
If in Definition \ref{def:flycatcher} we have $j=0$, the requirement 
for $\tau$ to be a mesh point is just that $\tau_0\in\partial\calD_\Theta$.   
\end{rem}  

We need to compare the concepts of height of Definition \ref{def:heightoftau}
with the algorithmic height of the above fly-catcher algorithm.

\begin{prop}
Unless $\tau\in\Hyp$ is a mesh point, we have that the equality of height
and algorithmic height $\mathbbm{n}(\tau)=\mathbbm{n}^\ast_{\mathrm{alg}}(\tau)$.
Moreover, at a mesh point $\tau\in\Hyp$, we have the inequality
$\mathbbm{n}^\ast_{\mathrm{alg}}(\tau)\le\mathbbm{n}(\tau)$. As a consequence,
the fly-catcher algorithm always terminates:
$\mathbbm{n}^\ast_{\mathrm{alg}}(\tau)<+\infty$ at all $\tau\in\Hyp$. 
\label{prop:n=n*}
\end{prop}

\begin{proof}
This follows from a comparison of the procedures which define the two height
concepts, if we take into account the relation \eqref{eq:Omega2.01}.  
\end{proof}

We also need the \emph{hyperbolically lifted Gauss-type map}
$\mathbbm{g}_2:\Hyp\to\calV_+\cup\calL_{-1}^+$ defined by
\begin{equation}
\mathbbm{g}_2(\tau)=\mathrm{mod}_2\circ\Strans
(\tau),\qquad\tau\in\Hyp.
\label{eq:g2map1}
\end{equation}
The fly-catcher algorithm of Definition \ref{def:flycatcher} can be modified
so as to involve the mapping $\mathbbm{g}_2$ in place of $\mathbbm{g}_2^\ast$.
This then results in another height number $\mathbbm{n}_{\mathrm{alg}}(\tau)$.
Pleasantly, these concepts coincide:
\begin{equation}
\mathbbm{n}_{\mathrm{alg}}(\tau)=\mathbbm{n}^\ast_{\mathrm{alg}}(\tau),\qquad
\tau\in\Hyp.
\label{eq:n=n*}
\end{equation}
The reason is that closed fundamental domain $\bar\calD_\Theta$ is
$\Rtrans^\ast$-invariant, and the mapping properties $\Strans=\Strans^\ast\circ
\Rtrans^\ast$ as well as 
\[
\Rtrans^\ast\!\circ\mathrm{mod}_2=\mathrm{mod}_2'\circ\Rtrans^\ast,
\]
if $\mathrm{mod}_2'(\tau)=\tau-2k$ where $k\in\Z$ is determined such that
$-1<\re(\tau-2k)\le1$. 

\begin{defn}
The Theta group $\Gamma_\Theta$ consists of all M\"obius automorphisms
$\gamma:\Hyp\to\Hyp$ generated by the elements $\Strans(\tau)=-1/\tau$
and $\Ttrans^2(\tau)=\tau+2$.
\end{defn}

This means that a group element $\gamma\in\Gamma_\Theta$ may be written in
the form
\begin{equation}
\gamma=\Ttrans^{2k_N}\circ\Strans^{\sigma_{N}}\circ\Ttrans^{2k_{N-1}}\circ\cdots
\circ \Strans^{\sigma_2}\circ\Ttrans^{2k_1},  
\label{eq:gamma-element}
\end{equation}
where the powers have $k_j,\sigma_j\in\Z$. Since $\Strans$ involutive, that is,
$\Strans^2(\tau)=\tau$ holds, we may assume $\sigma_j\in\{0,1\}$. Moreover,
instances of $\sigma_0=0$ may be skipped, and the same if $k_j=0$, unless
they are first or last (when they are kept for formal reasons). So we assume
$\sigma_j=1$ for all $j$, and $k_j\ne0$ for $2\le j\le N-1$. 
The Theta group may also be identified as consisting of the M\"obius
automorphisms of the form
\[
\gamma(\tau)=\frac{a\tau+b}{c\tau+d},\qquad a,b,c,d\in\Z\,\,\,\,\text{with}
\,\,\,\,ad-bc=1,  
\]  
where in addition
\[
\begin{pmatrix} a & b\\ c & d\\ \end{pmatrix}
\equiv\begin{pmatrix} 1&0\\0&1\\ \end{pmatrix}\,\,\,\,\text{or}\,\,\,\,
\begin{pmatrix} 0&1\\1&0\\ \end{pmatrix}\,\,\,\,(\text{mod}\,\,\,\,2).  
\]

\begin{defn}
The sets $\gamma(\calD_\Theta)$, where $\gamma\in\Gamma_\Theta$, are called
\emph{tiles}. Their union
\[
\Gamma_\Theta(\calD_\Theta):=\cup\{\gamma(\calD_\Theta):\,\,
\gamma\in\Gamma_\Theta\}
\]
is called the \emph{union of tiles}. 
\end{defn}
  
\begin{prop}
The mesh points form the set
\[
\Gamma_\Theta(\partial\calD_\Theta)=
\cup\{\gamma(\partial\calD_\Theta):\,\,\gamma\in \Gamma_\Theta\},
\]
which is a $\sigma$-finite set with respect to linear measure.  
It is disjoint from the union of tiles $\Gamma_\Theta(\calD_\Theta)$, 
and together these two sets cover the hyperbolic plane: $\Hyp=
\Gamma_\Theta(\calD_\Theta)\cup\Gamma_\Theta(\partial\calD_\Theta)$.
\label{prop:setka-tiles}
\end{prop}

\begin{prop}
Suppose $\gamma\in\Gamma_\Theta$ is of the form
\begin{equation}
\gamma=\Ttrans^{2k_N}\circ\Strans\circ\Ttrans^{2k_{N-1}}\circ\cdots
\circ \Strans\circ\Ttrans^{2k_1},  
\label{eq:gamma-element2}
\end{equation}
where $k_j\in\Z_{\ne0}$ for all $j$ with $2\le j\le N-1$. If
$\tau\in\gamma(\calD_\Theta)$,  we
find that $\mathbbm{g}_2^l(\mathrm{mod}_2(\tau))\in\Hyp\cap\D$ for
$0\le l\le N-2$
while
$\mathbbm{g}_2^{N-1}(\mathrm{mod}_2(\tau))=\gamma^{-1}(\tau)\in\calD_\Theta$.
In particular, the level of the point is given by 
$\mathbbm{n}(\tau)=\mathbbm{n}_{\mathrm{alg}}(\tau)=
\mathbbm{n}^\ast_{\mathrm{alg}}(\tau)=N-1$. 
\label{prop:identify_level}
\end{prop}  

\begin{proof}
We agree to write $\tau_\gamma:=\gamma^{-1}(\tau)\in\calD_\Theta$. Then for
$k_1\in\Z$,  no matter whether $k_1=0$ or not,
$\Strans\circ\Ttrans^{2k_1}(\tau_\gamma)\in \Hyp\cap\D$ holds. Next, since
$k_2\in\Z_{\ne0}$, we find that
$\Strans\circ\Ttrans^{2k_2}\circ\Strans\circ\Ttrans^{2k_1}(\tau_\gamma)
\in\Hyp\cap\D$.
As we keep going, we find that
\[
\Strans\circ\Ttrans^{2k_{N-1}}\circ\cdots
\circ \Strans\circ\Ttrans^{2k_1}(\tau_\gamma)\in\Hyp\cap\D.
\]
Then
\begin{multline}
\mathrm{mod}_2(\tau)=\mathrm{mod}_2(\gamma(\tau_\gamma))=
\mathrm{mod}_2(\gamma(\tau_\gamma))
=\mathrm{mod}_2(\Ttrans^{2k_N}\circ\Strans\circ\Ttrans^{2k_{N-1}}\circ\cdots
\circ \Strans\circ\Ttrans^{2k_1}(\tau_\gamma))
\\
=\Strans\circ\Ttrans^{2k_{N-1}}\circ\cdots
\circ \Strans\circ\Ttrans^{2k_1}(\tau_\gamma)\in\Hyp\cap\D,
\end{multline}
and an application of the lifted Gauss-type map gives that
\begin{multline}
\mathbbm{g}_2\circ\mathrm{mod}_2(\tau)=
\mathbbm{g}_2\circ\Strans\circ\Ttrans^{2k_{N-1}}\circ\cdots
\circ \Strans\circ\Ttrans^{2k_1}(\tau_\gamma)
=\Strans\circ\Ttrans^{2k_{N-2}}\circ\cdots
\circ \Strans\circ\Ttrans^{2k_1}(\tau_\gamma)\in\Hyp\cap\D.
\end{multline}
By repeated application of the lifted Gauss-map, we find that
\begin{equation}
\mathbbm{g}_2^{l}\circ\mathrm{mod}_2(\tau)=
\Strans\circ\Ttrans^{2k_{N-l-1}}\circ\cdots\circ \Strans\circ\Ttrans^{2k_1}
(\tau_\gamma)\in\Hyp\cap\D,\qquad l=1,\ldots, N-2.
\end{equation}
It now follows that
\begin{equation}
\mathbbm{g}_2^{N-1}\circ\mathrm{mod}_2(\tau)=
\mathbbm{g}_2\circ \Strans\circ\Ttrans^{2k_1}
(\tau_\gamma)=
\mathrm{mod}_2\circ\Ttrans^{2k_1}(\tau_\gamma)=\tau_\gamma\in\calD_\Theta,
\end{equation}
as claimed. This means that the fly-catcher algorithm of Definition
\ref{def:flycatcher} applied to $\mathbbm{g}_2$ in place of $\mathbbm{g}_2^\ast$
ends with a point $\tau_{N-1}=\mathbbm{g}_2^{N-1}(\tau)\notin\partial\calD_\Theta$,
in which case $\tau\in\gamma(\calD_\Theta)$ cannot be a mesh point.
Moreover, the claim that $\mathbbm{n}_{\mathrm{alg}}(\tau)=N-1$ now follows from
reading the algorithm. Moreover, by \eqref{eq:n=n*} and the observation that
$\tau$ is a non-mesh point, Proposition \ref{prop:n=n*} gives that
$\mathbbm{n}(\tau)=\mathbbm{n}^\ast_{\mathrm{alg}}(\tau)=
\mathbbm{n}_{\mathrm{alg}}(\tau)=N-1$, as claimed.
\end{proof}  

\begin{proof}[Proof of Proposition \ref{prop:setka-tiles}]
It is well-known that $\calD_\Theta$ is the fundamental domain for $\Hyp$ modulo
the Theta group $\Gamma_\Theta$. This entails that the union of the tiles 
fills the hyperbolic plane $\Hyp$, except for the boundaries of the tiles. The
union of the boundaries of the tiles is given by the set of mesh points, which
we may see from applying the argument of the proof of Proposition
\ref{prop:identify_level} to the case of $\tau\in\gamma(\partial\calD_\Theta)$. 
\end{proof}

\subsection{The harmonic extension algorithm in terms
of the Theta group and the extended Gauss-type map}
We follow the solution scheme in Subsection \ref{subsec:Solutionscheme-gen},
where a $2$-periodic harmonic function $h=h_\infty:\Hyp\to\C$ is found with
$\shin\,h=f$ for a given holomorphic function $f:\Hyp\to\C$. We note that the
only freedom to choose the harmonic function $h$ emerges in Step I of the
algorithm in Subsection \ref{subsec:Solutionscheme-gen}, where a solution $h_0$ 
is found to a given Dirichlet problem on the fundamental domain $\calD_\Theta$.
The harmonic extension algorithm outlined in Subsection
\ref{subsec:Solutionscheme-gen} then provides the harmonic
extension $h_\infty:\Hyp\to\C$ of the given initial solution $h_0$, which is
then uniquely determined by the choice of $h_0$. Our aim here is to represent
the extension $h=h_\infty$ in a fashion which allows for convenient estimation. 
We begin with the two basic properties of $h$, that $h=f$ on $\Te_+$, and
that by Steps II, III, and IV in the harmonic extension algorithm, we have 
that
\[
h\circ\Ttrans^2=h,\quad g\circ \Strans^\ast=-g,
\]
where $g:=h-f$ is harmonic in $\Hyp$. Inserting $h-f$ in place of $g$ the
second equality, we find that
\begin{equation}
h\circ\Ttrans^2=h,\quad h=-h\circ \Strans^\ast+f+f\circ\Strans^\ast.
\label{eq:basicequations1.01}
\end{equation}
It will be convenient to introduce the notation
\begin{equation}
f_{\mathrm{sym}}:=f+f\circ\Strans^\ast,
\end{equation}
which is a harmonic function in $\Hyp$ with the symmetry property
$f_{\mathrm{sym}}\circ\Strans^\ast=f_{\mathrm{sym}}$. In terms of $f_{\mathrm{sym}}$,
the equations \eqref{eq:basicequations1.01} take the simplified form
\begin{equation}
h\circ\Ttrans^2=h,\quad h=-h\circ \Strans^\ast+f_{\mathrm{sym}}.
\label{eq:basicequations1.02}
\end{equation}
These two relations may be combined to form, in a first step,
\begin{equation}
h=-h\circ\Ttrans^{2k}\circ\Strans^\ast+f_{\mathrm{sym}},\qquad k\in\Z,
\end{equation}
and, by iteration,
\begin{equation}
h=(-1)^N h\circ\Ttrans^{2k_N}\circ\Strans^\ast\circ\cdots\circ\Ttrans^{2k_1}
\circ\Strans^\ast+f_{\mathrm{sym}}
+\sum_{j=1}^{N-1}(-1)^jf_{\mathrm{sym}}\circ \Ttrans^{2k_j}\circ\Strans^\ast\circ\cdots
\circ\Ttrans^{2k_1}\circ\Strans^\ast,
\label{eq:basicequations1.03}
\end{equation}
where $k_1,\ldots,k_N\in\Z$ and $N=1,2,3,\ldots$. We apply
\eqref{eq:basicequations1.03} at a point $\tau\in\calV_+\cup\calL_{-1}^+$,
in order to estimate the value $h(\tau)$ using the right-hand side.
Here, we recall that $\tau\in\calV_+\cup\calL_{-1}^+$ means that
$\tau\in\Hyp$ and that $-1\le\re\,\tau<1$. Note that by the $2$-periodicity of
$h$, it is enough to consider such $\tau$. The result is
\begin{multline}
h(\tau)=(-1)^N h\circ\Ttrans^{2k_N}\circ\Strans^\ast\circ\cdots\circ\Ttrans^{2k_1}
\circ\Strans^\ast(\tau)+f_{\mathrm{sym}}(\tau)
\\
+\sum_{j=1}^{N-1}(-1)^jf_{\mathrm{sym}}\circ \Ttrans^{2k_j}\circ\Strans^\ast\circ\cdots
\circ\Ttrans^{2k_1}\circ\Strans^\ast(\tau).
\label{eq:basicequations1.04}
\end{multline}
We let the integer $k_j=k_j(\tau)\in\Z$ depend on $\tau$ in such a fashion that
\[
\mathbbm{g}^\ast_2(\tau)=\mathrm{mod}_2\circ\Strans^\ast(\tau)=\Ttrans^{2k_1}\circ
\Strans^\ast(\tau),
\]
and, more generally, 
\[
\mathbbm{g}^\ast_2(\tau_{j-1})=\mathrm{mod}_2\circ\Strans^\ast(\tau_j)=
\Ttrans^{2k_j}\circ \Strans^\ast(\tau_{j-1}),\quad\text{where}\,\,\,\,\tau_{j-1}
=(\mathbbm{g}_2^{\ast})^{j-1}(\tau). 
\]
It then follows from \eqref{eq:basicequations1.04} that
\begin{equation}
h(\tau)=(-1)^N h\circ(\mathbbm{g}_2^\ast)^N(\tau)
\\
+\sum_{j=0}^{N-1}(-1)^jf_{\mathrm{sym}}\circ (\mathbbm{g}_2^\ast)^j(\tau),
\label{eq:basicequations1.05}
\end{equation}
for $\tau\in\Hyp$ with $-1\le\re\,\tau<1$. This fits in nicely with the
fly-catcher algorithm of Definition \ref{def:flycatcher} if we pick
$N=\mathbbm{n}^\ast_{\mathrm{alg}}(\tau)$:
\begin{equation}
h(\tau)=(-1)^{{\mathbbm{n}}_{\mathrm{alg}}^\ast(\tau)}
h\circ(\mathbbm{g}_2^\ast)^{\mathbbm{n}_{\mathrm{alg}}^\ast(\tau)}(\tau)
\\
+\sum_{j=0}^{\mathbbm{n}_{\mathrm{alg}}^\ast(\tau)-1}
(-1)^jf_{\mathrm{sym}}\circ (\mathbbm{g}_2^\ast)^j(\tau),
\label{eq:basicequations1.06}
\end{equation}
again for $\tau\in\Hyp$ with $-1\le\re\,\tau<1$. \emph{Indeed, for
$N=\mathbbm{n}^\ast_{\mathrm{alg}}(\tau)$ we know that
$(\mathbbm{g}_2^\ast)^N(\tau)\in\bar\calD_\Theta$ (this is the stopping condition
for the algorithm) while $(\mathbbm{g}_2^\ast)^j(\tau)\in\Hyp\cap\D$ for
$0\le j<N$.
In addition, the imaginary parts of the iterates
$j\mapsto\im\,(\mathbbm{g}_2^\ast)^j(\tau)$ form a strictly
increasing sequence over the integer interval $0\le j\le N$, by Remark
\ref{rem:im_part_growing}.} In terms of the family of norms
\[
\|F\|_{(\epsilon)}:=\sup\big\{|F(\tau)|:\,\,\tau\in\Hyp\cap\D\,\,\,
\text{with}\,\,\,\im\,\tau\ge\epsilon\big\},\qquad 0<\epsilon<1,  
\]
we obtain the following key lemma.

\begin{lem}
If $\tau\in\Hyp$ with $-1\le\re\,\tau<1$, we choose
$N=\mathbbm{n}^\ast_{\mathrm{alg}}(\tau)$, which if $N\ge1$ gives the estimate
\[
|h(\tau)|\le |h\circ(\mathbbm{g}_2^\ast)^N(\tau)|+N\,
\|f_{\mathrm{sym}}\|_{(\im\,\tau)}.
\]  
\label{prop:pointwiseest_iter}
\end{lem}

\begin{rem}
In view of Corollary \ref{cor:n_estimate} and Proposition \ref{prop:n=n*},
we have that
\[
0\le \mathbbm{n}^\ast_{\mathrm{alg}}(\tau)\le\mathbbm{n}(\tau)
\le\frac12+\frac{1}{2\,\im\,\tau},\qquad \tau\in\Hyp.
\]
\label{rem:n*estimate}
\end{rem}

\subsection{Hyperbolic Fourier series for distributions on
  the extended real line }
\label{subsec:hfs_distr1}
We implement the estimate of Lemma \ref{prop:pointwiseest_iter},
while assuming that $f:\Hyp\to\C$ is holomorphic with
\begin{equation}
|f(\tau)|=\Ordo\bigg(\frac{1+|\tau|^2}{\im\tau}\bigg)^k
\label{eq:polgrowth0*}
\end{equation}
uniformly in $\tau\in\Hyp$, for some $k$ with $0<k<+\infty$.
It is immediate from \eqref{eq:polgrowth0*} that the estimate carries over to
$f_{\mathrm{sym}}$: 
\begin{equation}
|f_{\mathrm{sym}}(\tau)|=|f(\tau)+f(1/\bar\tau)|=
\Ordo\bigg(\frac{1+|\tau|^2}{\im\tau}\bigg)^k,
\label{eq:polgrowth0*.01}
\end{equation}
again uniformly in $\tau\in\Hyp$. 
As a next step, we find a solution $h_0$ to the Dirichlet problem 
outlined in Step I of the harmonic extension algorithm in
Subsection \ref{subsec:Solutionscheme-gen}. 
As outlined in Proposition \ref{prop:lambda-classical}, the modular lambda 
function maps the fundamental domain $\calD_\Theta$ conformally onto the
half-plane $\C_{\re<\frac12}$ with a slit removed along the negative real axis. 
Let us write $\lambda^{-1}$ for the inverse conformal mapping from the slit
half-plane onto $\calD_\Theta$, which then maps the boundary line
$\calL_{\frac12}=\{\zeta\in\C:\,\re\,\zeta=\tfrac12\}$ onto the semicircle
$\Te_+$.
Let $H_0$ be a harmonic function in the half-plane $\C_{\re<\frac12}$ which
equals $F:=f\circ\lambda^{-1}$ along $\calL_{\frac12}$. We choose the Poisson
extension given by \eqref{eq:Poisson0}, with the obvious modifications required
by rotating and shifting the half-plane:
\begin{equation}
H_0(\zeta):=\frac{1}{\pi}
\int_{\R}\frac{\frac12-\re\,\zeta}{|\zeta-\frac12-\imag t|^2}
\,F(\tfrac12+\imag t)\,\diff t,\qquad \zeta\in\C_{\re<\frac12}.
\end{equation}
The growth condition \eqref{eq:polgrowth0*} entails that the estimate
\begin{equation}
|f(\tau)|=\Ordo({\im\tau})^{-k},\qquad \tau\in\Te_+,
\label{eq:polgrowth0*1}
\end{equation}
holds uniformly, and in view of \eqref{eq:inverse_lambda1} and
\eqref{eq:inverse_lambda2}, we know that
\begin{equation}
\frac{1}{\im\,\lambda^{-1}(\tfrac12+\imag t)}=\pi^{-1}{\log(16|t|)}+\Ordo(1)
\label{eq:lambda_critical_line}  
\end{equation}
uniformly for real $t$ with $|t|\gg1$, it follows from this that
$F=f\circ\lambda^{-1}$
has
\begin{equation}
F(\tfrac12+\imag t)=\Ordo(\log^k(2+|t|))
\end{equation}
uniformly in $t\in\R$. 
We can now appeal to Proposition \ref{prop:loggrowth}, 
which tells us that the solution $H_0$ to the Dirichlet problem has the 
growth bound
\begin{equation}
H_0(\zeta)=\Ordo(\log^k(2+|\zeta|)),
\end{equation}
uniformly in $\C_{\re<\frac12}$. The induced function $h_0:=H_0\circ\lambda$ is
$2$-periodic and harmonic in $\Omega_0$, with boundary values $h_0=f$ along
$\Te_+$. Near the cusp at $-1$, we can use the approximation
\eqref{eq:lambdaproperty3} and the observation that $\im\,\tau\sim|\tau+1|$ for
$\tau\in\Omega_0$ near $-1$, together with $2$-periodicity to obtain that
\begin{equation}
h_0(\tau)=H_0\circ\lambda(\tau)=\Ordo(\log^k(2+|\lambda(\tau)|))
=\Ordo\big(1+(\im\tau)^{-k}\big),\qquad \tau\in\bar\Omega_0,  
\label{eq:est_h0_Omega0}
\end{equation}
with a uniform implicit constant. Here, $\bar\Omega_0$ stands for the relative
closure of the domain $\Omega_0$ in $\Hyp$. This means that we have an effective
estimate of the harmonic function $h=h_\infty:\Hyp\to\C$ in the region
$\bar\Omega_0$. 
We want to extend this estimate to all of $\Hyp$.
To this end, Lemma \ref{prop:pointwiseest_iter} comes in handy. 
In that context, we use periodicity to reduce to $-1\le\re\,\tau<1$,
and if addition, $\tau\notin\bar\Omega_0$, we have that
for $N=\mathbbm{n}^\ast_{\mathrm{alg}}(\tau)$, the iterate
$(\mathbbm{g}_2^\ast)^N(\tau)$ is in $\bar\calD_\Theta\subset\bar\Omega_0$,
so that we may apply the estimate \eqref{eq:est_h0_Omega0} at that point.
Moreover, by the monotonicity of the imaginary parts of the iterates, we have
\begin{equation}
\im\,(\mathbbm{g}_2^\ast)^N(\tau)>\im\,(\mathbbm{g}_2^\ast)^{N-1}(\tau)>\cdots
>\im\,\tau,  
\end{equation}
so that, consequently, the estimate \eqref{eq:est_h0_Omega0} entails that
uniformly
\begin{equation}
h_0\circ(\mathbbm{g}_2^\ast)^N(\tau)
=\Ordo\big(1+(\im\,(\mathbbm{g}_2^\ast)^N(\tau))^{-k}\big)
=\Ordo\big(1+(\im\,\tau)^{-k}\big),\qquad \tau\in\calV\setminus\bar\Omega_0.  
\label{eq:est_h0_Omega1}
\end{equation}
We arrive at the following proposition.

\begin{prop}
Suppose $f:\Hyp\to\C$ is holomorphic, subject to the growth bound
\eqref{eq:polgrowth0*} for some $k$ with $0<k<+\infty$. Then there
exists a $2$-periodic harmonic function
$h:\Hyp\to\C$ such that $h=f$ along the semicircle $\Te_+$, with
\begin{equation}
|h(\tau)|=\Ordo\big(1+(1+\mathbbm{n}^\ast_{\mathrm{alg}}(\tau))(\im\,\tau)^{-k}
\big)=\Ordo\big(1+(\im\,\tau)^{-k-1}\big),
\end{equation}
with implied constants that are uniform in $\tau\in\Hyp$. 
\label{prop:8.5.1}
\end{prop}

\begin{proof}
By the harmonic extension algorithm in Subsection
\ref{subsec:Solutionscheme-gen}, the function $h_0$ extends to harmonic
function $h=h_\infty:\Hyp\to\C$ that is $2$-periodic and solves $h=f$ along
$\Te_+$ (see Corollary \ref{cor:Hyperext1}). To obtain the claimed estimate,
we observe first that it holds on $\bar\Omega_0$, by \eqref{eq:est_h0_Omega0}. 
In the remainder $\Hyp\setminus\bar\Omega_0$, we use the $2$-periodicity of
$h$ to restrict to the half-strip with $-1\le\re\,\tau<1$ and then apply
Lemma \ref{prop:pointwiseest_iter} to get the asserted estimate from
\eqref{eq:polgrowth0*.01} combined with Remark \ref{rem:n*estimate}, as well  
as the estimates \eqref{eq:est_h0_Omega0} and \eqref{eq:est_h0_Omega1}. 
\end{proof}

We obtain a more precise version of Theorem \ref{thm:hfs2}.

\begin{cor}
Suppose $f:\Hyp\to\C$ is harmonic, subject to the growth bound
\eqref{eq:polgrowth0*} for some $k$ with $0<k<+\infty$. Then there exist
complex coefficients $a_0$ and $a_n,b_n$ for $n\in\Z_{\ne0}$, with growth
\[
a_n,b_n=\Ordo\big((|n|+1)^{k}\log^2(2+|n|)\big)
\quad\text{as}\,\,\,\,|n|\to+\infty,
\]
such that $f$ has the convergent hyperbolic Fourier series expansion 
\[
f(\tau)=a_0+\sum_{n\in\Z_{>0}}\big(a_n\e^{\imag\pi n\tau}+b_n\e^{-\imag\pi n/\tau}\big)
+\sum_{n\in\Z_{<0}}\big(a_n\e^{\imag\pi n\bar\tau}+b_n\e^{-\imag\pi n/\bar\tau}\big),
\qquad\tau\in\Hyp.
\]
As such, the coefficients are uniquely determined by $f$.
\label{cor:8.5.2}
\end{cor}
  
\begin{proof}
The uniqueness part follows from Theorem \ref{thm:uniq1}. 
We proceed to show that the coefficient sequence exists.
We apply Proposition \ref{prop:beta3} with $\alpha=k$, which gives the
unique splitting of the harmonic function $f$ in three parts:
\[
f(\tau)=f_+(\tau)+f_-(\tau)+f(\imag),
\]
where $f_+$ is holomorphic with $f_+(\imag)=0$, $f_-$ is conjugate-holomorphic
with $f_-(\imag)=0$, and $f(\imag)$ is constant. The functions $f_+$ and $f_-$
enjoy the same growth bound as $f$ itself. We apply Proposition
\ref{prop:8.5.1} to the holomorphic function $f_+$, which gives us a
$2$-periodic harmonic function $h_+:\Hyp\to\C$ with $h_+=f_+$ on $\Te_+$
subject to the growth bound
\begin{equation}
|h_+(\tau)|=\Ordo\big(1+(1+\mathbbm{n}^\ast_{\mathrm{alg}}(\tau))
(\im\,\tau)^{-k}\big)=\Ordo\big(1+(\im\,\tau)^{-k-1}\big),
\label{eq:h+bound1}
\end{equation}
uniformly in $\Hyp$. Then, by Proposition \ref{prop:function->coeff2.0},
$h_+$ has a convergent harmonic Fourier series expansion
\[
h_+(\tau)=\sum_{n\in\Z}c_n\,e_n(\tau),\qquad \tau\in\Hyp,
\]
and the Schwarz transform of $h_+$ gives back $f_+$: $\shin\,h_+=f_+$.
Moreover, Proposition \ref{prop:function->coeff0} together with the growth
control on $h_+$ ensures that the Fourier coefficients have the growth bound
$c_n=\Ordo(|n|^{k+1})$ as $|n|\to+\infty$. We can, however, do slightly better
if we use the intermediate bound in \eqref{eq:h+bound1}. Indeed, by
\eqref{eq:get_coeff_cn}, we may represent the coefficients in the form
\begin{equation}
c_n=\frac{\e^{\pi|n|s}}{2}\int_{-1}^{1}\e^{-\imag\pi nt}h_+(t+\imag s)\,\diff t,
\label{eq:get_coeff_cn01}
\end{equation}
provided that $s>0$. By the triangle inequality for integrals, this leads to
\begin{equation}
|c_n|\le\frac{\e^{\pi|n|s}}{2}\int_{-1}^{1}|h_+(t+\imag s)|\,\diff t,
\label{eq:get_coeff_cn02}
\end{equation}
and, as we implement the first estimate in \eqref{eq:h+bound1}, we find that
if $C>0$ is the implied constant, 
\begin{multline}
|c_n|\le C\,\e^{\pi|n|s}\Big(1+s^{-k}\int_{-1}^{1}
\big(1+\mathbbm{n}^\ast_{\mathrm{alg}}(t+\imag s)\big)\,\frac{\diff t}{2}\Big)
\le C\,\e^{\pi|n|s}\Big(1+s^{-k}\frac{1}{\pi^2}\log^2\frac{1}{s}+
\Ordo\Big(s^{-k}\log\frac1s\Big)\Big),
\label{eq:get_coeff_cn02.1}
\end{multline}
if $s>0$ is small. Here, we used the comparison of height functions in
Proposition \ref{prop:n=n*} as well as the asymptotic average height
calculation of Proposition \ref{prop:n_estimate_ave}.
For large $|n|$ we plug in $s=1/|n|$, which gives
\begin{equation}
|c_n|\le C\,\e^{\pi}\big(\pi^{-2}|n|^k\log^2|n|+\Ordo(|n|^k\log|n|)\big)
=\Ordo(|n|^k\log^2|n|).
\label{eq:get_coeff_cn02.2}
\end{equation}

From Proposition
\ref{prop:shinexpansion1}, we know that $\shin\,h_+=f_+$ is represented by a
holomorphic hyperbolic Fourier series. In an analogous fashion, we 
deal with the conjugate-holomorphic part. For instance, we can proceed as
follows. We take the complex
conjugate $\bar f_-$ of $f_-$ to get a holomorphic function, and find a
$2$-periodic harmonic function $h_-$ with $h_-=\bar f_-$ on $\Te_+$ with the
same type of growth bound as for $h_+$. Hence the function $h_-$ has a
harmonic Fourier series expansion, with the same kind of coefficient bound,
and if we take its Schwarz transform we 
recover $\bar f_-$: $\shin\,h_-=\bar f_-$. 
This supplies a conjugate-holomorphic hyperbolic Fourier series expansion
for $f_-$, and together, the holomorphic and conjugate-holomorphic hyperbolic
Fourier series give a harmonic hyperbolic Fourier series expansion of $f$.
\end{proof}  

\begin{rem}
As we discussed back in Subsections \ref{subsec:harmext-gen} and
\ref{subsec:harmext-hfs}, Theorems \ref{thm:hfs1} and \ref{thm:hfs2} are
equivalent. For this reason, we will not supply a separate proof of Theorem
\ref{thm:hfs1}. 
\label{rem:equivalent_hfs}
\end{rem}

\subsection{Hyperbolic Fourier series for bounded harmonic
functions}
\label{subsec:hfs_for_bounded}
The instance of $k=0$ in Proposition \ref{cor:8.5.2} is of particular
interest. We cannot plug in $k=0$ directly in Corollary \ref{cor:8.5.2},
but we could of course settle for any $k>0$ instead.
We prefer to derive a more precise result.
Let us begin with a holomorphic function $f:\Hyp\to\C$ subject to the
logarithmic growth bound
\begin{equation}
|f(\tau)|=\Ordo\bigg(\log\frac{1+|\tau|^2}{\im\,\tau}\bigg),
\label{eq:log_growth1.01}
\end{equation}
uniformly in $\tau\in\Hyp$.
It is immediate that \eqref{eq:log_growth1.01} carries over to
$f_{\mathrm{sym}}$ as well:
\begin{equation}
|f_{\mathrm{sym}}(\tau)|=|f(\tau)+f(1/\bar\tau)|=
\Ordo\bigg(\log\frac{1+|\tau|^2}{\im\,\tau}\bigg),
\label{eq:log_growth1.02}
\end{equation}
again uniformly in $\tau\in\Hyp$. Next, we find the solution $h_0$ to the
Dirichlet problem outlined in Step I of the harmonic extension algorithm
in Subsection \ref{subsec:Solutionscheme-gen}. 
As in Subsection \ref{subsec:hfs_distr1}, we let $H_0$ be a harmonic function
in the half-plane $\C_{\re<\frac12}$ which equals $F:=f\circ\lambda^{-1}$ along
the boundary line $\calL_{\frac12}$.
We choose the Poisson extension given by \eqref{eq:Poisson0}, with the
obvious modifications required by rotating and shifting the half-plane:
\begin{equation}
H_0(\zeta):=\frac{1}{\pi}
\int_{\R}\frac{\frac12-\re\,\zeta}{|\zeta-\frac12-\imag t|^2}
\,F(\tfrac12+\imag t)\,\diff t,\qquad \zeta\in\C_{\re<\frac12}.
\end{equation}
The growth condition \eqref{eq:log_growth1.01} entails that the estimate
\begin{equation}
|f(\tau)|=\Ordo\Big(\log\frac{2}{\im\,\tau}\Big),\qquad \tau\in\Te_+,
\label{eq:polgrowth0*1.1}
\end{equation}
holds uniformly, and in view of \eqref{eq:lambda_critical_line},
it follows from this that $F=f\circ\lambda^{-1}$ has
\begin{equation}
F(\tfrac12+\imag t)=\Ordo(\log\log(10+|t|))
\end{equation}
uniformly in $t\in\R$. 
We can now appeal to Proposition \ref{prop:logloggrowth}, 
which tells us that the solution $H_0$ to the Dirichlet problem has the 
growth bound
\begin{equation}
H_0(\zeta)=\Ordo(\log\log(10+|\zeta|)),
\end{equation}
uniformly in $\C_{\re<\frac12}$. The induced function $h_0:=H_0\circ\lambda$ is
$2$-periodic and harmonic in $\Omega_0$, with boundary values $h_0=f$ along
$\Te_+$. Near the cusp at $-1$, we can use the approximation
\eqref{eq:lambdaproperty3} and the observation that $\im\,\tau\sim|\tau+1|$ for
$\tau\in\Omega_0$ near $-1$, together with $2$-periodicity to obtain that
\begin{equation}
h_0(\tau)=H_0\circ\lambda(\tau)=\Ordo(\log\log(10+|\lambda(\tau)|))
=\Ordo\Big(\log\Big(2+\frac{1}{\im\,\tau}\Big)\Big),\qquad
\tau\in\bar\Omega_0,  
\label{eq:est_h0_Omega0.1}
\end{equation}
with a uniform implicit constant. This means that we have an effective
estimate of the harmonic function $h=h_\infty:\Hyp\to\C$ in the region
$\bar\Omega_0$. 
We want to extend this estimate to all of $\Hyp$.
To this end, Lemma \ref{prop:pointwiseest_iter} comes in handy. 
In that context, we use periodicity to reduce to $-1\le\re\,\tau<1$,
and note that if $\tau\notin\bar\Omega_0$,
we find that for $N=\mathbbm{n}^\ast_{\mathrm{alg}}(\tau)$, the iterate
$(\mathbbm{g}_2^\ast)^N(\tau)$ is in $\bar\calD_\Theta\subset\bar\Omega_0$,
so that we may apply the estimate \eqref{eq:est_h0_Omega0.1} at that point.
Moreover, by the monotonicity of the imaginary parts of the iterates, we have
\begin{equation}
\im\,(\mathbbm{g}_2^\ast)^N(\tau)>\im\,(\mathbbm{g}_2^\ast)^N-1(\tau)>\cdots
>\im\,\tau,  
\end{equation}
so that, consequently, the estimate \eqref{eq:est_h0_Omega0.1} entails that
uniformly
\begin{multline}
h_0\circ(\mathbbm{g}_2^\ast)^N(\tau)
=\Ordo\Big(\log\Big(2+\frac{1}{\im\,(\mathbbm{g}_2^\ast)^N(\tau)}\Big)\Big)
=\Ordo\Big(\log\Big(2+\frac{1}{\im\,\tau}\Big)\Big),
\qquad \tau\in\calV_+\setminus\bar\Omega_0.  
\label{eq:est_h0_Omega1.1}
\end{multline}
We arrive at the following proposition.

\begin{prop}
Suppose $f:\Hyp\to\C$ is holomorphic, subject to the logarithmic growth bound
\eqref{eq:log_growth1.01}. Then there
exists a $2$-periodic harmonic function
$h:\Hyp\to\C$ such that $h=f$ along the semicircle $\Te_+$, with
\begin{equation}
|h(\tau)|=\Ordo\Big((1+\mathbbm{n}^\ast_{\mathrm{alg}}(\tau))
\log\Big(2+\frac{1}{\im\,\tau}\Big)\Big),
\end{equation}
with an implied constant that is uniform in $\tau\in\Hyp$. 
\label{prop:8.6.1}
\end{prop}

\begin{proof}
By the harmonic extension algorithm in Subsection
\ref{subsec:Solutionscheme-gen}, the function $h_0$ extends to harmonic
function $h=h_\infty:\Hyp\to\C$ that is $2$-periodic and solves $h=f$ along
$\Te_+$ (see Corollary \ref{cor:Hyperext1}). To obtain the claimed estimate,
we observe first that it holds on $\bar\Omega_0$, by
\eqref{eq:est_h0_Omega0.1}. 
In the remainder $\Hyp\setminus\bar\Omega_0$, we use the $2$-periodicity of
$h$ to restrict to the half-strip with $-1\le\re\,\tau<1$ and then apply
Lemma \ref{prop:pointwiseest_iter} to get the asserted estimate from
\eqref{eq:log_growth1.01} as well as the estimates \eqref{eq:est_h0_Omega0.1}
and \eqref{eq:est_h0_Omega1.1}. 
\end{proof}

We obtain a slow growth bound on the coefficients in the hyperbolic Fourier
series expansion for a given bounded function. 

\begin{cor}
Suppose $f:\Hyp\to\C$ is harmonic and bounded. Then there exist
complex coefficients $a_0$ and $a_n,b_n$ for $n\in\Z_{\ne0}$, with growth
\[
a_n,b_n=\Ordo\big(\log^3(2+|n|)\big)
\quad\text{as}\,\,\,\,|n|\to+\infty,
\]
such that $f$ has the convergent hyperbolic Fourier series expansion 
\[
f(\tau)=a_0+\sum_{n\in\Z_{>0}}\big(a_n\e^{\imag\pi n\tau}+b_n\e^{-\imag\pi n/\tau}\big)
+\sum_{n\in\Z_{<0}}\big(a_n\e^{\imag\pi n\bar\tau}+b_n\e^{-\imag\pi n/\bar\tau}\big),
\qquad\tau\in\Hyp.
\]
As such, the coefficients are uniquely determined by $f$.
\label{cor:8.6.2}
\end{cor}
  
 \begin{rem}
 \label{rem:8.6.3}
 The estimate is uniform in $f$, provided that the $L^\infty$-norm of $f$
 is bounded by a given constant. 
 \end{rem} 
  
\begin{proof}[Proof of Corollary \ref{cor:8.6.2}]
The uniqueness part follows from Theorem \ref{thm:uniq1}. 
We proceed to show that the coefficient sequence exists.
We apply Proposition \ref{prop:beta2}, which gives the
unique splitting of the harmonic function $f$ in three parts:
\[
f(\tau)=f_+(\tau)+f_-(\tau)+f(\imag),
\]
where $f_+$ is holomorphic with $f_+(\imag)=0$, $f_-$ is conjugate-holomorphic
with $f_-(\imag)=0$, and $f(\imag)$ is constant. The functions $f_+$ and $f_-$
then enjoy the logarithmic growth bound \eqref{eq:log_growth1.01}.
We apply Proposition
\ref{prop:8.6.1} to the holomorphic function $f_+$, which gives us a
$2$-periodic harmonic function $h_+:\Hyp\to\C$ with $h_+=f_+$ on $\Te_+$
and the growth bound
\begin{equation}
|h_+(\tau)|=\Ordo\Big((1+\mathbbm{n}^\ast_{\mathrm{alg}}(\tau))\log\Big(
2+\frac{1}{\im\,\tau}\Big)\Big),
\label{eq:h+bound1.1}
\end{equation}
uniformly in $\Hyp$. Then, by Proposition \ref{prop:function->coeff2.0},
$h_+$ has a convergent harmonic Fourier series expansion
\[
h_+(\tau)=\sum_{n\in\Z}c_n\,e_n(\tau),\qquad \tau\in\Hyp,
\]
and the Schwarz transform of $h_+$ gives back $f_+$: $\shin\,h_+=f_+$.
By \eqref{eq:get_coeff_cn}, we may represent the coefficients in the form
\begin{equation}
c_n=\frac{\e^{\pi|n|s}}{2}\int_{-1}^{1}\e^{-\imag\pi nt}h_+(t+\imag s)\,\diff t,
\label{eq:get_coeff_cn01.1}
\end{equation}
provided that $s>0$. By the triangle inequality for integrals, this leads to
\begin{equation}
|c_n|\le\e^{\pi|n|s}\int_{-1}^{1}|h_+(t+\imag s)|\,\frac{\diff t}{2},
\label{eq:get_coeff_cn02.1*}
\end{equation}
and, as we implement the estimate in \eqref{eq:h+bound1.1}, we find that
if $C>0$ is the implied constant, 
\begin{multline}
|c_n|\le C\,\e^{\pi|n|s}\Big(\int_{-1}^{1}
\big(1+\mathbbm{n}^\ast_{\mathrm{alg}}(t+\imag s)\big)\,\frac{\diff t}{2}
\,\,\log\Big(2+\frac{1}{s}\Big)\Big)
\\
\le C\,\e^{\pi|n|s}
\Big(\frac{1}{\pi^2}\log^2\frac{1}{s}+\Ordo\Big(\log\frac1s\Big)\Big)
\log\Big(2+\frac{1}{s}\Big),
\label{eq:get_coeff_cn02.1.1}
\end{multline}
provided that $s>0$ is small.
Here, we used as before the comparison of height functions in
Proposition \ref{prop:n=n*} as well as the asymptotic average height
calculation of Proposition \ref{prop:n_estimate_ave}.
For large $|n|$, we plug in $s=1/|n|$, which
gives
\begin{equation}
|c_n|\le C\,\e^{\pi}\big(\pi^{-2}|n|^k\log^2|n|+\Ordo(\log|n|)\big)\log(2+|n|)
=\Ordo(|\log^3|n|).
\label{eq:get_coeff_cn02.2.1}
\end{equation}

From Proposition
\ref{prop:shinexpansion1}, we know that $\shin\,h_+=f_+$ is represented by a
holomorphic hyperbolic Fourier series. In an analogous fashion, we 
deal with the conjugate-holomorphic part. For instance, we can proceed as
follows. We take the complex
conjugate $\bar f_-$ of $f_-$ to get a holomorphic function, and find a
$2$-periodic harmonic function $h_-$ with $h_-=\bar f_-$ on $\Te_+$ with the
same type of growth bound as for $h_+$. Hence the function $h_-$ has a
harmonic Fourier series expansion, with the same kind of coefficient bound,
and if we take its Schwarz transform we 
recover $\bar f_-$: $\shin\,h_-=\bar f_-$. 
This supplies a conjugate-holomorphic hyperbolic Fourier series expansion
for $f_-$, and together, the holomorphic and conjugate-holomorphic hyperbolic
Fourier series give a harmonic hyperbolic Fourier series expansion of $f$.
\end{proof}  

\subsection{Harmonic hyperbolic Fourier series representation
of exponentially growing functions}
\label{subsec:exponentially_growing}

Here, we consider a holomorphic function $f:\Hyp\to\C$ whose growth is
allowed to be substantially faster:
\begin{equation}
|f(\tau)|=\Ordo\big(\exp(\beta\pi\,\Mfun(\tau))\big)
\label{eq:expbeta_est1}
\end{equation}
with an implicit constant that is uniform in $\tau\in\Hyp$. The parameter
$\beta$ is positive: $0<\beta<+\infty$. Moreover, we recall that
the function $\Mfun(\tau)$ is given by \eqref{eq:defM}, that is,
\[
\Mfun(\tau)=\frac{\max\{1,|\tau|^2\}}{\im\,\tau}. 
\]  
It is immediate from the symmetry of $\Mfun$ expressed in \eqref{eq:Mprop}
that the symmetrized function $f_{\mathrm{sym}}(\tau)=f(\tau)+f(1/\bar\tau)$
has the same estimate \eqref{eq:expbeta_est1}:
\begin{equation}
|f_{\mathrm{sym}}(\tau)|=\Ordo\big(\exp(\beta\pi\,\Mfun(\tau))\big)
\label{eq:expbeta_est2}
\end{equation}
uniformly in $\tau\in\Hyp$. 
As we restrict the estimate \eqref{eq:expbeta_est1} to the semi-circle $\Te_+$,
we find that 
\begin{equation}
|f(\tau)|=\Ordo\big(\exp(\beta\pi/\im\,\tau)\big),\qquad\tau\in\Te_+,
\label{eq:expbeta_est2.11}
\end{equation}
holds uniformly. We now follow the same scheme as outlined in
Subsections \ref{subsec:hfs_distr1} and \ref{subsec:hfs_for_bounded}.
We let $H_0$ be a harmonic function in the half-plane $\C_{\re<\frac12}$ which
equals the function $F=f\circ\lambda^{-1}$ along the boundary line
$\calL_{\frac12}$.  In view of \eqref{eq:lambda_critical_line},
the growth condition \eqref{eq:expbeta_est2.11} entails that
\begin{equation}
|F(\tfrac12+\imag t)|=\Ordo\big((1+|t|)^{\beta}\big),\qquad t\in\R,
\end{equation}
uniformly in $t$.
We may now appeal to Proposition \ref{prop:powergrowth}, which guarantees that
if $\beta\notin\Z_{>0}$, then such a harmonic function 
$H_0$ may be found with the same order of growth:
\begin{equation}
|H_0(\zeta)|=\Ordo\big((1+|\zeta|)^{\beta}\big),\qquad \zeta\in\C_{\re<\frac12},
\label{eq:H_0_est1.01}
\end{equation}
with a uniform implicit constant. Moreover, if instead $\beta\in\Z_{>0}$,
the estimate gets slightly worse,
\begin{equation}
|H_0(\zeta)|=\Ordo\big((1+|\zeta|)^{\beta}\log(2+|\zeta|)\big),
\qquad \zeta\in\C_{\re<\frac12},
\label{eq:H_0_est1.02}
\end{equation}
again with a uniform implicit constant. We put as before $h_0:=H_0\circ\lambda$,
which defines a $2$-periodic harmonic function in $\Omega_0$, which extends
continuously up to the relative boundary $\partial\Omega_0$, and as such
has $h_0=f$ on $\Te_+$. For $\tau\in\calD_\Theta$ close to the cusp point $-1$,
we calculate that
\begin{equation}
\re\frac{\imag\pi}{\tau+1}=\frac{\pi}{\im\,\tau}+\Ordo(\im\,\tau), 
\end{equation}
so that by \eqref{eq:lambdaproperty3}, 
\begin{equation}
|\lambda(\tau)|=\frac{1}{16}\big(1+\Ordo(\im\,\tau)\big)\,\e^{\pi/\im\,\tau}. 
\label{eq:lambda-estimate_101}
\end{equation}
Moreover, by $2$-periodicity, the estimate \eqref{eq:lambda-estimate_101} 
extends to $\bar\Omega_0$, in the limit as $\im\,\tau\to0^+$. 
It now follows from the estimate \eqref{eq:H_0_est1.01} that if
$\beta\notin\Z_{>0}$, 
\begin{equation}
|h_0(\tau)|=\Ordo\big((1+|\lambda(\tau)|)^{\beta}\big)
=\Ordo\big(\exp(\beta\pi/\im\,\tau)\big),\qquad \tau\in\bar\Omega_0,
\label{eq:h_0_est1.01}
\end{equation}
whereas if $\beta\in\Z_{>0}$,
\begin{multline}
|h_0(\tau)|=\Ordo\big((1+|\lambda(\tau)|)^{\beta}\log(2+|\lambda(\tau)|)\big)
=\Ordo\Big(\Big(1+\frac{1}{\im\,\tau}\Big)
\exp(\beta\pi/\im\,\tau)\Big),\qquad \tau\in\bar\Omega_0.
\label{eq:h_0_est1.02}
\end{multline}
Both estimates \eqref{eq:h_0_est1.01} and \eqref{eq:h_0_est1.02}
hold uniformly, for fixed $\beta$. This means that we have an effective
estimate of the function $h=h_\infty$ in the region $\bar\Omega_0$. 
In that context, we use periodicity to reduce to $-1\le\re\,\tau<1$,
and note that if $\tau\notin\bar\Omega_0$,
we find that for $N=\mathbbm{n}^\ast_{\mathrm{alg}}(\tau)$, the iterate
$(\mathbbm{g}_2^\ast)^N(\tau)$ is in $\bar\calD_\Theta\subset\bar\Omega_0$,
so that we may apply the estimate \eqref{eq:h_0_est1.01} at that point.
Moreover, by the monotonicity of the imaginary parts of the iterates, we have
\begin{equation}
\im\,(\mathbbm{g}_2^\ast)^N(\tau)>\im\,(\mathbbm{g}_2^\ast)^N-1(\tau)>\cdots
>\im\,\tau,  
\end{equation}
and, consequently, the estimate \eqref{eq:h_0_est1.01} entails that
\begin{equation}
|h_0\circ(\mathbbm{g}_2^\ast)^N(\tau)|
=\Ordo\Big(\exp\big(\beta\pi/\im(\mathbbm{g}_2^\ast)^N(\tau)\big)\Big)
=\Ordo\big(\exp(\beta\pi/\im\,\tau)\big)  
\label{eq:est_h0_Omega1.1.exp}
\end{equation}
holds uniformly for $\tau\in\Hyp\setminus\bar\Omega_0$ with $-1\le\re\,\tau<1$,
provided that $\beta\notin\Z_{>0}$. For $\beta\in\Z_{>0}$, we instead have the
estimate
\begin{equation}
|h_0\circ(\mathbbm{g}_2^\ast)^N(\tau)|
=\Ordo\Big(\Big(1+\frac{1}{\im\,\tau}\Big)\exp(\beta\pi/\im\,\tau)\Big), 
\label{eq:est_h0_Omega1.1.exp2}
\end{equation}
uniformly in $\tau\in\Hyp\setminus\bar\Omega_0$. 
We obtain the following proposition.

\begin{prop}
Suppose $f:\Hyp\to\C$ is holomorphic and meets the exponential growth bound
\eqref{eq:expbeta_est1} for a real parameter $\beta$ with $0<\beta<+\infty$.
Then, if $\beta$ is not an integer, there exists a $2$-periodic harmonic
function $h:\Hyp\to\C$ such that $h=f$ along the semicircle $\Te_+$, with
\begin{equation}
|h(\tau)|=\Ordo\big((1+\mathbbm{n}^\ast_{\mathrm{alg}}(\tau))
\big(\exp(\beta\pi/\im\,\tau)\big)\big),
\end{equation}
with an implied constant that is uniform in $\tau\in\Hyp$. On the other hand,
if $\beta$ is an integer, then there is still such a harmonic function, but the
growth estimate gets slightly worse:
\begin{equation}
|h(\tau)|=\Ordo\Big(\Big(1+\frac{1}{\im\,\tau}\Big)
\big(\exp(\beta\pi/\im\,\tau)\Big)\Big),
\end{equation}
uniformly in $\tau\in\Hyp$.
\label{prop:8.7.1}
\end{prop}

\begin{proof}
By the harmonic extension algorithm in Subsection
\ref{subsec:Solutionscheme-gen}, the function $h_0$ extends to harmonic
function $h=h_\infty:\Hyp\to\C$ that is $2$-periodic and solves $h=f$ along
$\Te_+$ (see Corollary \ref{cor:Hyperext1}). To obtain the claimed estimate,
we observe first that it holds on $\bar\Omega_0$, by
\eqref{eq:est_h0_Omega0.1}. 
In the remainder $\Hyp\setminus\bar\Omega_0$, we use the $2$-periodicity of
$h$ to restrict to the half-strip with $-1\le\re\,\tau<1$, and then we apply
Lemma \ref{prop:pointwiseest_iter} to get the asserted estimates from
\eqref{eq:expbeta_est2} as well as from the estimates
\eqref{eq:est_h0_Omega1.1.exp} and \eqref{eq:est_h0_Omega1.1.exp2} combined
with Remark \ref{rem:n*estimate}. 
\end{proof}

We now obtain the holomorphic hyperbolic Fourier series expansion
claimed in Theorem \ref{thm:function->coeff}. 

\begin{proof}[Proof of Theorem \ref{thm:function->coeff}]
If $h:\Hyp$ is the harmonic function of Proposition \ref{prop:8.7.1},  
we may appeal to Proposition \ref{prop:function->coeff1.0} with a slightly
bigger value of $\beta$. This guarantees that $h$ has a harmonic Fourier
series expansion \eqref{eq:Fourierseries1.01}:
\[
h(\tau)=\sum_{n\in\Z}c_n\,e_n(\tau),\qquad\tau\in\Hyp,
\]
where the coefficients are recovered by the formula \eqref{eq:get_coeff_cn},
that is, for any $s>0$, we have
\begin{equation}
c_n=\frac{\e^{\pi|n|s}}{2}\int_{-1}^{1}\e^{-\imag\pi nt}h(t+\imag s)\,\diff t,
\label{eq:get_coeff_cn01.ny}
\end{equation}
so that
\begin{equation}
|c_n|\le\frac{\e^{\pi|n|s}}{2}\int_{-1}^{1}|h(t+\imag s)|\,\diff t.
\label{eq:get_coeff_cn01.ny=ny}
\end{equation}
In case that $\beta$ is not an integer, we have from Proposition
\ref{prop:8.7.1} applied to \eqref{eq:get_coeff_cn01.ny=ny} that if $C>0$
is the implicit constant in the estimate of Proposition \ref{prop:8.7.1},
\begin{multline}
|c_n|\le C\,\e^{\pi|n|s}\exp(\beta\pi/s)\Big(1+\int_{-1}^{1}
\mathbbm{n}^\ast_{\mathrm{alg}}(t+\imag s)\,\frac{\diff t}{2}\Big)
\\
=C\,\e^{\pi|n|s}\exp(\beta\pi/s)\Big(\frac{1}{\pi^2}\log^2\frac{1}{s}+
\Ordo\Big(\log\frac{1}{s}\Big)\Big),
\label{eq:get_coeff_cn01.ny=ny2}
\end{multline}
provided $s>0$ is small enough. 
Here, we used the comparison of height functions in
Proposition \ref{prop:n=n*} as well as the asymptotic average height
calculation of Proposition \ref{prop:n_estimate_ave}.
For large $|n|$, we plug in $s=\beta^{\frac12}|n|^{-\frac12}$, which gives that
\[
|c_n|=\Ordo\big((\log^2|n|)\exp\big(2\pi\sqrt{\beta|n|}\big)\big)
\]
as $|n|\to+\infty$. In the remaining case when $\beta$ is an integer,
the analogous argument gives
\[
|c_n|=\Ordo\big(|n|^{\frac12}\exp\big(2\pi\sqrt{\beta|n|}\big)\big),
\]
again as $|n|\to+\infty$. Since $h=f$ holds on the semicircle $\Te_+$,
the Schwarz transform recovers $f$: $\shin\,h=f$. Moreover, the
harmonic Fourier series coefficients of $h$ become the holomorphic
hyperbolic Fourier series coefficients of $f$ in a natural fashion: $a_n=c_n$
for $n\ge0$, while $b_n=c_{-n}$ for $n>0$. The
proof is now complete.
\end{proof}

\section{Hyperbolic Fourier series expansion of a
point mass}
\label{sec:Diracmass}

\subsection{The Cauchy and Poisson kernels}
\label{subsec:Cauchy_Poisson}
For a given point $x\in\R$, we consider the associated \emph{Cauchy kernel}
\begin{equation}
f_x(\tau):=\frac{1}{\imag2\pi}\,\frac{1}{x-\tau},\qquad \tau\in\Hyp,
\label{eq:Cauchykernel1}
\end{equation}
whose boundary values on $\R$ in the sense of distribution theory are given by
\[
2f_x(t)=\delta_x(t)+\frac{\imag}{\pi}\,\mathrm{pv}\,\frac{1}{t-x},
\]
where ``pv'' indicates the \emph{principal value} interpretation of the
given function.
In particular, a direct calculation gives that
\[
2\re f_x(\tau)=P(\tau,x)=\pi^{-1}\frac{\im\tau}{|x-\tau|^2},\qquad \tau\in\Hyp,
\]
which is the Poisson kernel, and hence the harmonic extension of the unit 
point mass at $x$. 
Now, according to Theorem \ref{thm:hfs2} and its improvement
Corollary \ref{cor:8.5.2}, there exist coefficients $a_n=a_n(f_x)$ and
$b_n=b_n(f_x)$ with growth
\[
a_n,b_n=\Ordo\big((n+1)\log^2(2+n)\big)
\]
as $n\to+\infty$, such that
\[
f_x(\tau)=a_0+\sum_{n\in\Z_{>0}}\big(a_n\,\e^{\imag\pi n\tau}+b_n\,
\e^{-\imag\pi n/\tau}\big),\qquad\tau\in\Hyp.
\]
Moreover, as such, these coefficients are unique. Taking real parts on
both sides, we see that
\begin{multline}
\pi^{-1}\frac{\im\,\tau}{|x-\tau|^2}=
2\re\,f_x(\tau)=A_0(x)+\sum_{n\in\Z_{>0}}\big(A_n(x)\,\e^{\imag\pi n\tau}
+B_n(x)\,\e^{-\imag\pi n/\tau}\big)
\\
+\sum_{n\in\Z_{<0}}\big(A_n(x)\,\e^{-\imag\pi n\bar\tau}
+B_n(x)\,\e^{\imag\pi n/\bar\tau}\big),\qquad\tau\in\Hyp,
\end{multline}
where we write
\begin{equation}
A_0(x):=2\re\, a_0(f_x),\quad \forall n\in\Z_{>0}:\,\,\,A_n(x):=a_n(f_x)\,\,\,
\text{and}\,\,\,B_n(x):=b_n(f_x),
\label{eq:A_n_B_n}
\end{equation}
and the coefficient functions get extended to negative indices via the 
symmetry property
\begin{equation}
A_n(x)=\bar A_{-n}(x),\quad B_n(x)=\bar B_{-n}(x).
\label{eq:basicsymmetryA_nB_n}
\end{equation}
Here, the "bar" indicates complex conjugation of the whole expression. 
The Cauchy kernel has the functional properties
\begin{equation}
\bar f_{-x}\circ\Rtrans^\ast(\tau)=\bar f_{-x}(-\bar\tau)=f_x(\tau),\qquad
\tau\in\Hyp,
\label{eq:functionalprop0}
\end{equation}
and, for $x\ne0$,
\begin{equation}
f_x\circ\Strans(\tau)=f_x(-1/\tau)=\frac{1}{\imag 2\pi x}+
x^{-2}f_{-1/x}(\tau),\qquad\tau\in\Hyp.
\label{eq:functionalprop1}
\end{equation}
As we plug in the hyperbolic Fourier series expansion on both sides of
\eqref{eq:functionalprop0}, we find that
\begin{multline} 
\bar f_{-x}(-\bar\tau)=
\bar a_0(f_{-x})+\sum_{n\in\Z_{>0}}\big(\bar a_n(f_{-x})\,
\e^{\imag\pi n\tau}+\bar b_n(f_{-x})\,\e^{-\imag\pi n/\tau}\big)
\\
=a_0(f_{x})+
\sum_{n\in\Z_{>0}}\big(a_n(f_{x})\,\e^{\imag\pi n\tau}
+b_n(f_{x})\,\e^{-\imag\pi n/\tau}\big)
=f_{x}(\tau),\qquad\tau\in\Hyp.
\label{eq:identifycoeff0}
\end{multline}
On the other hand, if we instead work with the identity 
\eqref{eq:functionalprop1} we obtain that 
\begin{multline}
f_x(-1/\tau)=
a_0(f_x)+\sum_{n\in\Z_{>0}}\big(a_n(f_x)\,\e^{-\imag\pi n/\tau}
+b_n(f_x)\,\e^{\imag\pi n\tau}\big)
\\
=\frac{1}{\imag2\pi x}+x^{-2}a_0(f_{-1/x})+
x^{-2}\sum_{n\in\Z_{>0}}\big(a_n(f_{-1/x})\,\e^{\imag\pi n\tau}
+b_n(f_{-1/x})\,\e^{-\imag\pi n/\tau}\big)
\\
=\frac{1}{\imag2\pi x}+x^{-2}f_{-1/x}(\tau),\qquad\tau\in\Hyp.
\label{eq:identifycoeff1}
\end{multline}
Both sides of \eqref{eq:identifycoeff0} and \eqref{eq:identifycoeff1}
are hyperbolic Fourier series, so in view of Theorem \ref{thm:uniq1},
the corresponding coefficients must coincide, given that the coefficient
growth is rather modest. These two identities lead to symmetry properties 
of the functions $A_n$ and $B_n$. 

\begin{prop} {\rm{(symmetry properties)}}

\noindent {\rm(a)} We have that  $A_0(x)$ is real-valued and even, so that 
$A_0(-x)=A_0(x)$ holds for each $x\in\R$. Moreover, for $x\in\R_{\ne0}$, 
we have that $A_0(x)=x^{-2}A_0(-1/x)$. 

\noindent{\rm(b)} For $n\in\Z_{\ne0}$, the functions $A_n(x),B_n(x)$ have the symmetry 
properties that for each $x\in\R$, 
\[
A_n(x)=\bar A_{-n}(x)=A_{-n}(-x),\quad B_n(x)=\bar B_{-n}(x)=B_{-n}(-x),
\]
and, moreover, for $x\in\R_{\ne0}$, they are connected via 
\[
B_n(x)=x^{-2}A_n(-1/x).
\]
\label{prop:symmetry011}
\end{prop}

\begin{rem}
It is a consequence of the relations \eqref{eq:identifycoeff0} and 
\eqref{eq:identifycoeff1} that 
\[
\bar a_0(f_{-x})=a_0(f_x),\quad a_0(f_x)=\frac{1}{\imag2\pi x}
+x^{-2}a_0(f_{-1/x}),
\]
which at least looks superficially stronger than the corresponding assertion 
of part (a) to the effect that $A_0(-x)=a_0(x)$ and  $A_x(t)=x^{-2}A_0(-1/x)$.
\end{rem}

\begin{proof}[Proof of Proposition \ref{prop:symmetry011}]
The assertions are immediate consequences of 
\eqref{eq:basicsymmetryA_nB_n} and
the identities \eqref{eq:identifycoeff0}, \eqref{eq:identifycoeff1}, by identifying 
the coefficients in the two holomorphic Fourier series expansions. 
\end{proof}

\subsection{Hyperbolic Fourier series expansion of the Cauchy kernel}
\label{subsec:HFS_Cauchy}

For $x\in\R$, we let $h_x$ denote the function
\begin{equation}
\label{eq:function_h_x}
h_x(\tau):=\frac{1}{\pi}\int_{\Te_+}
\frac{\frac12-\re\,\lambda(\tau)}{|\lambda(\tau)-\lambda(\eta)|^2}\,f_x(\eta)\,
|\lambda'(\eta)|\,|\diff\eta|,\qquad\tau\in\calD_\Theta.
\end{equation}
Here, $f_x$ denotes the Cauchy kernel given by \eqref{eq:Cauchykernel1}.
Any possible worries about the convergence of the integral 
\eqref{eq:function_h_x} are dispelled by the relations \eqref{eq:lambdaproperty3}
and \eqref{eq:lambdaproperty4} combined with \eqref{eq:lambda'property3} 
and the comment thereafter. 
From the properties of the Poisson kernel in the upper half-plane 
\eqref{eq:Poissonkernel1.1} and the mapping properties of $\lambda$ outlined
in Proposition \ref{prop:lambda-classical}, it follows that $h_x$ is harmonic in the
fundamental domain $\calD_\Theta$, and that its boundary values on the 
semicircle $\Te_+$ equal those of $f_x$, so that $h_x=f_x$ on $\Te_+$. 
Moreover, as the function $h_x$ may be interpreted as the composition of the 
harmonic function
\[
H_x(\zeta):=\frac{1}{\pi}\int_{\Te_+}
\frac{\frac12-\re\,\zeta}{|\zeta-\lambda(\eta)|^2}\,f_x(\eta)\,
|\lambda'(\eta)|\,|\diff\eta|,\qquad\zeta\in\C_{\re<\frac12},
\] 
with the modular lambda function, that is, $h_x=H_x\circ\lambda$,  we conclude
from the mapping properties and the $2$-periodicity of $\lambda$ that $h_x$
extends to a $2$-periodic harmonic function in the domain $\Omega_0$ given by
\eqref{eq:Omega_0.01}. This puts us in the setting of the algorithm of Subsection
\ref{subsec:Solutionscheme-gen}, with $f_x$ as $f$ and $h_x$ as $h_0$. By 
Corollary \ref{cor:Hyperext1}, $h_x$ extends harmonically to all of $\Hyp$, and by
Corollary \ref{cor:HFSthm1}, the fact that $h_x=f_x$ holds on $\Te_+$  tells us that
$f_x=\shin\,h_x$ is the holomorphic hyperbolic Fourier series expansion of $f_x$.
Moreover, the harmonic Fourier series coefficients of $h_x$ give us the 
holomorphic hyperbolic Fourier series coefficients of $f_x$. 
We proceed to obtain expressions for the harmonic Fourier coefficients 
$c_n=c_n(h_x)$ associated with the expansion of $h_x$:
\begin{equation}
h_x(\tau)=\sum_{n\in\Z}c_n\,e_n(\tau)=c_0+
\sum_{n\in\Z_{>0}}c_n\,\e^{\imag\pi n\tau}+
\sum_{n\in\Z_{<0}}c_n\,\e^{\imag\pi n\bar\tau},\qquad\tau\in\Hyp.
\end{equation}
The corresponding holomorphic hyperbolic Fourier series expansion of the Cauchy
kernel $f_x$ is then
\begin{equation}
f_x(\tau)=\shin\,h_x(\tau)=c_0+
\sum_{n\in\Z_{>0}}c_n\,\e^{\imag\pi n\tau}+
\sum_{n\in\Z_{<0}}c_n\,\e^{\imag\pi n/\tau},\qquad\tau\in\Hyp,
\end{equation}
which means that the coefficients $a_n(f_x)$ and $b_n(f_x)$ are given by
\begin{equation}
a_0(f_x)=c_0(h_x),\,\,\,\,\forall n\in\Z_{>0}:\,\,\,a_n(f_x)=c_n(h_x),\,\,\,
b_n(f_x)=c_{-n}(h_x).
\label{eq:a_n_b_n_c_n}
\end{equation}
Clearly, the function $f_x$ meets the uniform growth bound 
\[
|f_x(\tau)|=\Ordo(\im\,\tau)^{-1},\qquad\tau\in\Hyp, 
\]
which is uniform not only in $\tau$ but also in $x\in\R$.
In particular, the growth condition \eqref{eq:polgrowth0*} holds with parameter
value $k=1$, so that by 
Corollary \ref{cor:8.5.2}, we know that 
\begin{equation}
a_n(f_x),b_n(f_x)=\Ordo\big((n+1)\log^2(2+n)\big)
\label{eq:coefficient_growth_a_n_b_n798}
\end{equation}
holds uniformly in $n$ and $x\in\R$. We can improve this estimate for 
large values of $|x|$. 

\begin{prop}
We have the growth control
\[
A_n(x),\,B_n(x)=\Ordo\Big(\frac{(|n|+1)\log^2(2+|n|)}{1+x^2}\Big),
\]
uniformly in $n\in\Z$ and $x\in\R$.
\label{prop:improved_bound_A_n_B_n999}
\end{prop}

\begin{proof}
The assertion is immediate from \eqref{eq:coefficient_growth_a_n_b_n798}
combined with the symmetry properties of Proposition \ref{prop:symmetry011}.
\end{proof}

We now attempt to calculate the coefficient functions $A_n(x)$ and $B_n(x)$. 
To obtain these coefficients, we appeal to formula \eqref{eq:get_coeff_cn}, 
which asserts that for any $s>0$, we have
\begin{equation}
c_n(h_x)=\frac{\e^{\pi|n|s}}{2}\int_{-1}^{1}\e^{-\imag\pi nt}
h_x(t+\imag s)\,\diff t,\qquad n\in\Z.
\label{eq:get_coeff_cn01.ny_h_x}
\end{equation}
As we implement the formula \eqref{eq:function_h_x} which defines $h_x$ in
$\calD_\Theta$, the formula \eqref{eq:get_coeff_cn01.ny_h_x} becomes, for
$s>1$,
\begin{equation}
c_n(h_x)=\frac{\e^{\pi|n|s}}{2\pi}\int_{-1}^{1}\int_{\Te_+}\e^{-\imag\pi nt}
\frac{\frac12-\re\,\lambda(t+\imag s)}{|\lambda(t+\imag s)-
\lambda(\eta)|^2}\,f_x(\eta)\,|\lambda'(\eta)|\,|\diff\eta|\,\diff t,\qquad n\in\Z.
\label{eq:get_coeff_cn01.ny_h_x.2}
\end{equation}
Fubini's theorem allows us to interchange the integrations, which leaves us 
with the formula (for $n\in\Z$)
\begin{equation}
c_n(h_x)=\frac{\e^{\pi|n|s}}{2\pi}\int_{\Te_+}\int_{-1}^{1}\e^{-\imag\pi nt}
\frac{\frac12-\re\,\lambda(t+\imag s)}{|\lambda(t+\imag s)-
\lambda(\eta)|^2}\,\diff t\,f_x(\eta)\,|\lambda'(\eta)|\,|\diff\eta|.
\label{eq:get_coeff_cn01.ny_h_x.3}
\end{equation}
The instance $n=0$ is of particular interest:
\begin{equation}
c_0(h_x)=\frac{1}{2\pi}\int_{\Te_+}\int_{-1}^{1}
\frac{\frac12-\re\,\lambda(t+\imag s)}{|\lambda(t+\imag s)-
\lambda(\eta)|^2}\,\diff t\,f_x(\eta)\,|\lambda'(\eta)|\,|\diff\eta|
\label{eq:get_coeff_cn01.ny_h_x.4}
\end{equation}
In view of the asymptotics \eqref{eq:lambdaf2.02}, it follows that
$\lambda(t+\imag s)\to0$ uniformly in $t\in\R$ as $s\to+\infty$, 
so that by letting $s\to+\infty$ in \eqref{eq:get_coeff_cn01.ny_h_x.4},
we find that
\begin{equation}
a_0(f_x)=c_0(h_x)=\frac{1}{2\pi}\int_{\Te_+}
\frac{|\lambda'(\eta)|}{|\lambda(\eta)|^2}\,f_x(\eta)\,|\diff\eta|.
\label{eq:get_coeff_cn01.ny_h_x.5}
\end{equation}
As we take real parts in \eqref{eq:get_coeff_cn01.ny_h_x.5},
we obtain an expression for $A_0$.

\begin{prop}
\label{prop:A_0_formula999}
We have the formula for $A_0$:
\begin{multline}
A_0(x)=2\re\,a_0(f_x)=\frac{1}{2\pi}\int_{\Te_+}
\frac{|\lambda'(\eta)|}{|\lambda(\eta)|^2}\,P(\eta,x)\,|\diff\eta|
=\frac{1}{2\pi^2}\int_{\Te_+}
\frac{|\lambda'(\eta)|}{|\lambda(\eta)|^2}\,\frac{\im\,\eta}{|x-\eta|^2}
\,|\diff\eta|,\qquad x\in\R,
\label{eq:get_coeff_cn01.ny_h_x.6}
\end{multline}
so that  $A_0(x)>0$ for each $x\in\R$. Moreover, the function $A_0$ 
is $C^\infty$-smooth on $\R$.
\end{prop}

\begin{proof}
Taking real parts in \eqref{eq:get_coeff_cn01.ny_h_x.5} gives us
the indicated formula for $A_0$, and the pointwise positivity is
immediate from the formula. As for the smoothness, we calculate
successive derivatives in \eqref{eq:get_coeff_cn01.ny_h_x.5}:
\[
\partial_x^j a_0(f_x)=\frac{1}{2\pi}\int_{\Te_+}
\frac{|\lambda'(\eta)|}{|\lambda(\eta)|^2}\,\partial_x^j f_x(\eta)\,
|\diff\eta|
=\frac{(-1)^j}{\imag 4\pi^2}\int_{\Te_+}
\frac{|\lambda'(\eta)|}{|\lambda(\eta)|^2}\, (x-\eta)^{-j-1}\,
|\diff\eta|
\] 
which integrals converge nicely by the known asymptotics of $\lambda$
and $\lambda'$ near the cusp points at $\pm1$ (see 
\eqref{eq:lambdaproperty3}, \eqref{eq:lambdaproperty4}, 
\eqref{eq:lambda'property3}). Taking real parts, this smoothness carries 
over to $A_0$.
\end{proof}

We proceed to evaluate
$A_n$ for $n>0$, since the symmetry properties of  Proposition 
\ref{prop:symmetry011} then give $A_n$ and $B_n$ for all $n\ne0$. 
In view of \eqref{eq:A_n_B_n} and \eqref{eq:a_n_b_n_c_n} in combination 
with \eqref{eq:get_coeff_cn01.ny_h_x.3}, we have, for $s>1$, 
\begin{multline}
A_n(x)=a_n(f_x)=c_n(h_x)
\\
=\frac{1}{2\pi}
\int_{\Te_+}\int_{-1}^{1}\e^{-\imag\pi n(t+\imag s)}
\frac{\frac12-\re\,\lambda(t+\imag s)}{|\lambda(t+\imag s)-
\lambda(\eta)|^2}\,\diff t\,f_x(\eta)\,|\lambda'(\eta)|\,|\diff\eta|
\\
=\frac{1}{2\pi}
\int_{\Te_+}\int_{\gamma_s}\e^{-\imag\pi n\tau}
\frac{\frac12-\re\,\lambda(\tau)}{|\lambda(\tau)-
\lambda(\eta)|^2}\,\diff \tau\,f_x(\eta)\,|\lambda'(\eta)|\,|\diff\eta|,
\qquad n\in\Z_{>0},
\label{eq:get_coeff_cn01.ny_h_x.7}
\end{multline}
provided that the line segment $\gamma_s:=[-1,1]+\imag s$
oriented from left to right.
The mapping properties of the modular lambda function presented in 
Proposition \ref{prop:lambda-classical} entail that 
$\re\,\lambda(\eta)=\frac12$ for $\eta\in\Te_+$, and we calculate that
\begin{equation}
\frac{\frac12-\re\,\lambda(\tau)}{|\lambda(\tau)-\lambda(\eta)|^2}
=\frac{1-\lambda(\tau)}{2(1-\lambda(\eta))(\lambda(\eta)-\lambda(\tau))}
+\frac{\bar\lambda(\tau)}
{2\bar\lambda(\eta)(\bar\lambda(\eta)-\bar\lambda(\tau))}.
\label{eq:Poisson_lambdadecomp}
\end{equation}
The mapping properties of $\lambda$ also show that 
as $\eta$ moves along the semicircle $\Te_+$ in the counterclockwise
direction, we have the equality of forms
\begin{equation}
\imag\,\lambda'(\eta)\,\diff\eta=-\imag\,\bar\lambda'(\eta)\,\diff\bar\eta=
|\lambda'(\eta)|\,|\diff\eta|.
\label{eq:Poisson_lambdadecomp2}
\end{equation}
It now follows from \eqref{eq:Poisson_lambdadecomp} that
\begin{multline}
\int_{\gamma_s}\e^{-\imag\pi n\tau}
\frac{\frac12-\re\,\lambda(\tau)}{|\lambda(\tau)-
\lambda(\eta)|^2}\,\diff \tau
\\
=\frac12\int_{\gamma_s}\e^{-\imag\pi n\tau}
\Big(\frac{1-\lambda(\tau)}
{(1-\lambda(\eta))(\lambda(\eta)-\lambda(\tau))}
+\frac{\bar\lambda(\tau)}
{\bar\lambda(\eta)(\bar\lambda(\eta)-\bar\lambda(\tau))}\Big)
\,\diff \tau
\\
=\frac1{2(1-\lambda(\eta))}\int_{\gamma_s}\e^{-\imag\pi n\tau}
\,\frac{1-\lambda(\tau)}
{\lambda(\eta)-\lambda(\tau)}
\,\diff \tau,
\qquad n\in\Z_{>0}.
\label{eq:Poisson_lambdadecomp3}
\end{multline}
The omitted term in the last step of \eqref{eq:Poisson_lambdadecomp3}
vanishes, as it is the complex conjugate of
\begin{multline}
\frac12\int_{\gamma_s}\e^{\imag\pi n\bar\tau}
\frac{\lambda(\tau)}{\lambda(\eta)(\lambda(\eta)-\lambda(\tau))}
\,\diff \tau
=\frac{\e^{2\pi ns}}{2\lambda(\eta)}\int_{\gamma_s}
\e^{\imag\pi n\tau}
\frac{\lambda(\tau)}{\lambda(\eta)-\lambda(\tau)}
\,\diff \tau=0,
\\
\label{eq:Poisson_lambdadecomp4}
\end{multline}
where the latter equality holds because the integral 
\begin{equation}
\int_{\gamma_s}\e^{\imag\pi n\tau}
\frac{\lambda(\tau)}{\lambda(\eta)-\lambda(\tau)}
\,\diff \tau
\label{eq:Poisson_lambdadecomp5}
\end{equation} 
is independent of $s>1$ (by deformation of countours), and at the same 
time it tends to $0$ as $s\to+\infty$, for each $n>0$. Moreover, a 
calculation gives that
\[
\partial_\tau\,\frac{1-\lambda(\tau)}{\lambda(\eta)-\lambda(\tau)}=
(1-\lambda(\eta))\,\frac{\lambda'(\tau)}{(\lambda(\eta)-\lambda(\tau))^2},
\]
so that by integration by parts and periodicity,
\begin{equation}
\int_{\gamma_s}\e^{-\imag\pi n\tau}\,
\frac{1-\lambda(\tau)}{\lambda(\eta)-\lambda(\tau)}\,\diff\tau
=\frac{1-\lambda(\eta)}{\imag\pi n}\int_{\gamma_s}\e^{-\imag\pi n\tau}
\,\frac{\lambda'(\tau)}{(\lambda(\eta)-\lambda(\tau))^2}\,\diff \tau.
\label{eq:Poisson_lambdadecomp6}
\end{equation}
If we combine \eqref{eq:get_coeff_cn01.ny_h_x.7} with 
\eqref{eq:Poisson_lambdadecomp2}, \eqref{eq:Poisson_lambdadecomp3}, 
and \eqref{eq:Poisson_lambdadecomp6}, we arrive at the formula
\begin{multline}
A_n(x)
=\frac{1}{2\pi}
\int_{\Te_+}\int_{\gamma_s}\e^{-\imag\pi n\tau}
\frac{\frac12-\re\,\lambda(\tau)}{|\lambda(\tau)-
\lambda(\eta)|^2}\,\diff \tau\,f_x(\eta)\,|\lambda'(\eta)|\,|\diff\eta|
\\
=\frac{1}{4\pi^2n}\int_{\Te_+}\int_{\gamma_s}\e^{-\imag\pi n\tau}
\,\frac{\lambda'(\tau)}{(\lambda(\eta)-\lambda(\tau))^2}\diff\tau\,
f_x(\eta)\lambda'(\eta)\,\diff\eta,
\qquad n\in\Z_{>0},
\label{eq:get_coeff_cn01.ny_h_x.7.1}
\end{multline}
for any $s>1$. 

We shall need the polynomials $S_n$ defined by the following property.
 
\begin{prop}
\label{prop:pols_S_n}
For each $n\in\Z_{>0}$, there exists a unique polynomial $S_n$ of 
degree $n$ with $S_n(0)=0$,
such that 
\[
\e^{-\imag\pi n\tau}-S_n(1/\lambda(\tau))=\Ordo(1)
\]
holds as $\im\,\tau\to+\infty$.
\end{prop}
 
\begin{proof}
In terms of the nome variable $q=\e^{\imag\pi\tau}$, the lambda function
is given by 
\[
\lambda(\tau)=\sum_{n=1}^{+\infty}\hat\lambda(n)\,q^n,
\] 
where $\hat\lambda(1)=16$, and the series converges absolutely for 
$|q|<1$, with a simple zero at $q=0$ and no other zeros in the disk 
$|q|<1$. As a consequence, for $k\in\Z_{>0}$,  the function 
\[
\frac{1}{(\lambda(\tau))^k}
\]
is holomorphic in $q\in\D\setminus\{0\}$ with a pole of order $k$ at $q=0$. 
Consequently, modulo the holomorphic functions in the variable $q\in\D$, 
the functions $1/(\lambda(\tau))^k$, $k=1,\ldots,n$ span the same vector 
space as $q^{-j}$ for $j=1,\ldots,n$. In particular, there exists a unique 
polynomial $S_n$ of degree $n$ with $S_n(0)=0$ such that
\[
q^{-n}-S_n(1/\lambda(\tau))
\]
is holomorphic at the point $q=0$. Finally, for holomorphic functions in
a punctured disk, boundedness at the puncture is the same as extending
holomorphically across it.
\end{proof}

In terms of the polynomial $S_n$, we obtain expressions for $A_n$ and $B_n$, 
for $n>0$.

\begin{prop}
\label{prop:formula_A_n_B_n999}
For $n\in\Z_{>0}$, the function $A_n$ is given by
\[
A_n(x)=-\frac{1}{4\pi^2n}\int_{\Te_+}\frac{S_n(1/\lambda(\eta))}{(x-\eta)^2}\,
\diff\eta,\qquad x\in\R,
\]
while the function $B_n$ is given by 
\[
B_n(x)=-\frac{1}{4\pi^2n}\int_{\Te_+}
\frac{S_n(1/\lambda(\eta))}{(1+x\eta)^2}\,\diff\eta,\qquad x\in\R.
\]
As a consequence, the functions $A_n$ and $B_n$ are both 
$C^\infty$-smooth on $\R$.
\end{prop}

\begin{proof}
We split in accordance with Proposition \ref{prop:pols_S_n}
the function 
\begin{equation}
\e^{-\imag\pi n\tau}=S_n(1/\lambda(\tau))+R_n(\lambda(\tau)),
\end{equation}
or, equivalently, 
\begin{equation}
\e^{-\imag\pi n\lambda^{-1}(\zeta)}=S_n(1/\zeta)+R_n(\zeta),
\label{eq:decomp_explambda}
\end{equation}
where $R_n(\zeta)$ is holomorphic and bounded in a neighborhood of  $0$. 
The reason that we may represent the remainder as a function of
$\lambda(\tau)$ is that $\e^{-\imag\pi n\tau}$ is $2$-periodic, while
$\lambda(\tau)$ is $2$-periodic too and has the mapping properties outlined
in Proposition \ref{prop:lambda-classical}. Indeed, we see that the function 
$R_n(\zeta)$ extends holomorphically to the slit plane $\C\setminus[1,+\infty[$. 
As we apply the change-of-variables formula to the inner integral in 
\eqref{eq:get_coeff_cn01.ny_h_x.7.1}, we find that
\begin{equation}
\int_{\gamma_s}\e^{-\imag\pi n\tau}
\,\frac{\lambda'(\tau)}{(\lambda(\eta)-\lambda(\tau))^2}\diff\tau\,
=\int_{\lambda(\gamma_s)}\e^{-\imag\pi n\lambda^{-1}(\zeta)}
\,\frac{\diff\zeta}{(\lambda(\eta)-\zeta)^2}
\label{eq:decomp_explambda2}
\end{equation}
where the contour $\lambda(\gamma_s)$ is a closed loop in the half-plane 
$\C_{\re<\frac12}$ which goes around the origin once in the positive direction 
(counter-clockwise). In view of \eqref{eq:decomp_explambda}, we can handle
the expressions $S_n(1/\zeta)$ and $R_n(\zeta)$ separately. By Cauchy's theorem,
\begin{equation}
\int_{\lambda(\gamma_s)}R_n(\zeta)
\,\frac{\diff\zeta}{(\lambda(\eta)-\zeta)^2}=0,\qquad \eta\in\Te_+,
\end{equation}
since $\lambda(\eta)\in\frac12+\imag\R$ is exterior to the loop 
$\lambda(\gamma_s)$. On the other hand, Cauchy's theorem for the domain
exterior to the loop $\lambda(\gamma_s)$ gives that 
\begin{equation}
\int_{\lambda(\gamma_s)}S_n(1/\zeta)
\,\frac{\diff\zeta}{(\lambda(\eta)-\zeta)^2}=
-\imag2\pi\,\partial_\zeta\,S_n(1/\zeta)\Big|_{\zeta:=\lambda(\eta)},
\qquad \eta\in\Te_+.
\end{equation}
It now follows from \eqref{eq:decomp_explambda2} and 
\eqref{eq:decomp_explambda2} together with the above two computations
that
\begin{equation}
\int_{\gamma_s}\e^{-\imag\pi n\tau}
\,\frac{\lambda'(\tau)}{(\lambda(\eta)-\lambda(\tau))^2}\diff\tau\,
=-\imag2\pi\,\partial_\zeta\,S_n(1/\zeta)\Big|_{\zeta:=\lambda(\eta)}
=\imag2\pi\,\lambda(\eta)^{-2}\,S'_n(1/\lambda(\eta)),
\label{eq:decomp_explambda3}
\end{equation}
so that \eqref{eq:get_coeff_cn01.ny_h_x.7.1} gets simplified to
(recall the definition \eqref{eq:Cauchykernel1} of $f_x$)
\begin{multline}
A_n(x)
=\frac{1}{4\pi^2n}\int_{\Te_+}\int_{\gamma_s}\e^{-\imag\pi n\tau}
\,\frac{\lambda'(\tau)}{(\lambda(\eta)-\lambda(\tau))^2}\diff\tau\,
f_x(\eta)\lambda'(\eta)\,\diff\eta
\\
=\frac{1}{4\pi^2n}\int_{\Te_+}
\frac{\lambda'(\eta)S_n'(1/\lambda(\eta))}{\lambda(\eta)^2}
\frac{\diff\eta}{x-\eta},
\qquad n\in\Z_{>0},
\label{eq:get_decomp_explambda4}
\end{multline}
where $s>1$ in the middle expression. Finally, by integration by parts, 
it follows that
\begin{multline}
A_n(x)
=\frac{1}{4\pi^2n}\int_{\Te_+}
\frac{\lambda'(\eta)S_n'(1/\lambda(\eta))}{\lambda(\eta)^2}
\frac{\diff\eta}{x-\eta}
=-\frac{1}{4\pi^2n}\int_{\Te_+}
S_n(1/\lambda(\eta))\,\frac{\diff\eta}{(x-\eta)^2},
\qquad n\in\Z_{>0},
\label{eq:get_decomp_explambda5}
\end{multline}
as claimed, but we need to motivate why there is no contributions from the
endpoints. The reason is that according to \eqref{eq:lambdaproperty3}
and \eqref{eq:lambdaproperty4}, the function $1/\lambda(\eta)$ decays rapidly
to $0$ as $\eta\in\Te_+$ approaches the endpoints $\pm1$ while the 
polynomial $S_n$ vanishes at the origin, and the function $(x-\eta)^{-1}$ cannot
perturb much this rapid decay to $0$. This establishes the formula for $A_n$, 
and the formula for $B_n(x)$ follows from Proposition \ref{prop:symmetry011}
for $x\in\R_{\ne0}$, and for $x=0$ by continuity in the parameter $x$.
As for the smoothness, we may represent successive derivatives of $A_n$ and
$B_n$ as absolutely convergent integrals:
 \begin{equation}
A_n^{(j)}(x)
=\frac{(-1)^{j+1}(j+1)!}{4\pi^2n}\int_{\Te_+}
S_n(1/\lambda(\eta))\,\frac{\diff\eta}{(x-\eta)^{j+2}},
\qquad n\in\Z_{>0},
\label{eq:get_decomp_explambda6}
\end{equation}
for $j=0,1,2,\ldots$, and likewise for $B_n$. This settles the smoothness issue
and completes the proof.
\end{proof}

\begin{rem}
We note that if $\lambda^{-1}$ is  (as in Subsection
\ref{subsec:lambda-map}) the local inverse mapping to $\lambda$, we may apply
the change-of-variables
$\zeta=\lambda(\eta)$ in Proposition \ref{prop:formula_A_n_B_n999}
and arrive at the representation (the contour is assumed directed upwards)
\begin{multline}
A_n(x)=\frac{1}{4\pi^2n}\int_{\frac12+\imag\R}
\frac{S_n(1/\zeta)}{(x-\lambda^{-1}(\zeta))^2}\,
(\lambda^{-1})'(\zeta)\,\diff\zeta
\\
=\frac{1}{\imag 4\pi^3n}\int_{\frac12+\imag\R}
\frac{S_n(1/\zeta)}{\zeta(1-\zeta)\big(xF_{\frac12,\frac12;1}(\zeta)-
\imag F_{\frac12,\frac12;1}(1-\zeta)\big)}\,
\,\diff\zeta,\qquad x\in\R,
\end{multline}
with an analogous formula for $B_n$ as well. Here, the Gauss hypergeometric
function appears, as it is well-known that the local
inverse $\lambda^{-1}$ may be expressed in terms of a ratio of
hypergeometric functions (see, e.g., \cite{bh2}).
\end{rem}

The explicit representation of the functions $A_n,B_n$ has leads to
a simplified formula for interpolating basis $u_{(n,0)}$ and $u_{(0,n)}$ in
spacelike quarterplanes. We formulate the statement for $u_{(n,0)}$ with $n\ge0$
only. As for the notation, see Theorem \ref{thm:PBeurling}.

\begin{cor}
\label{cor:PBeurlingformula}
The function $u_{(0,0)}$ is real-valued and has the symmetry property
$u_{(0,0)}(x,y)=u_{(0,0)}(-x,-y)=u_{(0,0)}(y,x)$,
and moreover, it is expressed by the formula
\[
u_{(0,0)}(x,y)=\frac{1}{2\pi}\int_{\Te_+}
\frac{|\lambda'(\eta)|}{|\lambda(\eta)|^2}\,\e^{\imag x\eta+\imag y/\eta}
\,|\diff\eta| \quad\text{if}\quad x>0,\,y<0.
\]
Furthermore, for $n\in\Z_{>0}$, the function $u_{(n,0)}$ is real-valued,
and in the spacelike quarter-planes we have that
\[
u_{(n,0)}(x,y)=0\quad\text{if} \quad x<0,\,\,y>0,
\]
while
\[
u_{(n,0)}(x,y)=\frac{1}{2\pi n}\int_{\Te_+}
S_n(1/\lambda(\eta))\,(x-y\eta^{-2})\,\e^{\imag x\eta+\imag y/\eta}\,
\diff\eta\quad\text{if} \quad x>0,\,\,y<0.
\]  
\end{cor}

\begin{proof}
We focus our attention to $n>0$, the case $n=0$ being analogous.
The real-valuedness is a consequence of the symmetry properties of
$A_n$ as presented in Proposition \ref{prop:symmetry_hfs_first1}. 
In view of Proposition \ref{prop:formula_A_n_B_n999} combined with the
definition of $u_{(n,0)}$ in Theorem \ref{thm:PBeurling}, we have that
\begin{equation}
u_{(n,0)}(x,y)=\int_\R \e^{\imag xt+\imag y/t}A_n(t)\,\diff t
=-\frac{1}{4\pi^2n}\int_{\Te_+}S_n(1/\lambda(\eta))\int_{\R}
\frac{\e^{\imag xt+\imag y/t}}{(t-\eta)^2}\,\diff t\diff \eta,  
\end{equation}
where in the last step we used Fubini's theorem to interchange the order of
integration. When $x<0,\,y>0$, the function $t\mapsto\e^{\imag xt+\imag y/t}$
extends to a bounded holomorphic function in the lower half-plane,
and we may deform the contour to see that
\[
\int_{\R}
\frac{\e^{\imag xt+\imag y/t}}{(t-\eta)^2}\,\diff t=0,\qquad \eta\in\Hyp,
\]
from which $u_{(n,0)}(x,y)=0$ is immediate. On the other hand, when $x>0,\,y>0$,
we have instead that
\[
\int_\R \frac{\e^{\imag xt+\imag y/t}}{(t-\eta)^2}\,\diff t=
-2\pi\,(x-y\eta^{-2})\,\e^{\imag x\eta+\imag y/\eta},\qquad\eta\in\Hyp,  
\]
by deforming the contour in the upper half-plane,
which gives the corresponding expression for $u_{(n,0)}(x,y)$. 
\end{proof}

\begin{proof}[Proof of Corollary \ref{cor:hfs1.1}]
The fact that the functions $A_0$ and $A_n,B_n$ for $n\in\Z_{\ne0}$ belong to the
test function space $C^\infty_1(\R\cup\{\infty\})$ follows from  
Propositions \ref{prop:A_0_formula999} and 
\ref{prop:formula_A_n_B_n999} combined with the symmetry properties
of Proposition \ref{prop:symmetry011}. The growth control on the
coefficient functions
can be found in Proposition \ref{prop:improved_bound_A_n_B_n999}.
The expansion of the Dirac point mass $\delta_x$ is a direct consequence of
\eqref{eq:A_n_B_n} if we take into account the correspondence between
distributions
on the extended real line and harmonic functions in $\Hyp$ outlined
in Subsection \ref{subsec:harmext-gen}.
\end{proof}

\begin{proof}[Proof of Corollary \ref{cor:biorthogonal_2.3.5}]
Since Proposition \ref{prop:symmetry011} asserts that $\Jop_1 A_0=A_0$ holds,  
while  $\Jop_1 A_n=B_n$ and consequently $\Jop_1 B_n=A_n$ holds for 
$n\in\Z_{\ne0}$, we may focus our attention to the assertions that
\[
\langle A_0,e_m\rangle_\R=\delta_{m,0},\qquad m\in\Z,
\]
while 
\[
\langle A_n,e_m\rangle_\R=\delta_{m,n},\qquad m\in\Z,\,\,n\in\Z_{\ne0},
\]
and 
\[
\langle B_n,e_m\rangle_\R=0,\qquad m\in\Z,\,\,n\in \Z_{\ne0}.
\]

We note that the quantities $\langle A_n,e_m\rangle_\R$ and
$\langle B_n,e_m\rangle_\R$ are well-defined in the usual integral sense 
by the estimate in \ref{cor:hfs1.1}, and that we in fact have
\begin{equation}
\langle A_n,e_m\rangle_\R,\,\,\langle B_n,e_m\rangle_\R
=\Ordo\big((|n|+1)\log^2(2+|n|)\Big)
\label{eq:coeff_e_m1}
\end{equation}
uniformly in $n$.
Next, we observe that in the sense of distribution theory,
we may represent
\[
e_m(t)=\int_\R e_m(x)\,\delta_x(t)\,\diff x,
\] 
so by the expansion of the Dirac point mass in Corollary \ref{cor:hfs1.1}
we have the hyperbolic Fourier series expansion
\[
e_m(t)=\e^{\imag\pi mt}=\langle A_0,e_m\rangle_\R+\sum_{n\in\Z_{\ne0}}\big(
\langle A_n,e_m\rangle_\R\,\e^{\imag\pi nt}+\langle B_n,e_m\rangle_\R
\,\e^{-\imag\pi n/t}\big),
\] 
understood as an equality of distributions on the extended real line. In view
of the modest growth of the coefficients established in \eqref{eq:coeff_e_m1},
the uniqueness aspect of Theorem \ref{thm:hfs1} allows us to equate the
coefficients. After all, the left-hand side is a hyperbolic Fourier series too.
This gives rise to all the claims regarding $\langle A_n,e_m\rangle_\R$ 
$\langle B_n,e_m\rangle_\R$. 
\end{proof}

\begin{proof}[Proof of Theorem \ref{thm:PBeurling}]
The assertion is an immediate consequence of the biorthogonality properties 
stated in Corollary \ref{cor:biorthogonal_2.3.5}.
\end{proof}  

\begin{thm}
The $L^1(\R)$-norms of $A_n,B_n$ obey the growth bound
\[
\int_\R |A_n(x)|\,\diff x=\int_\R|B_n(x)|\,\diff x=\Ordo(\log^3|n|)
\]
as $|n|\to+\infty$.
\label{thm:A_n_L1_est}
\end{thm}

\begin{proof}
The equality of norms holds for $n\ne0$ by the change-of-variables formula
in the integral if we take into account the symmetry property of Proposition
\ref{prop:symmetry011}.
Next, by the argument employed in the proof of Corollary 
\ref{cor:biorthogonal_2.3.5}, we have that
the coefficients of the hyperbolic Fourier series expansion of a bounded 
measurable function $g\in L^\infty(\R)$ are given by
\[
a_n(g)=\langle A_n,g\rangle_\R=\int_\R A_n(x)g(x)\,\diff x,\quad
b_n(g)=\langle B_n,g\rangle_\R=\int_\R B_n(x)g(x)\,\diff x.
\]
We now estimate the $L^1(\R)$-norm of $A_n$, since the $L^1(\R)$-norm
of $B_n$ is just the same.
In view of Corollary \ref{cor:8.6.2}, we have the growth bound
\[
a_n(g)=\Ordo(\log^3|n|)\quad\text{as}\,\,\,\,|n|\to+\infty,
\]
and the estimate is uniform in the unit ball of $L^\infty(\R)$, by
Remark \ref{rem:8.6.3}.
So, by choosing $g=g_n$ with $|g_n|\le1$ and $A_ng_n=|A_n|$, 
the assertion is immediate.
\end{proof}

\begin{proof}[Proof of Theorem \ref{thm:interpolation0.01}]
The assertion is an immediate consequence of
Corollary \ref{cor:biorthogonal_2.3.5} and the estimate
of Theorem \ref{thm:A_n_L1_est}.
\end{proof}

\subsection{Summation properties of the biorthogonal system}

We may now identify the periodizations of the functions
$A_n$ and $B_n$.

\begin{thm}
The function $A_0$ has the symmetry $A_0(x)=x^{-2}A_0(-1/x)$
and the summation property
\[
\sum_{j\in\Z}A_0(x+2j)=\frac12,\qquad x\in\R.
\]
Moreover, for each $n\in\Z_{\ne0}$, we have the symmetry
property 
\[
B_n(x)=x^{-2}A_n(-1/x),\qquad x\in\R_{\ne0},
\]
and the summation properties
\[
\sum_{j\in\Z}A_n(x+2j)=\frac12\,{\e^{-\imag\pi nx}},\quad
\sum_{j\in\Z}B_n(x+2j)=0,
\]
hold for each $x\in\R$. 
\label{thm:symmetry+summation}
\end{thm}

\begin{proof}
The symmetry properties were established in 
Proposition \ref{prop:symmetry011}.
By Corollary \ref{cor:hfs1.1}, the functions $A_n,B_n$ are 
$C^\infty$-smooth and satisfy the estimate
\[
|A_0(x)|\le \frac{C_0}{1+x^2},\qquad x\in\R,
\]
and for $n\in\Z_{\ne0}$, the further estimates
\begin{equation}
|A_n(x)|,\,|B_n(x)|\le \frac{C_n}{1+x^2},\qquad x\in\R,
\end{equation}
for a sequence of positive constants $C_n$ with
\[
C_n=\Ordo\big((|n|+1)\log^2(|n|+2)\big)
\]
as $|n|\to+\infty$. Since $n$ is fixed here, we need not worry
about the growth of $C_n$. It is now easy to see that the
two periodization series
\[
\pmb{\Sigma}[A_n](x):=\sum_{j\in\Z}A_n(x+2j),\quad
\pmb{\Sigma}[B_n](x)=\sum_{j\in\Z}B_n(x+2j)
\]
converge absolutely and uniformly on $\R$, and the two sums
$\pmb\Sigma[A_n](x)$ and $\pmb\Sigma[B_n](x)$ are then 
$2$-periodic and continuous. Moreover, we calculate that for 
$m,n\in\Z$, 
\begin{equation}
\int_{[-1,1]}\pmb\Sigma[A_n](x)\,\e^{\imag\pi mx}\,\diff x=
\int_{[-1,1]}\sum_{j\in\Z}A_n(x+2j)\,\e^{\imag\pi m(x+2j)}\,\diff x
=\int_\R A_n(x)\,\e^{\imag\pi m x}\,\diff x=\delta_{m,n}
\end{equation}
where the last delta is in Kronecker's sense, and the right-most
equality relies on Corollary \ref{cor:biorthogonal_2.3.5}. Similarly,
for $n\in\Z_{\ne0}$, we obtain from  
Corollary \ref{cor:biorthogonal_2.3.5} that for $m\in\Z$,
\begin{equation}
\int_{[-1,1]}\pmb\Sigma[B_n](x)\,\e^{\imag\pi mx}\,\diff x=
\int_{[-1,1]}\sum_{j\in\Z}B_n(x+2j)\,\e^{\imag\pi m(x+2j)}\,\diff x
=\int_\R B_n(x)\,\e^{\imag\pi m x}\,\diff x=0.
\end{equation}
It is now an immediate consequence of the completeness of the
Fourier system 
that $\pmb\Sigma[B_n](x)=0$ holds almost everywhere, and then, 
by continuity, we conclude that $\Sigma[B_n](x)=0$ holds everywhere.
Similarly, to identify $\pmb\Sigma[A_n]$, we consider the $2$-periodic
function
\[
\Psi_n(x)=\pmb\Sigma[A_n](x)-\frac{1}{2}\,\e^{-\imag\pi n x},
\]
and see all its Fourier coefficients vanish. Consequently, then,
$\Psi_n$ vanishes almost everywhere, and hence everywhere, by
continuity.
\end{proof}

It is a natural question to ask if it is possible to express the properties
of Theorem \ref{thm:symmetry+summation} in a way that produces more
straightforward equations. One such approach is explored below.

\begin{cor}
For $n\in\Z_{\ne0}$, we put
\[
A_n^+(x):=A_n(x)+B_n(x),\quad A_n^-(x):=A_n(x)-B_n(x).
\]
These functions have the symmetries 
\[
A_n^+(x)=x^{-2}A_n^+(-1/x),\quad A_n^-(x)=-x^{-2}A_n^{-}(x),
\]
and the summation properties
\[
\sum_{j\in\Z}A_n^+(x+2j)=\frac12\,\e^{-\imag\pi nx},\quad \sum_{j\in\Z}
A_n^-(x+2j)=\frac12\,\e^{-\imag\pi nx}.
\]
hold for all $x\in\R$.
\label{cor:A_n^+_A_n^-}
\end{cor}

\begin{proof}
The assertions are immediate consequences of Theorem 
\ref{thm:symmetry+summation}.
\end{proof}

We recall the notation 
\[
\Tope_0 f(x)=\sum_{j\in\Z_{\ne0}}\frac{1}{(2j-x)^2}\,
f\Big(\frac{1}{2j-x}\Big)
\]
from \eqref{eq:Topeomega}.

\begin{cor}
For $n\in\Z_{\ne0}$, the functions $A_n^+$ and $A_n^-$ 
solve the equations
\[
(\Iop+\Tope_0)[A_n^+](x)=\frac12\,\e^{-\imag\pi nx},\quad
(\Iop-\Tope_0)[A_n^-](x)=\frac{1}{2}\,\e^{-\imag\pi nx},
\]
for all $x\in\R\setminus2\Z$, where $\Iop$ denotes the identity 
operator. 
\end{cor}

\begin{proof}
By the symmetry and summation property of $A_n^+$ stated in
Corollary \ref{cor:A_n^+_A_n^-}, 
\begin{multline}
\frac12\,\e^{-\imag\pi nx}=\sum_{j\in\Z}A_n^+(x-2j)
=\sum_{j\in\Z}\frac{1}{(2j-x)^{2}}\,A_n^+\Big(\frac{1}{2j-x}\Big)
\\
=x^{-2}A_n^+(-1/x)+\Tope_0\, A_n^+(x)=A_n^+(x)
+\Tope_0\, A_n^+(x),
\end{multline}
as claimed. This establishes the equation for $A_n^+$, and the
analogous argument applied to $A_n^-$ gives the remaining 
equation.
\end{proof}

\section{Skewed hyperbolic Fourier series}
\label{sec:skewed}

\subsection{Power skewed hyperbolic Fourier series}
\label{subsec:powerskewed2}

It is well-known that the theta function $\vartheta_{00}$
is zero-free in the upper half-plane:
\begin{equation}
\vartheta_{00}(\tau)=\sum_{n\in\Z}\e^{\imag\pi n^2\tau}\ne0,
\qquad \tau\in\Hyp.
\end{equation}
This is a direct consequence of the classical Jacobi triple product
representation (see any book on theta functions).
Moreover, along the imaginary semi-axis $\imag\,\R_{>0}$,
we have that
\[
\vartheta_{00}(\imag\,y)=\sum_{n\in\Z}\e^{-\pi n^2 y}=1+
\sum_{n\in\Z_{\ne0}}\e^{-\pi n^2 y}>1.
\]
From these two properties, it follows that there exists a unique
holomorphic function $L_{00}(\tau)$ such that 
\[
\exp(L_{00}(\tau))=\vartheta_{00}(\tau),\qquad \tau\in\Hyp,
\]
while at the same time $L_{00}(\imag\,y)>0$ for 
$y\in\imag\,\R_{>0}$. This function has limit 
$L_{00}(+\imag\,\infty)=0$, so that $L_{00}(\tau)=
\log\theta_{00}(\tau)$ with the principal branch of the logarithm,
provided that $\im\,\tau$ is big enough. But then $L_{00}$ is $2$-periodic
far up, and consequently it must be $2$-periodic throughout:
\[
L_{00}(\tau+2)=L_{00}(\tau),\qquad\tau\in\Hyp.
\]
The functional equation
\[
\vartheta_{00}(\tau)=(\tau/\imag)^{-1/2}\vartheta_{00}(-1/\tau)
\]
tells us that 
\[
L_{00}(\tau)=-\frac12\log\frac{\tau}{\imag}+L_{00}(-1/\tau),
\qquad\tau\in\Hyp,
\]
for some branch of the logarithm. By inspection of the two sides, we
realize that the logarithm $\log(\tau/\imag)$ must be holomorphic in
$\Hyp$, and that should be real-valued along $\imag\,\R_{>0}$. The
principal branch of the logarithm has these properties, and there is no 
other possibility. This permits us to form powers
\[
(\vartheta_{00}(\tau))^{2\beta}=\exp(2\beta\,L_{00}(\tau))
\]
and the functional relation becomes
\begin{multline}
(\vartheta_{00}(\tau))^{2\beta}=\exp(2\beta\,L_{00}(\tau))=
\exp\Big(2\beta\Big(-\frac12\log\frac{\tau}{\imag}
+L_{00}(-1/\tau)\Big)\Big)
\\
=\exp\Big(-\beta\log\frac{\tau}{\imag}\Big)
\exp\big(2\beta\,L_{00}(-1/\tau)\big)
=(\tau/\imag)^{-\beta}(\vartheta_{00}(-1/\tau))^{2\beta},
\label{eq:functional_theta_beta}
\end{multline}

\begin{proof}[Proof of Proposition \ref{prop:pskewed1}]
The result is an immediate consequence of the
functional relationship \eqref{eq:functional_theta_beta}.
\end{proof}

\subsection{Exponentially skewed hyperbolic Fourier series}
\label{subsec:expskewed3}
Since the modular lambda function has the form 
\begin{equation}
\lambda(\tau)=\frac{\vartheta_{10}(\tau)^4}{\vartheta_{00}(\tau)^4},
\qquad \tau\in\Hyp,
\label{eq:lambdaf1.001}
\end{equation}
and both $\vartheta_{00}$ and $\vartheta_{10}$ lack zeros in 
$\Hyp$, $\lambda(\tau)\ne0$ holds for all $\tau\in\Hyp$. As each of the theta
functions $\vartheta_{10}$ and $\vartheta_{00}$ possesses a Jacobi triple
product representation, the modular lambda function has a similar product
representation as well, and the logarithm of $\lambda(\tau)$ can hence be
expressed as a sum. We now derive some basic properties of this logarithm,
which are most likely known, but we are unable to supply a precise reference. 
Since $\lambda$ has the functional identity
\begin{equation}
\lambda(-1/\tau)=1-\lambda(\tau),
\label{eq:function_lambda_99}
\end{equation}
it follows that $\lambda(\tau)\ne1$ as well for $\tau\in\Hyp$.
Taking logarithms in \eqref{eq:lambdaf1.001}, we obtain that
\begin{equation}
\log\lambda(\tau)=4\log\vartheta_{10}(\tau)-4\log\vartheta_{00}(\tau),
\qquad \tau\in\Hyp,
\label{eq:lambdaf1.002}
\end{equation}
where the right-hand side can be thought of as defining the 
left-hand side. On the right-hand side, the choice of one of the 
logarithms is the function $\log\vartheta_{00}=L_{00}$ discussed in 
Subsection \ref{subsec:powerskewed2}. As for the other logarithm,
we know that
\begin{equation}
\vartheta_{10}(\imag\, y)=\sum_{n\in\Z}\e^{-\pi(n+\frac12)^2y}>
2\,\e^{-\frac14\pi y}>0,
\qquad
y\in\R_{>0},
\end{equation}
so there is a unique holomorphic logarithm function $L_{10}(\tau)$ 
with $\exp(L_{10}(\tau))=\vartheta_{10}(\tau)$, which is in addition
real-valued along $\imag\,\R_{>0}$. Since we may write 
$\vartheta_{10}$ in the form
\[
\vartheta_{10}(\tau)=\sum_{n\in\Z}\e^{\imag\,\pi(n+\frac12)^2\tau}
=2\,\e^{\imag\,\pi\tau/4}\Big(1+\sum_{n=1}^{+\infty}
\e^{\imag\pi n(n+1)\tau}\Big),
\]
we see that 
\[
L_{10}(\tau)=\log2+\frac{\imag\,\pi\tau}{4}+\log\Big(1+
\sum_{n=1}^{+\infty}\e^{\imag\pi n(n+1)\tau}\Big).
\]
For $\im\,\tau\gg1$,  Taylor's formula shows the function 
\[
\tilde L_{10}(\tau):=\log\Big(1+
\sum_{n=1}^{+\infty}\e^{\imag\pi n(n+1)\tau}\Big)
\]
can be expanded in a convergent power series in 
$\e^{\imag2\pi\tau}$, and this property then extends
to all of $\Hyp$. It follows that 
\[
L_{10}(\tau)=\log2+\frac{\imag\,\pi\tau}{4}+\tilde L_{10}(\tau),
\]
so that in particular,  $L_{10}$ has the functional property
\[
L_{10}(\tau+1)=L_{10}(\tau)+\frac{\imag\,\pi}{4}.
\]
Note that $\tilde L_{10}(\tau)$ is holomorphic in $\Hyp$,
and has period $1$, that is, $\tilde L_{10}(\tau+1)=\tilde L_{10}(\tau)$
holds. In addition, $\tilde L_{10}(\imag\,y)>0$ holds for all $y\in\R_{>0}$.
In terms of these functions, we use as logarithm
\begin{equation}
\log\lambda(\tau)=\Lambda(\tau):=4\log2+\imag\,\pi\tau+
4\tilde L_{10}(\tau)-4L_{00}(\tau),
\qquad \tau\in\Hyp.
\label{eq:lambdaf1.003}
\end{equation}
The periodicity property of the logarithm $\Lambda(\tau)$ we read
off from \eqref{eq:lambdaf1.003} is that
\begin{equation}
\Lambda(\tau+2)=\Lambda(\tau)+\imag\,2\pi,
\label{eq:Lambda_per2}
\end{equation}
while the functional identity \eqref{eq:function_lambda_99} takes the form
\[
\exp(\Lambda(-1/\tau))=1-\exp(\Lambda(\tau)).
\]
The fact that $\lambda(\tau)\ne1$ corresponds to having 
\[
\Lambda(\tau)\in\C\setminus\imag\,2\pi\Z,\qquad \tau\in\Hyp.
\]
For $\im\,\tau\gg1$, we can define a holomorphic logarithm of $\lambda(-1/\tau)$ via
\begin{equation}
\log\lambda(-1/\tau)=\log(1-\lambda(\tau))=
-\sum_{k=1}^{+\infty}\frac{1}{k}\,{\lambda(\tau)^k},
\label{eq:lambdaf1.004}
\end{equation}
which takes real values along the imaginary semi-axis $\imag\,\R_{>0}$. 
It must then necessarily coincide with the function $\Lambda$ defined by 
\eqref{eq:lambdaf1.003}, and since the right-hand side of \eqref{eq:lambdaf1.004} 
is $2$-periodic, we find that
\begin{equation}
\Lambda(-1/\tau)=-\sum_{k=1}^{+\infty}\frac{1}{k}\,{\lambda(\tau)^k},\qquad
\im\,\tau\gg1,
\label{eq:Lambda_expansion1}
\end{equation}
and that in addition to \eqref{eq:Lambda_per2}, 
$\Lambda$ enjoys the periodicity property
\begin{equation}
\Lambda\Big(-\frac{1}{\tau+2}\Big)=
\Lambda\Big(-\frac{1}{\tau}\Big),\qquad \tau\in\Hyp.
\end{equation}
For $\omega\in\R$, we use the logarithm to define the power
\[
\lambda(\tau)^\omega=\exp(\omega\Lambda(\tau)),\qquad \tau\in\Hyp,
\]
and then, according to \eqref{eq:Lambda_expansion1},
\begin{equation}
\lambda(-1/\tau)^\omega=\exp(\omega\Lambda(-1/\tau))
=\exp(\omega\log(1-\lambda(\tau)))
=(1-\lambda(\tau))^{\omega},
\qquad \tau\in\Hyp,
\label{eq:binomial999}
\end{equation}
where the right-hand side may be expanded first as a binomial series in 
$\lambda(\tau)$, and hence as a power series in the
nome $q=\e^{\imag\,\pi\tau}$. On the other hand, we also have
\begin{equation}
\lambda(\tau)^\omega=\exp(\omega\Lambda(\tau))
=16^\omega\,\e^{\imag\,\pi\omega\tau}
\exp\big(4\omega(\tilde L_{10}(\tau)-L_{00}(\tau))\big),
\qquad \tau\in\Hyp.
\label{eq:binomial999.1}
\end{equation}

\begin{proof}[Proof of Proposition \ref{prop:expskewed1}]
We see from \eqref{eq:binomial999.1} that
\[
\e^{-\imag\,\pi\omega\tau}\lambda(\tau)^\omega
=16^\omega\,
\exp\big(4\omega(\tilde L_{10}(\tau)-L_{00}(\tau))\big),
\qquad \tau\in\Hyp,
\]
which we may express as convergent Taylor series in the
nome $q=\e^{\imag\,\pi\tau}$. As we apply this formula to 
$\omega=\omega_j$, for $j=1,2$, with $\tau$ possibly replaced
by $-1/\tau$, the correspondence between the holomorphic
hyperbolic Fourier series and the exponentially skewed 
holomorphic hyperbolic Fourier series follows.
\end{proof}


\section{Relation with radial Fourier interpolation}
\label{sec:radial_FI}


\subsection{Heat and Schr\"odinger evolution}
\label{subsec:heatSchr}
The \emph{heat equation} in $\R^d$, where $d\in\Z_{>0}$, reads
\begin{equation}
\partial_s u(x,s)=\Delta_x u(x,s),
\label{eq:heat}
\end{equation}
where $s\in\R$ is the temporal coordinate while $x\in\R^d$ is the spacial
coordinate. The Laplacian is defined in the standard fashion:
\[
\Delta_x:=\partial_{x_1}^2+\cdots+\partial_{x_d}^2,\qquad
x=(x_1,\ldots,x_d)\in\R^d.
\]  
If we endow the heat equation with an initial datum (at time $s=0$)
\begin{equation}
u(x,0)=\varphi(x),\qquad x\in\R^d,
\label{eq:initial}
\end{equation}
the corresponding solution $u(x,s)$ may be written formally as
\begin{equation}
\label{eq:heatevol1}
u(\cdot,s)=\e^{s\Delta}\varphi.
\end{equation}
For $s>0$, it may be expressed in integral form
\begin{equation}
u(x,s)=\int_{\R^d}H_d(x-y,s)\,\varphi(y)\,\mathrm{dvol}_d(y),
\label{eq:heatsoln}
\end{equation}
where $\mathrm{vol}_d$ denotes volume measure in $\R^d$, and
$H_d(x,s)$ denotes the \emph{heat kernel}
\begin{equation}
H_d(x,s):=(4\pi s)^{-d/2}\,\exp\bigg(-\frac{|x|^2}{4s}\bigg).
\label{eq:heatkernel}
\end{equation}
We think of \eqref{eq:heatevol1} and \eqref{eq:heatsoln} as expressing
the same solution, with some caveats. The solution \eqref{eq:heatevol1} is
well-defined for $s=0$, since $\exp(0)$ is interpreted as the
identity operator, whereas the integral \eqref{eq:heatsoln}
is undefined for $s=0$. Moreover, for $s<0$ the operator $\exp(s\Delta)$
would make sense as the densely defined operator inverse of $\exp(-s\Delta)$,
while the integral \eqref{eq:heatsoln} does not make sense for any $s<0$
given the singularity of the kernel $H(x,s)$ for $s<0$.

Some comments are in order here.
For the solution given by \eqref{eq:heatsoln} to be well-defined, the
integral must of course make sense. Moreover, it is a natural question to
ask whether the solution given by \eqref{eq:heatsoln} for $s>0$ is unique.
As a matter of fact, this is not the case in general. If we
consider the homogeneous initial boundary value problem
\begin{equation}
\partial_s v(x,s)=\Delta_x v(x,s),\quad v(x,0)=0
\label{eq:heatinit}
\end{equation}
for time $s\ge0$, the solution is $v=0$ provided that the following rather
weak growth assumption is made:
\begin{equation}
\forall \epsilon>0:\,\,\,v(x,s)=\Ordo\big(\exp(\epsilon|x|^2)\big)
\end{equation}
holds uniformly over $x$ and $0\le s\le s_0$ for each $s_0>0$. For the
necessary details in dimension $d=1$, see, e.g., \cite{cannon-book}, while
\cite{vlad} covers $d>1$ as well. 

We may extend the sense of the heat kernel $H_d(x,s)$ to complex values of
$s$ using the same formula \eqref{eq:heatkernel}. For odd dimensions, this
would involve a choice of the power $(4\pi s)^{-d/2}$, which we do by
using the principal branch of the logarithm, which a branch cut along the
negative real axis. The modulus of the heat kernel is
\begin{equation}
|H_d(x,s)|= (4\pi|s|)^{-d/2}\exp\bigg(-\re\,\frac{|x|^2}{4s}\bigg)
=(4\pi|s|)^{-d/2}\exp\bigg(-\frac{|x|^2\re\, s}{4|s|^2}\bigg),  
\end{equation}
which has a nasty singularity when $\Re s<0$, but has Gaussian decay
for $\re\, s>0$. The imaginary axis is clearly special:
\begin{equation}
|H_d(x,s)|= (4\pi|s|)^{-d/2},
\qquad
\re\, s=0.
\end{equation}
This means that the heat evolution \eqref{eq:heatevol1} makes sense
in terms of \eqref{eq:heatsoln} for $\Re s>0$. 

Given that we like to work with the upper half-plane $\Hyp$ in place of
the right half-plane, we apply the coordinate change $s\tau=\imag s$,
since then $\re s>0$ holds if and only if $\im\,\tau>0$, i.e. if $\tau\in\Hyp$. 
If we agree to write $\tilde u(x,\tau)=u(x,\tau/\imag)$, the heat equation
\eqref{eq:heat} becomes
\begin{equation}
\imag \partial_\tau \tilde u(x,\tau)=\Delta_x \tilde u(x,\tau),
\label{eq:heat2}
\end{equation}
which we recognize as the \emph{Schr\"odinger equation} with trivial potential.
The solution to the Schr\"odinger equation \eqref{eq:heat2} with inital datum
\begin{equation}
\tilde u(x,0)=\varphi(x),\qquad x\in\R^d,
\label{eq:initial2}
\end{equation}
is given by the analogues of \eqref{eq:heatevol1} and \eqref{eq:heatsoln},
\begin{equation}
\label{eq:heatevol2}
\tilde u(\cdot,\tau)=\e^{-\imag\tau \Delta}\varphi
\end{equation}
and, more explicitly,
\begin{equation}
\tilde u(x,\tau)=\int_{\R^d}H_d(x-y,\tau/\imag)\,\varphi(y)\,\mathrm{dvol}_d(y),
\label{eq:heatsoln2}
\end{equation}
for $\tau\in\Hyp$. 

\begin{defn}
The \emph{Schr\"odinger transform} of the function $\varphi$ is given by
\begin{multline}
\mathfrak{S}\varphi(\tau)=\tilde u(0,\tau)=
\big(\e^{-\imag\tau\Delta_x}\varphi\big)(0)
=\int_{\R^d}H_d(y,\tau/\imag)\,\varphi(y)
\,\mathrm{dvol}_d(y)   
\\
= (4\pi\tau/\imag)^{-d/2}\int_{\R^d}
\exp\bigg(-\imag\,\frac{|y|^2}{4\tau}\bigg)
\varphi(y)\,\mathrm{dvol}_d(y),\qquad\tau\in\Hyp.
\end{multline}
\end{defn}

\begin{rem}
(a) A sufficient condition on $\varphi:\R^d\to\C$ for the Schr\"odinger
transform $\mathfrak{S}\varphi$ to be well-defined as a function on $\Hyp$
is that $\varphi$ is Lebesgue measurable, and that
\[
\forall \epsilon>0:\,\,\int_{\R^d}|\varphi(x)|\,
\e^{-\epsilon|x|^2}\mathrm{dvol}_d(x)<+\infty.
\]
(b) Given the shape of the heat kernel \eqref{eq:heatkernel}, the Schr\"odinger 
transform only takes into account the radial part $\varphi_{\mathrm{rad}}$ 
of $\varphi$, that is, $\mathfrak{S}\varphi=\mathfrak{S}\varphi_{\mathrm{rad}}$
holds. Here, $\varphi_{\mathrm{rad}}$ is the unique radial function whose integrals 
along spherical shells coincide with those of $\varphi$.
\label{rem:Schrodinger1}
\end{rem}

For a more restrictive collection of functions $\varphi$ the Schr\"odinger
transform extends continuously to $\R_{\ne0}$, and the formula gives
$\mathfrak{S}\varphi(\tau)$ also for $\tau\in\R_{\ne0}$. This is the case if
$\varphi\in L^1(\R^d)$. Indeed, for $\tau\in\Hyp\cup\R_{\ne0}$, we have the
estimate
\begin{multline}
|\mathfrak{S}\varphi(\tau)|\le\int_{\R^d}|H_d(y,\tau/\imag)|\,|\varphi(y)|
\,\mathrm{dvol}_d(y)   
\\
= (4\pi|\tau|)^{-d/2}\int_{\R^d}\exp\bigg(-\frac{|y|^2\im\tau}{4|\tau|^2}\bigg)
|\varphi(y)|\,\mathrm{dvol}_d(y)
\le(4\pi|\tau|)^{-d/2}\int_{\R^d}|\varphi(y)|\,\mathrm{dvol}_d(y),
\end{multline}
and an approximation argument gives the continuity. Next, if
$\varphi\in L^1(\R^d)$ and $\psi\in L^1(\R,\mu_d)$, where
$\diff\mu_d(t):=|t|^{-d/2}\diff t$, then, by Fubini's theorem,
\begin{multline}
\int_\R\mathfrak{S}\varphi(t)\psi(t)\,\diff t=
\int_\R\int_{\R^d} H_d(y,t/\imag)\psi(t)\,\varphi(y)
\,\mathrm{dvol}_d(y) \,\diff t  
\\
=\int_{\R^d}\int_\R (4\pi t/\imag)^{-d/2}\exp\bigg(-\imag\frac{|y|^2}{4t}\bigg)
\psi(t)\,\diff t\,\varphi(y)\,\mathrm{dvol}_d(y)
=\int_{\R^d}\mathfrak{S}^\ast
\psi(y)\,\varphi(y)\,\mathrm{dvol}_d(y),
\end{multline}
where we write
\begin{multline}
\mathfrak{S}^\ast\psi(y)
:=\int_\R (4\pi t/\imag)^{-d/2}\exp\bigg(-\imag\frac{|y|^2}{4t}\bigg)
\psi(t)\,\diff t
\\
=\int_\R (\imag 4\pi/t)^{-d/2}\exp\bigg(\imag \frac{|y|^2t}{4}\bigg)\,
\Jop_1\psi(t)\,\diff t,\qquad y\in\R^d,
\label{eq:Schrod*}
\end{multline}
which is a radial expression.
The involutive operator $\Jop_1$ is as before given by equation
\eqref{eq:Jop1}. The expression \eqref{eq:Schrod*} is of Fourier type,
essentially the Fourier transform of $(t/\imag)^{d/2}\Jop_1\psi(t)$.  

\subsection{The Schr\"odinger transform of complex exponential waves
and point masses}

If $\varphi=\delta_\xi$, the unit point mass at a point $\xi\in\R^d$, we see
from the definition that the associated Schr\"odinger transform is
\begin{equation}
\mathfrak{S}\delta_\xi(\tau)=H_d(\xi,\tau/\imag)=(4\pi\tau/\imag)^{-d/2}
\exp\Big(-\imag\,\frac{|\xi|^2}{4\tau}\Big),\qquad\tau\in\Hyp,
\label{eq:Schrodingertrans-point}
\end{equation}
which makes sense also if $\tau\in\R_{\ne0}$. We shall also need
the Schr\"odinger transform of a complex exponential wave
\[
e_\eta(x):=\e^{\imag \langle x,\eta\rangle}, 
\]
where we write
\[
\langle x,\eta\rangle=x_1\eta_1+\ldots+x_d\eta_d,\qquad x,\eta\in\R^d,
\]
and use the notational convention $x=(x_1,\ldots,x_d)$, 
$\eta=(\eta_1,\ldots,\eta_d)$. The function 
\[
\tilde u(x,\tau)=\exp\big(\imag\langle x,\eta\rangle+\imag|\eta|^2\tau\big)
\]
is bounded in $\tau\in\Hyp\cup\R$, and has initial datum
$\tilde u(x,0)=e_\eta(x)$ for $x\in\R^d$. It moreover solves the 
Schr\"odinger equation \eqref{eq:heat2}, so that by the uniqueness theorem
for the heat equation, this must be the same as the function given by
\eqref{eq:heatsoln2}. Consequently, the Schr\"odinger transform of the
wave $e_\eta$ is given by
\begin{equation}
\mathfrak{S}e_\eta(\tau)=\tilde u(0,\tau)=\exp(\imag|\eta|^2\tau),\qquad
\tau\in\Hyp\cup\R.
\label{eq:Schrodingertrans-complexexp}
\end{equation}
In view of \eqref{eq:Schrodingertrans-point} and 
\eqref{eq:Schrodingertrans-complexexp}, we recognize the basis functions
of a power skewed hyperbolic Fourier series, provided that 
\[
\frac{|\eta|^2}{\pi}\in\Z_{\ge0}, \quad \frac{|\xi|^2}{4\pi}\in\Z_{>0}.
\]

\subsection{The Fourier transform on radial functions}
We write 
\[
\hat f_d(y):=\int_{\R^d}\e^{-\imag2\pi\langle x,y\rangle}f(x)\,
\mathrm{dvol}_d(x),\qquad y\in\R^d,
\]
for the $d$-dimensional Fourier transform.
When $f$ is radial, we may think of the function as defined on $\R_{\ge0}$
via the identification $f(|x|)=f(x)$ for $x\in\R$, and then its Fourier transform
$\hat f_d$ is radial  too, and given by the formula
\[
\hat f_d(s)=2\pi\,s^{1-\frac{d}{2}}\int_{0}^{+\infty}J_{\frac{d}{2}-1}(2\pi st)
\,f(t)\,t^{d/2}\diff t, \qquad s\in\R_{\ge0},
\]
extended in the same fashion to $\R^d$ via $\hat f_d(y)=\hat f_d(|y|)$. Here,
$J_\nu$ is the familiar Bessel function, 
\[
J_\nu(x)=\sum_{n=0}^{+\infty}\frac{(-1)^n}{n!\Gamma(n+\nu+1)}
\Big(\frac{x}{2}\Big)^{2n+\nu}.
\]
This means that the radialization in $x\in\R^d$ of the complex exponential 
wave $e_{\eta}$ is given by 
\begin{equation}
(e_{\eta})_{\mathrm{rad}}(x)=2^{\frac{d}{2}-1}{\Gamma(d/2)}\,
|x|^{1-\frac{d}{2}}|\eta|^{1-\frac{d}{2}}\,
J_{\frac{d}{2}-1}(|x|\,|\eta|).
\end{equation}

\subsection{Skewed hyperbolic Fourier series expansion of a Schr\"odinger
transform}
We begin with a measurable radial function $f:\,\R^d\to\C$, such that its
Schr\"odinger transform 
\[
F(\tau):=\mathfrak{S}f(\tau),\qquad\tau\in\Hyp,
\] 
is a well-defined holomorphic function (see Remark \ref{rem:Schrodinger1}). 
As such, it may be re\-present\-ed by a power skewed holomorphic hyperbolic 
Fourier series
\begin{equation}
F(\tau)=a_0+\sum_{n=1}^{+\infty}\big(a_n\,\e^{\imag\pi n\tau}+
b_n\,(\tau/\imag)^{-d/2}\e^{-\imag\pi n/\tau}\big),\qquad\tau\in\Hyp.
\label{eq:4.4.1}
\end{equation}
This fact can be seen by combining the holomorphic version of
Theorem \ref{thm:HFS-general.1.1} with Proposition 
\ref{prop:pskewed1} (which connects power skewed hyperbolic Fourier
series with ordinary hyperbolic Fourier series).
Depending on the growth control on $F(\tau)$, it may be that we do not
have uniqueness in the representation \eqref{eq:4.4.1} (this issue is 
basically governed by Theorem \ref{thm:uniq2.1}).  If we have 
non-uniqueness even for slowly growing $F$, that is, for $F$ with the
estimate 
\begin{equation}
\label{eq:F_growthbound1}
|F(\tau)|=\Ordo\Big(\frac{1+|\tau|^2}{\im\,\tau}\Big)^k,
\end{equation}
for some finite $k$, the solution is to throw away some terms in the 
series until we get unique coefficients. Instead of just throwing these 
terms away we can say that we declare some of the coefficients to 
vanish for each such slowly growing $F$ (alternatively,
some finite linear combinations of coefficients would vanish). 
With such a choice of coefficients,
we can get a unique representation of the form \eqref{eq:4.4.1}, 
where $a_n=a_n(F)$ and $b_n=b_n(F)$ are linear functionals.
Suppose we are able to control the growth of the coefficients like
\[
a_n(F),b_n(F)=\Ordo(|n|+1)^{k'}
\]
for some $k'=k'(k)$ for each $F$ with the growth bound 
\eqref{eq:F_growthbound1}, and that this holds for all finite $k$, 
then it can be shown that there exist 
test functions $A_{n,d}$ and $B_{n,d}$ in the class 
$C^\infty_1(\R\cup\{\infty\})$, such that
\[
a_n(F)=\langle A_{n,d},F\rangle_\R,\quad
b_n(F)=\langle B_{n,d},F\rangle_\R,
\]
where $F$ is understood as a distribution on the extended real line
$\R\cup\{\infty\}$. Next, we pick points 
$\xi^{\langle n\rangle},\eta^{\langle n\rangle}\in \R^d$ such that
$|\xi^{\langle n\rangle}|^2=4\pi n$ and 
$|\eta^{\langle n\rangle}|^2=\pi n$
and write the representation \eqref{eq:4.4.1}
in the form
\begin{multline}
\mathfrak{S}f(\tau)=F(\tau)=\langle A_{0,d},F\rangle_\R
+\sum_{n=1}^{+\infty}\big(\langle A_{n,d},F\rangle_\R\,
\mathfrak{S}e_{\eta_n}(\tau)
+
\langle B_{n,d},F\rangle_\R\,(4\pi)^{d/2} 
\mathfrak{S}\delta_{\xi_n}(\tau)\big),
\qquad\tau\in\Hyp.
\label{eq:4.4.2}
\end{multline}
We can remove the Schr\"odinger operator on both sides provided we
have projected to the radial functions:
\begin{equation}
f(x)=\langle A_{0,d},F\rangle_\R
+\sum_{n=1}^{+\infty}\big(\langle A_{n,d},F\rangle_\R\,
(e_{\eta_n})_{\mathrm{rad}}(x)
+
\langle B_{n,d},F\rangle_\R\,(4\pi)^{d/2} 
(\delta_{\xi_n})_{\mathrm{rad}}(x)\big),
\label{eq:4.4.3}
\end{equation}
where the identity is understood in the sense of distribution theory on 
$\R^d$. If $g$ is a $C^\infty$-smooth radial function with compact support, 
this would mean that
\begin{multline}
\int_{\R^d}f(x)\,g(x)\,\mathrm{dvol}_d(x)=\langle A_{0,d},F\rangle_\R
\\
+\sum_{n=1}^{+\infty}\big(\langle A_{n,d},F\rangle_\R\,
\int_{\R^d}e_{\eta_n}(x)\,g(x)\,\mathrm{dvol}_d(x)
+
\langle B_{n,d},F\rangle_\R\,(4\pi)^{d/2} 
g(\xi_n)\big)
\\
=\langle A_{0,d},\mathfrak{S}f\rangle_\R
+\sum_{n=1}^{+\infty}\big(\langle A_{n,d},\mathfrak{S}f\rangle_\R\,
\hat g_d(-\eta_n/{2\pi})
+\langle B_{n,d},\mathfrak{S}f\rangle_\R\,(4\pi)^{d/2} 
g(\xi_n)\big)
\\
=\langle \mathfrak{S}^\ast A_{0,d},f\rangle_{\R^d}
+\sum_{n=1}^{+\infty}\big(\langle \mathfrak{S}^\ast A_{n,d},f\rangle_{\R^d}\,
\hat g_d(\tfrac12\sqrt{n/\pi})
+\langle \mathfrak{S}^\ast B_{n,d},f\rangle_{\R^d}\,(4\pi)^{d/2} 
g(2\sqrt{\pi n})\big).
\label{eq:4.4.4}
\end{multline}
Finally, we change our perspective and think of $f$ as a radial test 
function, and realize that since the 
functions $\mathfrak{S}^\ast A_{n,d}$ and $\mathfrak{S}^\ast B_{n,d}$ 
are all radial, we must have the Fourier interpolation formula
\begin{equation}
g(x)
=\mathfrak{S}^\ast A_{0,d}
+\sum_{n=1}^{+\infty}\big(\hat g_d(\tfrac12\sqrt{n/\pi})\,
\mathfrak{S}^\ast A_{n,d}(x)
+(4\pi)^{d/2} 
g(2\sqrt{\pi n})
\mathfrak{S}^\ast B_{n,d}(x)\big),
\label{eq:4.4.5}
\end{equation}
for $x\in\R^d$, understood in the sense of distribution theory.
Here, once we obtain sufficient quantitative control of the terms of the
interpolation formula, we can thn extend it to wider classes of test functions
$g$.

\begin{rem}
We should mention that Martin Stoller considers more general Fourier
interpolation formul\ae{} for non-radial Schwartzian test functions
in \cite{Stoller-thesis}, \cite{Stoller-TAMS}.
\end{rem}

\subsection*{Acknowledgements}
The work of Hedenmalm is supported
Vetenskapsr\aa{}det grant 2020-03733, and by grant 075-15-2021-602 of the
government of the Russian Federation for the state support of scientific
research carried out under the supervision of leading scientists.
The work of Montes-Rodr\'\i{}guez is supported partially by Plan Nacional
I+D+I ref. PID2021-127842NB-I00, and Junta de Andaluc\'\i{}a FQM-260.
Furthermore, we thank Maryna Viazovska for several valuable conversations
and her interest in this work. In addition, we thank Danylo Radchenko 
for sharing his insights into sphere packing problems and associated 
Riemann-Hilbert problems. We thank Bakan, Radchenko, and Viazovska for taking
part in the earlier preprint version of the present work \cite{BHMRV1}.


\begin{thebibliography}{99}


\bibitem{Ahlfors} Ahlfors, L. V., \emph{Complex analysis}.
An introduction to the theory of analytic functions of one complex variable.
Third edition. International Series in Pure and Applied Mathematics.
McGraw-Hill Book Co., New York, 1978. 



\bibitem{bh2} Bakan, A., Hedenmalm, H., \emph{Exponential integral
representations of theta functions}.
Comput. Methods Funct. Theory {\bf{20}} (2020), 591--621.

\bibitem{bhm} Bakan, A., Hedenmalm, H., Montes-Rodriguez, A.,
Radchenko, D. and Viazovska, M., 
\emph{Fourier uniqueness in even dimensions}.
Proc. Natl. Acad. Sci. USA  {\bf{118}},  (2021), no. 15, e2023227118.

\bibitem{BHMRV1} Bakan, A., Hedenmalm, H., Montes-Rodr\'\i{}guez,
A., Radchenko, D., Viazovska, M., \emph{Hyperbolic Fourier series}.
arXiv:2110.00148






\bibitem{bon}
Bondarenko, A., Radchenko, D., Seip K., \emph{Fourier 
interpolation with zeros of zeta and $L$-functions}. 
Constr. Approx. \textbf{57} (2023), 405-461.

\bibitem{BorHed} Borichev, A., Hedenmalm, H., 
\emph{Completeness of translates in weighted spaces on the half-line}.
Acta Math. \textbf{174} (1995), 1-84.

\bibitem{cannon-book}
Cannon, J. R., 
\emph{The one-dimensional heat equation}.
With a foreword by Felix E. Browder.
Encyclopedia of Mathematics and its Applications, \textbf{23}.
Addison-Wesley Publishing Company, Advanced Book 
Program, Reading, MA, 1984.

\bibitem{CHM}
Canto-Mart\'\i{}n, F., Hedenmalm, H., Montes-Rodr\'\i{}guez, A.,
\emph{Perron-Frobenius operators and the Klein-Gordon equation}.
J. Eur. Math. Soc. (JEMS) \textbf{16} (2014), 31-66. 

\bibitem{cha}  Chandrasekharan, K., \emph{Elliptic functions}. 
Grundlehren der Mathematischen
Wissenschaften [Fundamental Principles of Mathematical Sciences], 
\textbf{281}. Springer-Verlag, Berlin, 1985.

\bibitem{cohn} Cohn, H., \emph{From sphere packing to Fourier 
interpolation}.
Bull. Amer. Math. Soc. \textbf{61} (2024), 3-22.

\bibitem{ckrv} Cohn, H., Kumar, A., Miller, S. D., Radchenko, D.,
Viazovska, M.,
\emph{The sphere packing problem in dimension 24}. Ann. of Math. (2) 
\textbf{185} (2017), 1017-1033.






\bibitem{ger}  Gera, B., \emph{The problem of recovering the
temperature and moisture fields in a porous 
body from incomplete initial data}. J. Math. Sci.
\textbf{86} (1997), no. 2,  2578-2684.

\bibitem{FinkSchein}
Finkelstein, M., Scheinberg, S.,
\emph{Kernels for solving problems of Dirichlet type in a half-plane}.
Adv. Math. \textbf{18} (1975), 108-113. 


\bibitem{Hed-Jacobi} Hedenmalm, H., \emph{Jacobi theta functions and
hyperbolic Fourier series}. In preparation. 

\bibitem{hed}  Hedenmalm, H.,  Montes-Rodriguez, A.,
\emph{Heisenberg uniqueness pairs and the Klein-Gordon equation}. 
Ann. of Math. (2) \textbf{173} (2011), no. 3, 1507--1527.

\bibitem{hed1}
Hedenmalm, H.,  Montes-Rodr\'\i{}guez, A., 
\emph{The Klein-Gordon equation, the Hilbert transform and
dynamics of Gauss-type maps}. 
J. Eur. Math. Soc. (JEMS) \textbf{22} (2020), 1703-1757.

\bibitem{hed2}
 Hedenmalm, H.,  Montes-Rodr\'\i{}guez, A., 
 \emph{The Klein-Gordon equation, the Hilbert transform, and 
 Gauss-type maps: $H^{\infty}$ approximation}. J. Anal. Math. 
 \textbf{144} (2021), 119-190.

\bibitem{HedWenn}
Hedenmalm, H., Wennman, A.,
\emph{Polynomial growth bounds for solutions to the Dirichlet 
problem in the half-plane}, in preparation.
 




\bibitem{Kor-Acta1}
Korenblum, B., \emph{An extension of the Nevanlinna theory}.
Acta Math. \textbf{135} (1975), 187-219. 






\bibitem{lio} Lions, J.-L.,
\emph{Sur les sentinelles des syst\`{e}mes distribu\'{e}s.
Conditions fronti\`{e}res, termes sources, coefficients
incompl\`{e}tement connus} [On the sentinels of distributed systems.
Situations where
boundary conditions, or source terms, or coefficients are incompletely known].
C. R. Acad. Sci., Paris, S\'{e}r. I Math. \textbf{307} (1988), no. 17, 865-870.




\bibitem{rad}  Radchenko, D.,  Viazovska, M.,
\emph{Fourier interpolation on the real line}.
Publ. Math. Inst. Hautes \'Etudes Sci., \textbf{129} (2019), 51-81.

\bibitem{rud} Rudin, W., \emph{Real and complex analysis}. 2nd ed.
McGraw-Hill Book Co., New York, 1974.





\bibitem{Stoller-thesis}
Stoller, M. P., \emph{Fourier uniqueness and interpolation in Euclidean
space}. PhD thesis, EPFL 2022.
  
\bibitem{Stoller-TAMS}
Stoller, M., 
\emph{Fourier interpolation from spheres}. 
Trans. Amer. Math. Soc. \textbf{374} (2021), no. 11, 8045-8079. 
  
\bibitem{Straube}
Straube, E. J.,
\emph{Harmonic and analytic functions admitting a distribution boundary
value}.
Ann. Scuola Norm. Sup. Pisa Cl. Sci. (4) \textbf{11} (1984), no. 4, 559-591. 

\bibitem{tsuji}
Tsuji, M.,
\emph{Potential theory in modern function theory}.
Reprinting of the 1959 original. Chelsea Publishing Co., New York, 1975.

\bibitem{Viaz}
Viazovska, M. S.,
\emph{The sphere packing problem in dimension 8}. 
Ann. of Math. (2) \textbf{185} (2017), no. 3, 991-1015.

\bibitem{vlad}
Vladimirov, V. S.,  \emph{Methods of the theory of generalized functions}.
Taylor \& Francis, London, 2002.


\end{thebibliography}
\end{document}